\documentclass[12pt,a4paper]{article}
\usepackage{amsmath}
\usepackage{amssymb}
\usepackage{amsthm}
\usepackage{amsfonts}
\usepackage{latexsym}

\theoremstyle{plain}
\newtheorem{thm}{Theorem}[section]
\newtheorem{cor}{Corollary}[section]
\newtheorem{lem}{Lemma}[section]
\newtheorem{prop}{Proposition}[section]

\theoremstyle{remark}

\theoremstyle{definition}
\newtheorem{defn}{Definition}[section]

\newtheorem{exmp}{Example}[section]

\newtheorem{rem}{Remark}[section]

\title{Conic reductions for Hamiltonian actions of $U(2)$ and its maximal torus}
\author{Roberto Paoletti\footnote{\noindent{\bf Address:}
Dipartimento di Matematica e Applicazioni, Universit\`a degli Studi
di Milano-Bicocca, Via Roberto Cozzi 55, 20126 Milano,
Italy; {\bf e-mail}: roberto.paoletti@unimib.it }}
\date{}

\begin{document}
\maketitle

\begin{abstract}
Suppose given a Hamiltonian and holomorphic action of 
$G=U(2)$ on a compact K\"{a}hler manifold 
$M$, with nowhere vanishing moment map. Given an integral coadjoint orbit $\mathcal{O}$
for $G$, under transversality assumptions 
we shall consider two naturally associated \lq conic\rq\, reductions. One, which will be denoted 
$\overline{M}^G_{\mathcal{O}}$, is taken with respect to the action of $G$
and the cone over $\mathcal{O}$; another, which will be denoted 
$\overline{M}^T_{\boldsymbol{\nu}}$, is taken with respect to the action of the standard maximal 
torus $T\leqslant G$ and the ray $\mathbb{R}_+\,\imath\boldsymbol{\nu}$ 
along which the cone over 
$\mathcal{O}$
intersects the positive Weyl chamber. 
These two reductions share a common \lq divisor\rq,\, which may be viewed heuristically 
as bridging between their structures. This point of view 
motivates studying the (rather different) ways
in which the two reductions relate to the the latter divisor. In this paper we provide some results in this
directions. Furthermore, we give explicit transversality criteria for a large class of such actions
in the projective setting,
as well as a description of corresponding reductions as weighted projective
varieties, depending on combinatorial data associated to the action and the orbit. 
\end{abstract}

\section{Introduction}

Let $M$ be a $d$-dimensional compact and connected K\"{a}hler manifold,
with complex structure $J$, and K\"{a}hler form $\omega$.
For instance, $M$ might be complex projective space $\mathbb{P}^d$,
and $\omega$ the Fubini-Study form.

Let us assume, in addition, that $G=U(2)$ 
and $\phi:G\times M\rightarrow M$ is a holomorphic
and Hamiltonian action, with moment map
$\Phi:M\rightarrow \mathfrak{g}^\vee$, 
where $\mathfrak{g}=\mathfrak{u}(2)$ is the Lie algebra of $G$ 
(we refer to \cite{gs st} for generalities on Hamiltonian actions and moment maps). For example, $M$ might be $\mathbb{P}W$, where $W$ is a 
complex unitary representation space for $G$, with the naturally associated
$G$-action.
We shall equivariantly identify $\mathfrak{g}\cong \mathfrak{g}^\vee$ 
by the inner product $\langle\beta_1,\beta_2\rangle :=
\mathrm{trace}\big(\beta_1\,\overline{\beta}_2^t\big)$; hence one can
equivalently view $\Phi$ as being a $\mathfrak{g}$-valued equivariant map.

An important and ubiquitous geometric construction associated to 
Hamiltonian actions is the symplectic reduction with respect to an
invariant submanifold $\mathcal{R}\subset
\mathfrak{g}^\vee$, assuming that $\Phi$ is transverse to $\mathcal{R}$;  
the geometry of the action
may lead to different 
choices of $\mathcal{R}$ (\cite{gs-gq}, \cite{gs-hq}).

Here we shall assume that $\mathbf{0}\not\in \Phi (M)$. 
In this situation, a natural choice for $\mathcal{R}$, suggested 
by geometric quantization, is the cone 
$\mathcal{C}(\mathcal{O})
=\mathbb{R}_+\,\mathcal{O}\subset \mathfrak{g}^\vee$
over an integral coadjoint orbit $\mathcal{O}$ \cite{gs-hq}. 

\begin{exmp}
\label{exmp:PWLK}
To fix ideas on a specific case, consider the Hamiltonian $G$-space
$\mathbb{P}(W_{\mathbf{L},\mathbf{K}})$
associated to a unitary representation 
\begin{equation}
\label{eqn:WLK}
W_{\mathbf{L},\mathbf{K}}:=
\bigoplus_{a=1}^r
{\det}^{\otimes l_a}\otimes \mathrm{Sym}^{k_a}\left(\mathbb{C}^2\right),
\end{equation}
where $\mathbf{L}=(l_a)\in \mathbb{Z}^r$, $\mathbf{K}=(k_a)\in \mathbb{N}^r$.
Then $\mathbf{0}\not\in \Phi (M)$ if and only if either
$k_a+2\,l_a>0$ for all $a=1,\ldots,r$, or $k_a+2\,l_a<0$ for all 
$a=1,\ldots,r$ (see Proposition \ref{prop:equivalent non zero phi}).

More explicitly (to be precise, 
with an extra genericity assumption on
$W_{\mathbf{L},\mathbf{K}}$ - 
see Definition \ref{defn:restricted class}) 
the image of $\Phi$ is the convex hull of the subsets
$\imath\,L_{k_a}+\imath\,l_a\,I_2\subset \mathfrak{g}$,
where $L_{k_a}$ is the set of positive semidefinite Hermitian matrices of
trace $k_a$, for $a=1,\ldots,r$
(see (\ref{eqn:immaginePhi'}) and Proposition \ref{prop:Phi inviluppo convesso}).
Furthermore, if 
$\boldsymbol{\nu}=\begin{pmatrix}
\nu_1&\nu_2
\end{pmatrix}\in \mathbb{R}^2$ and $D_{\boldsymbol{\nu}}$ is the diagonal
matrix with entries $\nu_1,\,\nu_2$, then 
$\imath\,D_{\boldsymbol{\nu}}$ belongs to the image of $\Phi$ if and
only if $\boldsymbol{\nu}$ belongs to the convex hull of the all the vectors
$\begin{pmatrix}
k_a+l_a&l_a
\end{pmatrix}$ and $\begin{pmatrix}
l_a&k_a+l_a
\end{pmatrix}$, for $a=1,\ldots,r$
(Corollaries \ref{cor:PHIKL convex hull} and \ref{cor:intersetion=projection general}).
In addition, if $\nu_1\neq \nu_2$ then
$\Phi$ is transverse to the cone over the orbit
$\mathcal{O}_{\boldsymbol{\nu}}$ 
of $\imath\,D_{\boldsymbol{\nu}}$
if and only if $\boldsymbol{\nu}$ 
does not belong to the one of rays sprayed by the vectors
$\begin{pmatrix}
k_a-j+l_a &j+l_a
\end{pmatrix}$, for $a=1,\ldots,r$ and $j=0,\ldots,k_a$ (Theorem \ref{thm:transversality general}).

\end{exmp}

Assume that $\mathbf{0}\not\in \Phi(M)$, that $\mathcal{O}$ is an
integral orbit, and that 
$\Phi$ is transverse to 
$\mathcal{C}(\mathcal{O})$; then the (coisotropic, real) hypersurface
$M^G_{\mathcal{O}}:=\Phi^{-1}\big(\mathcal{C}(\mathcal{O})\big)\subset M$ is compact and 
connected (Theorem 1.2 of \cite{gp}). 
Let $\sim$ be the equivalence relation given by the null foliation.
The symplectic reduction of $M$ with
respect to $\mathcal{C}(\mathcal{O})$ is 
$\overline{M}^G_{\mathcal{O}}:=M^G_{\mathcal{O}}/\sim$,
together with its naturally induced reduced orbifold symplectic structure 
$\omega_{•\overline{M}^G_{\mathcal{O}}}$. 
We shall refer to $(\overline{M}^G_{\mathcal{O}},\,\omega_{•\overline{M}^G_{\mathcal{O}}})$ 
as \textit{the conic reduction of $M$ with respect to $G$ and $\mathcal{O}$}.

There are other reductions associated to the integral orbit $\mathcal{O}$
built into this picture. 
Let $T\leqslant G$ be the maximal torus of diagonal unitary
matrices, and $\psi:T\times M\rightarrow M$ the restricted action.
Then $\psi$ is also Hamiltonian; let $\Psi:M\rightarrow \mathfrak{t}
\cong \mathfrak{t}^\vee$ be its moment map. We shall identify
$\mathfrak{t}$ with $\imath\,\mathbb{R}^2$.

Assume that $\mathbf{0}\not\in \Psi (M)$ (this is in principle
a stronger hypothesis than  $\mathbf{0}\not\in 
\Phi(M)$), and that $\Psi$ is transverse to a ray $\mathbb{R}_+\cdot \imath 
\,\boldsymbol{\nu}$, 
where $\boldsymbol{\nu}=\begin{pmatrix}
\nu_1&\nu_2
\end{pmatrix}\in \mathbb{Z}^2\setminus \{\mathbf{0}\}$.
Let us set $\boldsymbol{\nu}_\perp:=
\begin{pmatrix}
-\nu_2&\nu_1
\end{pmatrix}\in \mathbb{Z}^2$. Let
$T^1_{\boldsymbol{\nu}_\perp}\leqslant T$ be 
the subgroup generated by 
$\imath\,\boldsymbol{\nu}_\perp$.
If non-empty, $M^T_{•\boldsymbol{\nu}}:=\Psi^{-1}(\mathbb{R}_+\cdot \imath 
\,\boldsymbol{\nu})$ is then a connected compact hypersurface in $M$, 
whose null foliation $\sim'$ is given by the orbits of 
$T^1_{\boldsymbol{\nu}_\perp}$.

The quotient 
$\overline{M}^T_{•\boldsymbol{\nu}}=M^T_{•\boldsymbol{\nu}}/\sim'$ is then  
also an orbifold, with a reduced K\"{a}hler structure 
$(\overline{M}^T_{•\boldsymbol{\nu}}, J_0,\Omega_0)$, which can be viewed
as the symplectic quotient (symplectic reduction at $0$) for the Hamiltonian
action of $T^1_{\boldsymbol{\nu}_\perp}$ on $M$.
We shall refer to $(\overline{M}^T_{•\boldsymbol{\nu}}, J_0,\Omega_0)$
as the \textit{conic reduction of $M$ with respect to $T$ and $\boldsymbol{\nu}$}.

The two hypersurfaces $M^G_{\mathcal{O}}$ and 
$M^T_{•\boldsymbol{\nu}}$ meet tangentially along the smooth connected locus
$M^G_{\boldsymbol{\nu}}:=\Phi^{-1}(\mathbb{R}_+\cdot \imath 
\,\boldsymbol{\nu})$ (Theorem 1.2 of \cite{gp} - in \textit{loc. cit.}
$M$ was assumed to be projective, but Theorem 1.2 holds true in the K\"{a}hler setting). 
Furthermore, 
the null foliations of $M^G_{\mathcal{O}}$ and 
$M^T_{•\boldsymbol{\nu}}$ are tangent to $M^G_{\boldsymbol{\nu}}$
since the latter is $T$-invariant, and they actually coincide along it. 
Therefore, the quotient 
$\overline{M}^G_{\boldsymbol{\nu}}:=M^G_{\boldsymbol{\nu}}/\sim$ is an orbifold.
$\overline{M}^G_{\boldsymbol{\nu}}$
has an intrinsic symplectic structure
$\omega_{•\overline{M}^G_{\boldsymbol{\nu}}}$, and in fact
$\big(\overline{M}^G_{\boldsymbol{\nu}},\omega_{•\overline{M}^G_{\boldsymbol{\nu}}}\big)$ 
can be interpreted as a symplectic quotient of a symplectic cross section for the
$G$-action, in the sense of \cite{gs cp1}. Furthermore, 
$\big(\overline{M}^G_{\boldsymbol{\nu}},\omega_{•\overline{M}^G_{\boldsymbol{\nu}}}\big)$
embeds symplectically in both
$(\overline{M}^T_{•\boldsymbol{\nu}}, \Omega_0)$ 
and $(\overline{M}^G_{\mathcal{O}},\,\omega_{•\overline{M}^G_{\mathcal{O}}})$.
Hence, $\overline{M}^G_{\boldsymbol{\nu}}$ can be viewed
as bridging between $\overline{M}^G_{\mathcal{O}}$
and $\overline{M}^T_{•\boldsymbol{\nu}}$.
This heuristic point of view motivates investigating 
$\overline{M}^G_{\mathcal{O}}$
and $\overline{M}^T_{•\boldsymbol{\nu}}$ in relation to 
$\overline{M}^G_{\boldsymbol{\nu}}$. 

Regarding $\overline{M}^G_{\mathcal{O}}$, we shall prove that
in a large class of cases the 
symplectic orbifold 
$(\overline{M}^G_{\mathcal{O}},\,\omega_{•\overline{M}^G_{\mathcal{O}}})$ 
factors as the product of
$(\overline{M}^G_{\boldsymbol{\nu}},\omega_{•\overline{M}^G_{\boldsymbol{\nu}}})$
and $\mathbb{P}^1$, endowed with a suitable rescaling of the Fubini-Study form 
(Theorem \ref{thm:Delta simplettomorfismo}). In the more general situation, 
$\overline{M}^G_{\mathcal{O}}$ is still, in some sense, topologically close to 
being a product (Theorem \ref{thm:MGO general case}).

Regarding $\overline{M}^T_{\boldsymbol{\nu}}$, we shall see that 
$\overline{M}^G_{\boldsymbol{\nu}}$ embeds in it as the zero locus of a transverse 
section of an orbifold line bundle $L$; this section is naturally associated to 
the moment map (Theorem \ref{thm:orbifold MTnu}). 
The curvature of $L$ is the form $\Omega_0'$ introduced in 
\cite{dh} to study the variation of the cohomology class of a symplectic
reduction, namely, the curvature to the orbifold $S^1$-bundle 
$M^T_{\boldsymbol{\nu}}\rightarrow \overline{M}^T_{\boldsymbol{\nu}}$
(striclty speaking, $\Omega_0'$ is not uniquely defined as a form, but in our context there will be a natural choice).
If $\Omega_0'$ is symplectic and there exists an
orbifold complex structure on $\overline{M}^T_{\boldsymbol{\nu}}$ 
compatible with $\Omega_0'$, we shall call the triple 
$\big(\overline{M}^T_{\boldsymbol{\nu}},J'_0,\Omega_0'\big)$ 
the \textit{$\boldsymbol{\nu}$-th DH-conic reduction of $M$}.

We shall see that this is the case for the spaces
$\mathbb{P}(W_{\mathbf{L},\mathbf{K}})$ in Example \ref{exmp:PWLK}.
More precisely, we shall classify the corresponding DH-reductions
and explicitly describe them as K\"{a}hler weighted projective varieties 
parametrized by certain combinatoric data depending on
$\boldsymbol{\nu},\,\mathbf{L},\,\mathbf{K}$.
In these cases $L$ is an ample orbifold line bundle on $M$
(Theorem \ref{thm:general MTnu}). Furthermore, for a class of representations
that we call \textit{uniform} (Definition \ref{defn:uniform rep})
the complex orbifold
$\big(\overline{M}^T_{\boldsymbol{\nu}},J_0'\big)$ remains constant
as $\boldsymbol{\nu}$ ranges within one of the fundamental wedges cut out
by the \lq critical rays\rq\, (see Example \ref{exmp:PWLK}).

Finally, we shall focus on the specific case of the irreducible representations 
$\mathrm{Sym}^k(\mathbb{C}^2)$. We shall see that if $\nu_1>(k-1)\,\nu_2>0$
then $\overline{M}^T_{\boldsymbol{\nu}}$ is the weighted projective space $\mathbb{P}(1,2,\ldots,k)$, and that if $\nu_1\gg \nu_2>0$
(the bounds might be made effective and depend on $k$) then
$\overline{M}^T_{\boldsymbol{\nu}}$ is smoothly and symplectically 
isotopic to $\mathbb{P}(2,\ldots,k)\subset \mathbb{P}(1,2,\ldots,k)$
(Theorem \ref{thm:isotopy}).

\section{Transversality criteria}
\label{sctn:transversality criteria}
In this section we shall provide some general transversality criteria
involving the moment map $\Phi:M\rightarrow \mathfrak{g}^\vee$ and 
a cone $\mathcal{C}(\mathcal{O})$ over a coadjoint orbit
in the case of Hamiltonian $G$-actions associated to unitary 
$G$-representations.
We shall equivariantly identify $\mathfrak{g}\cong \mathfrak{g}^\vee$
and $\mathfrak{t}\cong \mathfrak{t}^\vee$.

Let $\mathcal{C}:=(e_1,e_2)$ be the standard basis of $\mathbb{C}^2$.
For any $k=1,2,\ldots$, 
$W_k:=\mathrm{Sym}^k\left(\mathbb{C}^2\right)$ has an Hermitian structure
naturally induced from the standard one of $\mathbb{C}^2$.
An orthonormal basis of $W_k$ may be taken
$\mathcal{B}_k=\left(E_{k,j}\right)$, where 
\begin{equation}
\label{eqn:defn Ejk}
E_{k,j}:=c_{k,j}\,e_1^{k-j}\,e_2^j,
\quad c_{k,j}:=\sqrt{•\frac{(k+1)!}{\pi\,j!\,(k-j)!•}},\quad
j=0,1,\ldots,k.
\end{equation}
%
By means of $\mathcal{B}_k$, we shall unitarily identify $W_k\cong \mathbb{C}^{k+1}$, and a
point $w=\sum_{j=0}^k\,z_j\,E_{k,j}\in W_k$ with $Z=(z_j)_{j=0}^k\in \mathbb{C}^{k+1}$.

Consider the unitary representation $\mu=\mu_1$
of $G=U(2)$ on $W_1:=\mathbb{C}^2$ 
given by $B\mapsto (B^t)^{-1}$ with respect to  
$\mathcal{C}$. 
Then $\mu_1$ naturally induces for every $k$ a unitary representation of 
$G$ on $W_k$, which we may regard (given $\mathcal{B}_k$) as a 
a Lie group homomorphism 
$\mu_k:G\rightarrow U(k+1)$, with derivative
$\mathrm{d}\mu_k:\mathfrak{g}\rightarrow \mathfrak{u}(k+1)$.
Consequently, we have an induced holomorphic Hamiltonian action 
$\phi_k$
of $G$ on 
$\mathbb{P}^k=\mathbb{P}(W_k)$ with respect to $2\,\omega_{FS}$
(here $\omega_{FS}$ is the Fubini-Study form); let us compute its moment map
$\Phi_k:\mathbb{P}^k\rightarrow \mathfrak{g}$.

Let us set for simplicity 
$E_j=E_{k,j}$. We have for $\alpha\in \mathfrak{g}$
\begin{eqnarray}
\label{eqn:differential alpha}
\mathrm{d}\mu_k(\alpha)(E_j)&=&-\sqrt{j\,(k-j+1)}\,\alpha_{21}\,E_{j-1}\nonumber\\
&&-
\big[(k-j)\,\alpha_{11}+
j\,\alpha_{22}\big]\,E_{j}\nonumber\\
&&-\sqrt{(k-j)\,(j+1)}\,\alpha_{12}\,E_{j+1}.
\end{eqnarray}
Hence the only non-zero entries of $\mathrm{d}\mu_k(\alpha)$ are
\begin{eqnarray}
\label{eqn:non zero entries}
\mathrm{d}\mu_k(\alpha)_{j-1,j}&=&-\sqrt{j\,(k-j+1)}\,\alpha_{21},\nonumber\\
\mathrm{d}\mu_k(\alpha)_{j,j}&=&-
\big[(k-j)\,\alpha_{11}+
j\,\alpha_{22}\big],\nonumber\\
\mathrm{d}\mu_k(\alpha)_{j+1,j}&=&-\sqrt{(k-j)\,(j+1)}\,\alpha_{12}
\end{eqnarray}
for $j=0,\ldots,k$.
For $Z=(z_0,\ldots,z_k)^t\in \mathbb{C}^{k+1}$,
let us define the Hermitian matrix $(Z\odot \overline{Z})_{ij}:=
z_i\,\overline{z}_j$. As is well-known, 
the moment map for the action of $U(k+1)$ on 
$\left(\mathbb{P}^k,2\,\omega_{FS}\right)$, 
$\Gamma:\mathbb{P}^k\rightarrow \mathfrak{u}(k+1)$, 
is 
\begin{equation}
\label{eqn:moment map projective space}
\Gamma([Z]):=-\frac{\imath}{\|Z\|^2•}\,Z\odot \overline{Z}.
\end{equation}
Given (\ref{eqn:non zero entries}) and (\ref{eqn:moment map projective space}),
one obtains by standard arguments that the entries $\Phi_{ij}$ are given by
\begin{eqnarray}
\label{eqn:Phikji}
(\Phi_k)_{11}([Z])&=&\frac{\imath}{•\|Z\|^2}\,\sum_{j=0}^k\,(k-j)\,|z_j|^2,\\
(\Phi_k)_{12}([Z])&=&\frac{\imath}{•\|Z\|^2}\,
\sum_{j=0}^{k-1}\,\sqrt{•(k-j)\,(j+1)}z_{j+1}\,\overline{z}_j,\nonumber\\
(\Phi_k)_{21}([Z])&=&\frac{\imath}{•\|Z\|^2}\,\sum_{j=1}^k\,\sqrt{j\,(k-j+1)}\,
z_{j-1}\,\overline{z}_j,\nonumber\\
(\Phi_k)_{22}([Z])&=&\frac{\imath}{•\|Z\|^2}\,\sum_{j=0}^k\,j\,|z_j|^2.\nonumber
\end{eqnarray}

We can reformulate this in a more compact form, as follows. Let us define $F_{k,a}:
\mathbb{C}^{k+1}\rightarrow \mathbb{C}^k$ for $a=1,2$
by setting
\begin{eqnarray}
F_{k,1}(Z)&:=&
\begin{pmatrix}
\sqrt{k}	\,z_0\\
\sqrt{k-1}\,z_1\\
\vdots\\
z_{k-1}
\end{pmatrix}=\left(\sqrt{k-j+1}\,z_{j-1}\right)_{j=1}^k,\label{eqn:F1}\\
F_{k,2}(Z)&:=&
\begin{pmatrix}
z_1\\
\sqrt{2}\,z_2\\
\vdots\\
\sqrt{k} z_{k}
\end{pmatrix}=\left(\sqrt{j}\,z_{j}\right)_{j=1}^k.\label{eqn:F2}
\end{eqnarray}

Then 
\begin{equation}
\label{eqn:PhiFj}
\Phi_k([Z])=\frac{\imath}{\|Z\|^2•}\,
\begin{pmatrix}
\|F_{k,1}(Z)\|^2& F_{k,2}(Z)^t\,\overline{F_{k,1}(Z)}\\
F_{k,1}(Z)^t\,\overline{F_{k,2}(Z)}&\|F_{k,2}(Z)\|^2
\end{pmatrix}.
\end{equation}

\begin{defn}
\label{defn:L-1L_k}
Let $k\ge 1$. We shall denote by $L_k'$ the set
of all positive semidefinite Hermitian matrices of trace $k$ and rank $1$; thus
$L'_1$ is the set of orthogonal projectors onto a
1-dimensional vector subspace of $\mathbb{C}^2$, and $L'_k=k\,L'_1$.
Similarly, $L_k$ will denote the set
of all $2\times 2$ Hermitian positive semidefinite matrices of trace $k$.

\end{defn}
In particular, $L_k$ is the convex hull of $L_k'$, and $L_k=k\,L_1$.

\begin{prop}
\label{prop:Phi(M)}
$\Phi_1 \left(\mathbb{P}^1\right)=\imath\,L_1'$. If $k\ge 2$,
$\Phi_k \left(\mathbb{P}^k\right)=\imath\,L_k$.
\end{prop}

\begin{proof}
For $k=1$,
(\ref{eqn:PhiFj}) specializes to
\begin{equation}
\label{eqn:PhiF1}
\Phi_1([Z])=\frac{\imath}{\|Z\|^2•}\,
\begin{pmatrix}
|z_0|^2& z_1\,\overline{z_0}\\
z_0\,\overline{z_1}&|z_1|^2
\end{pmatrix},
\end{equation}
which implies the first statement.

Let us then assume $k\ge 2$.
It is evident from (\ref{eqn:Phikji}) and (\ref{eqn:PhiFj})
that $\Phi_k \left(\mathbb{P}^k\right)\subseteq \imath\,L_k$. Since
$\Phi_k \left(\mathbb{P}^k\right)$ is $G$-invariant in view of the 
$G$-equivariance of $\Phi_k$, to prove the reverse inclusion it suffices
to show that for any $\lambda\in [0,k]$ we have
$$
\imath\,\begin{pmatrix}
\lambda &0\\
0&k-\lambda
\end{pmatrix}\in \Phi_k \left(\mathbb{P}^k\right).
$$
To this end, we need only set $z_0=\sqrt{\lambda/k}$, $z_j=0$ for $j=1,\ldots,k-1$,
$z_k=\sqrt{(k-\lambda)/k}$.

\end{proof}

If $\boldsymbol{\nu}=\begin{pmatrix}
\nu_1&\nu_2
\end{pmatrix}^t\in \mathbb{R}^2$, we shall denote by
$D_{\boldsymbol{\nu}}$ the diagonal matrix with entries $\nu_1,\,\nu_2$ and
by $\mathcal{O}_{\boldsymbol{\nu}}\subset \mathfrak{g}$ the orbit of 
$\imath\,D_{\boldsymbol{\nu}}$.

Also, let us set
$$
J_k:=\left\{\begin{pmatrix}
\nu_1\\
\nu_2
\end{pmatrix}\,:\,\nu_1,\,\nu_2\ge 0,\,
\nu_1+\nu_2=k\right\},\quad J_{k+}:=\left\{\begin{pmatrix}
\nu_1\\
\nu_2
\end{pmatrix}\in J\,:\,\nu_1\ge \nu_2\right\}.
$$
In other words, $J_k$ is the segment joining the points 
$\begin{pmatrix}
k&0
\end{pmatrix}^t,\,\begin{pmatrix}
0&k
\end{pmatrix}^t\in \mathbb{R}^2$.

\begin{cor}
\label{cor:nonemptyintersection}
In the situation of Proposition \ref{prop:Phi(M)}, 
$\Phi_1 \left(\mathbb{P}^1\right)=
\mathcal{O}_{\boldsymbol{\epsilon}_1}$,
where $\boldsymbol{\epsilon}_1=
\begin{pmatrix}
1&0
\end{pmatrix}$, while 
for any $k\ge 2$
\begin{equation}
\label{eqn:Phi(M)orbits}
\Phi_k \left(\mathbb{P}^k\right)=\bigcup_{\boldsymbol{\nu}\in J_k}\mathcal{O}_{\boldsymbol{\nu}}
=\bigcup_{\boldsymbol{\nu}\in J_{k+}}\mathcal{O}_{\boldsymbol{\nu}}.
\end{equation}
In particular, if $\boldsymbol{\nu}\neq \mathbf{0}$ and 
$\nu_1\ge \nu_2$, then 
$\Phi_k \left(\mathbb{P}^k\right)\cap \mathcal{C}(\mathcal{O}_{\boldsymbol{\nu}})
\neq \emptyset$ if and only if $\nu_2\ge 0$.
\end{cor}

The second equality in (\ref{eqn:Phi(M)orbits}) follows from the fact that
if $\boldsymbol{\nu}=\begin{pmatrix}
\nu_1&\nu_2
\end{pmatrix}^t$ and 
$\boldsymbol{\nu}'=\begin{pmatrix}
\nu_2&\nu_1
\end{pmatrix}^t$, then $\mathcal{O}_{\boldsymbol{\nu}}=
\mathcal{O}_{\boldsymbol{\nu}'}$.

Let us denote by $\psi_k$ the restricted action of $T$ on
$\mathbb{P}^k$, and by $\Psi_k:M\rightarrow \mathfrak{t}^\vee\cong \mathfrak{t}$ 
its moment map. Then $\Psi_k$
is the composition of $\Phi_k$ with the orthogonal projection
$\pi:\mathfrak{g}\rightarrow \mathfrak{t}$; the latter amounts to selecting the diagonal
component of a matrix in $\mathfrak{g}$.

\begin{cor}
\label{cor:imagePsi}
For any $k\ge 1$,
$\Psi_k \left(\mathbb{P}^k\right)=\imath\,J_k\subset \imath\,\mathbb{R}^2$.

\end{cor}
\begin{proof}
[Proof of Corollary \ref{cor:imagePsi}]
For $k=1$, this is immediate from
(\ref{eqn:PhiF1}). Assume then $k\ge 2$.
Any matrix in $L_k$ has diagonal part in $J_k$, hence
$\Psi_k \left(\mathbb{P}^k\right)\subseteq\imath\,J_k\subset \imath\,\mathbb{R}^2$
by Proposition \ref{prop:Phi(M)}.
Conversely, for any $\boldsymbol{\lambda}:=\begin{pmatrix}
\lambda&k-\lambda
\end{pmatrix}^t\in J_k$ in the proof of Proposition \ref{prop:Phi(M)}
 we have found $[Z]\in \mathbb{P}^k$ such that $\Phi ([Z])=\imath\,D_{\boldsymbol{\lambda}}$.
 Hence $\Psi_k([Z])=\imath\,\boldsymbol{\lambda}$. 
\end{proof}

Let us notice the following consequence of 
Proposition \ref{prop:Phi(M)}, due to the fact the diagonal part of a matrix in $L_k$ is in $L_k$:

\begin{cor}
\label{cor:intersetion=projection k}
For any $k\ge 2$,
$\Psi_k \left(\mathbb{P}^k\right)=\Phi_k \left(\mathbb{P}^k\right)
\cap \mathfrak{t}$.
\end{cor}

\begin{proof}
[Proof of Corollary \ref{cor:intersetion=projection k}]
Obviously $\Psi_k \left(\mathbb{P}^k\right)\supseteq
\Phi_k \left(\mathbb{P}^k\right)
\cap \mathfrak{t}$. Conversely, suppose $\alpha\in \Psi_k \left(\mathbb{P}^k\right)$.
Viewing $\alpha$ as the diagonal component of a matrix $\alpha'\in \Phi_k \left(\mathbb{P}^k\right)$, we conclude that $-\imath\,\alpha$ 
has non-negative (diagonal) entries and trace $k$. Hence $\alpha\in \imath\,L_k=\Phi_k \left(\mathbb{P}^k\right)$.
\end{proof}

Having characterized the images of $\Phi_k$ and $\Psi_k$, let us determine
the orbital cones 
to which they are transverse.
By Corollary \ref{cor:nonemptyintersection}
we may assume $k\ge 2$.

\begin{thm}
\label{thm:trasnversality1}
Assume that $k\ge 2$, 
$\nu_1,\,\nu_2\ge 0$ and $\nu_1\neq \nu_2$. Then the following conditions are equivalent:
\begin{enumerate}
\item $\Phi_k$ is transverse to $\mathcal{C}(\mathcal{O}_{\boldsymbol{\nu}})$;
\item $j\,\nu_1\neq (k-j)\,\nu_2$ for all $j\in \{0,1,\ldots,k\}$.
\end{enumerate}

\end{thm}

\begin{rem}
Since $\Phi_k (\mathbb{P}^k)=\imath\,L_k$, if $\boldsymbol{\nu}=\pm
\begin{pmatrix}
1&-1
\end{pmatrix}$ then $\Phi_k (\mathbb{P}^k)\cap \imath\,\mathbb{R}_+\cdot \boldsymbol{\nu}
=\emptyset$, hence we may assume $\nu_1+\nu_2\neq 0$.
Furthermore, 
$\Phi_k (\mathbb{P}^k)$ is $G$-invariant and if
$\boldsymbol{\nu}':=\begin{pmatrix}
\nu_2&\nu_1
\end{pmatrix}$ 
then the matrices 
the diagonal matrices $\imath\,D_{\boldsymbol{\nu}}$ and 
$\imath\,D_{\boldsymbol{\nu}'}$ belong to the same orbit. We may assume therefore 
$\nu_1\ge \nu_2$, hence - under the hypothesis of Theorem \ref{thm:trasnversality1} - that
$\nu_1>\nu_2$.
\end{rem}

\begin{proof}[Proof of Theorem \ref{thm:trasnversality1}]
Let $X_k=S^{2k+1}$ be viewed as the unit circle bundle of the
tautological line bundle on $\mathbb{P}^k=\mathbb{P}(W_k)$, with projection
$\pi_k:X_k\rightarrow \mathbb{P}^k$ (the Hopf map), and let us set
$$
(X_k)^G_{\boldsymbol{\nu}}:=
\pi_k^{-1}\left(\mathbb{P}(W_k)^G_{\boldsymbol{\nu}}\right),
\quad
(X_k)^G_{\mathcal{O}}:=\pi_k^{-1}\left(\mathbb{P}(W_k)^G_{\mathcal{O}}\right).
$$
Since $\phi_k$ is induced by the unitary representation $\mu_k$
on $W_k$, there is
by restriction of $\mu_k$ a natural lift of $\phi_k$ to an action on $X_k$, 
which we shall denote $\tilde{\phi}_k$. 
We shall also set
$\tilde{\Phi}_k:=\Phi_k\circ \pi_k:X_k\rightarrow \mathfrak{g}$,
$Z\mapsto \Phi_k([Z])$.

By the discussions in \S 2.2 of \cite{pao-tori} 
and \S 4.1.1 of \cite{gp}, $\Phi_k$ is transverse to 
$\mathcal{C}(\mathcal{O}_{\boldsymbol{\nu}})$ if and only if
$\tilde{\phi}_k$ is locally free on $ (X_k)^G_{\mathcal{O}}$;
furthermore, since
$ (X_k)^G_{\mathcal{O}}$ is the $G$-saturation of $(X_k)^G_{\boldsymbol{\nu}}$,
the latter condition is in turn equivalent to $\tilde{\phi}_k$ being locally free
along $(X_k)^G_{\boldsymbol{\nu}}$.

For any $\beta\in \mathfrak{g}$, let $\beta_{X_k}\in \mathfrak{X}(X_k)$ denote
the associated vector field on $X_k$. For any 
$Z\in X_k$, let $\mathfrak{g}_{X_k}(Z)\subseteq T_ZX_k$ denote the vector subspace
given by the evaluations of all the $\beta_{X_k}$'s at $Z$, 
and similarly for $\mathfrak{t}$. 
Then $\tilde{\phi}$ is
locally free at $Z$ 
if and only if the evaluation map 
$\mathrm{val}_Z:\mathfrak{g}\rightarrow T_ZX_k$, $\beta\mapsto \beta_{X_k}(Z)$,
has maximal rank, that is, $\mathfrak{g}\cong \mathfrak{g}_{X_k}(Z)$. 

Let us prove that 2.)
implies 1.). Let us remark that 2.) can be equivalently reformulated as follows:

\begin{equation}
\label{eqn:critical values}
\nu_1\cdot \nu_2\neq 0 \quad
\text{and}\quad\nu_1\neq \frac{k-j}{j•}\,\nu_2,\quad \text{for all}\quad j=1,\ldots,k-1.
\end{equation}

Let us consider $Z=(z_0,\ldots,z_k)^t\in (X_k)^G_{\boldsymbol{\nu}}$, so that
$$
\tilde{\Phi}_k(Z)=\imath\,
\begin{pmatrix}
\|F_{k,1}(Z)\|^2& F_{k,2}(Z)^t\,\overline{F_{k,1}(Z)}\\
F_{k,1}(Z)^t\,\overline{F_{k,2}(Z)}&\|F_{k,2}(Z)\|^2
\end{pmatrix}=\imath\,\frac{k}{\nu_1+\nu_2}\,
\begin{pmatrix}
\nu_1&0\\
0&\nu_2
\end{pmatrix}.
$$
In particular, 
\begin{equation}
\label{eqn:norms relation}
\nu_2\,\|F_{k,1}(Z)\|^2=\nu_1\,\|F_{k,2}(Z)\|^2.
\end{equation}

\begin{lem} 
\label{lem:nonzerocomponents}
Given (\ref{eqn:critical values}), 
for any $Z\in (X_k)^G_{\boldsymbol{\nu}}$ there exist $j,l\in \{0,1,\ldots,k\}$ with $j\neq l$ and
$z_j\cdot z_l\neq 0$.
\end{lem}

\begin{proof}
[Proof of Lemma \ref{lem:nonzerocomponents}]
If not, $Z$ has only one non-zero component, say $z_j\in S^1$.
Since by (\ref{eqn:critical values}) and (\ref{eqn:norms relation}) $F_1(Z),\,F_2(Z)\neq \mathbf{0}$,
we need to have $0<j<k$ in view of the definition of $F_j$. We conclude again
by (\ref{eqn:norms relation}) that
$\nu_2\,(k-j)=\nu_1\,j$ for some $j=1,\ldots,k-1$, against the assumption.
\end{proof}

Let $D\in T\leqslant G$ be a diagonal matrix
with entries $e^{\imath\,\vartheta_1},\,e^{\imath\,\vartheta_2}\in S^1$. 
By definition of $\phi$ and of the $E_a$'s in (\ref{eqn:defn Ejk}), 
we have with $Z=(z_a)_{a=0}^k$
$$
\tilde{\phi}_D(Z)=\left(
e^{-\imath\,[(k-a)\,\vartheta_1+a\,\vartheta_2]}\,z_a\right).
$$
Now suppose that $D$ is close to $I_2$, so that we may assume 
$\vartheta_1,\,\vartheta_2\sim 0$,
and that $D$ fixes $Z$. Then 
$e^{\imath\,[(k-a)\,\vartheta_1+a\,\vartheta_2]}\,z_a=z_a$ 
for every $a=0,\ldots,k$ implies in particular
$(k-j)\,\vartheta_1+j\,\vartheta_2=(k-l)\,\vartheta_1+l\,\vartheta_2=0$,
and so $\vartheta_1=\vartheta_2=0$. Hence, there is a neighborhood
$T'\subseteq T$ of $I_2$ such that the only $D\in T'$ that fixes $Z$ is 
$I_2$. In other words, $T$ acts locally freely on $(X_k)^G_{\boldsymbol{\nu}}$
at $Z$. 
In particular, $\mathrm{val}_Z:\mathfrak{t}\rightarrow T_ZX_k$
is injective.

By the equivariance of $\Phi$, for
any $W\in X_k$ and $\beta\in \mathfrak{g}$ we have
\begin{equation}
\label{eqn:derivata Phi}
\mathrm{d}_W\tilde{•\Phi }\big(\beta_{X_k}(W)\big)=
\left[\beta,\tilde{•\Phi }(W)\right].
\end{equation}
Hence if $\beta\in \mathfrak{t}\subset \mathfrak{g}$ 
and $Z\in (X_k)^G_{\boldsymbol{\nu}}$ then
$\mathrm{d}_Z\tilde{•\Phi }\big(\beta_{X_k}(Z)\big)=0$; that is,
\begin{equation}
\label{eqn:kerdPhiandt}
{t}_{X_k}(Z)\subseteq \ker\big(\mathrm{d}_Z\tilde{•\Phi }\big)
\qquad (Z\in (X_k)^G_{\boldsymbol{\nu}}).
\end{equation}

Now let us define
\begin{equation}
\label{eqn:defn of xieta}
\eta:=\begin{pmatrix}
0&1\\
-1&0
\end{pmatrix},
\quad
\xi:=\begin{pmatrix}
0&\imath\\
\imath&0
\end{pmatrix},\quad
\mathfrak{a}:=\mathrm{span}(\eta,\,\xi)\subset \mathfrak{g},
\end{equation}
so that $\mathfrak{g}=\mathfrak{a}\oplus \mathfrak{t}$. 
By (\ref{eqn:derivata Phi}) we have at $Z\in (X_k)^G_{\boldsymbol{\nu}}$:
\begin{equation}
\label{eqn:xieta Phi-derivative}
\mathrm{d}_Z\tilde{\Phi}\big(\xi_{X_k}(Z)\big)
=\frac{k\,(\nu_1-\nu_2)}{\nu_1+\nu_2}\,\eta,\quad 
\mathrm{d}_Z\tilde{\Phi}\big(\eta_{X_k}(Z)\big)
=-\frac{k\,(\nu_1-\nu_2)}{\nu_1+\nu_2}\,\xi .
\end{equation}
Let us set
$$
\rho:=\begin{pmatrix}
\imath&0\\
0&0
\end{pmatrix},
\quad \gamma
:=\begin{pmatrix}
0&0\\
0&\imath
\end{pmatrix}.
$$
Then $(\rho,\gamma)$ is a basis for $\mathfrak{t}$, 
and $(\eta,\xi,\rho,\gamma)$ is a basis for $\mathfrak{g}$.

Suppose that for some $x,y,z,t\in \mathbb{R}$ we have
$x\,\eta+y\,\xi+z\,\rho+t\,\gamma\in \ker (\mathrm{val}_Z)$:
\begin{equation}
\label{eqn:zero evaluation}
x\,\eta_{X_k}(Z)+y\,\xi_{X_k}(Z)+z\,\rho_{X_k}(Z)+t\,\gamma_{X_k}(Z)=0.
\end{equation}
Applying $\mathrm{d}_Z\Phi$, we get by 
(\ref{eqn:kerdPhiandt}) and (\ref{eqn:xieta Phi-derivative}):
\begin{eqnarray}
\label{eqn:xietarhogamma Phi-derivative}
0&=&x\,\mathrm{d}_Z\tilde{\Phi}\big(\eta_{X_k}(Z)\big)
+y\,\mathrm{d}_Z\tilde{\Phi}\big(\xi_{X_k}(Z)\big)\nonumber\\
&=&\frac{k\,(\nu_1-\nu_2)}{\nu_1+\nu_2}\,(-x\,\xi+y\,\eta).
\end{eqnarray}
Hence $x=y=0$, so that $z\,\rho_{X_k}(Z)+t\,\gamma_{X_k}(Z)=0$. But this
means that $z\,\rho+t\,\gamma\in 
\ker\left(\left.\mathrm{val}_Z\right|_{\mathfrak{t}}\right)$=(0);
thus we also have $z=t=0$.
We conclude that $\ker (\mathrm{val}_Z)=(0)$ 
for any $Z\in (X_k)^G_{\boldsymbol{\nu}}$, as claimed.

Now let us suppose instead 
that (\ref{eqn:critical values}) does not hold.
We aim to show
that then $\tilde{\phi}$ is not everywhere locally free along $(X_k)^G_{\boldsymbol{\nu}}$.  
If $\nu_2=0$, let $Z:=
\begin{pmatrix}
1&0&\cdots&0
\end{pmatrix}^t$. Then $\tilde{\Phi}=\imath\,D$, where $D$ is the diagonal matrix 
with diagonal entries $\begin{pmatrix}
k&0
\end{pmatrix}$; hence $Z\in (X_k)^G_{\boldsymbol{\nu}}$. On the other hand,
$Z$ is fixed by the 1-dimensional subgroup of $G$ 
of diagonal matries with diagonal entries $\begin{pmatrix}
1&e^{\imath \vartheta}
\end{pmatrix}$, hence $\tilde{\phi}$ is not free at $Z$. One argues similarly when
$\nu_1=0$, by choosing instead $Z:=
\begin{pmatrix}
0&\cdots&0&1
\end{pmatrix}^t$. 
If instead $\nu_1\cdot \nu_2\neq 0$, then
$\nu_1=[(k-j)/j]\,\nu_2$ for some $j=1,\ldots,k-1$. 
Let us consider $Z=(z_l)$ with $z_l=\delta_{lj}$, $l=0,\ldots,k$. Then
by (\ref{eqn:PhiFj}) $Z\in (X_k)^G_{\boldsymbol{\nu}}$.
On the other hand now $Z$ is fixed by the 1-dimensional subgroup of diagonal matrices 
with diagonal entries $\begin{pmatrix}
e^{-\imath\,j\,\vartheta}&e^{\imath (k-j)\,\vartheta}
\end{pmatrix}$, hence again $\tilde{\phi}$ is not free at $Z$.

\end{proof}

Let us note in passing that 
the argument in the proof of Theorem \ref{thm:trasnversality1} can be phrased 
in slightly more general terms and actally establishes the
following criterion.

\begin{lem}
\label{lem:trasversality criterion}
Suppose that $(M,J)$ is a complex projective manifold, with
$\omega$ a Hodge form on it, associated to a positive line bundle 
$(A,h)$.
Let $\phi:G\times M\rightarrow M$ be a holomorphic Hamiltonian 
action on $(M,2\,\omega)$, with moment map $\Phi:M\rightarrow \mathfrak{g}$.
Let $X\subset A^\vee$ be the unit circle bundle, with projection
$\pi:X\rightarrow M$, and assume
that there is a contact lift $\tilde{\phi}:G\times X\rightarrow X$
of the Hamiltonian action $(\phi,\Phi)$.
Suppose $\nu_1\neq \nu_2$, $x\in X$, $\Phi\circ \pi(x)\in \mathbb{R}_+\cdot \imath\,D_{\boldsymbol{\nu}}$,
and that $T$ acts locally freely at $x$. Then $G$ acts locally freely at $x$.
\end{lem}

\begin{cor}
In the situation of Lemma \ref{lem:trasversality criterion}, assume in addition that
$T$ acts locally freely along the inverse image $X^G_{\boldsymbol{\nu}}$
of $M^G_{\boldsymbol{\nu}}$ in $X$. Then $\Phi$ is transverse to 
$\mathcal{C}(\mathcal{O}_{\boldsymbol{\nu}})$.
\end{cor}

Next we shall consider the transversality issue for $\Psi_k$.

\begin{thm}
\label{thm:psi transversality}
For any $k\ge 1$,
$\Psi_k$ is not transverse to a ray 
$\mathbb{R}_+\,\boldsymbol{\nu}\subset \imath\,\mathfrak{t}\cong \mathbb{R}^2$
if and only if $\boldsymbol{\nu}$ is a positive multiple of 
$\begin{pmatrix}
k-j&j
\end{pmatrix}^t$ for some
$j=0,1,\ldots,k$.
\end{thm}

In other words, the critical rays are those through the points in the intersection
$J\cap \mathbb{Z}^2$, up to the factor $\imath$.

\begin{proof}[Proof of Theorem \ref{thm:psi transversality}]
Let $\tilde{\psi}_k$ denote the action of $T$ on $X_k$.
As argued in the proof of Theorem \ref{thm:trasnversality1},
$\tilde{•\psi}_k$ is not locally free at $Z=(z_l)\in X$ if and only if $|z_l|=\delta_{lj}$
for some $j=0,\ldots,k$. Hence the rays in $\mathfrak{t}$ to which 
$\Psi$ is not transverse are those through the images 
under $\Psi$ if the vectors of the standard basis of $\mathbb{C}^{k+1}$. As we
have remarked, their images under $\tilde{•\Phi}_k$ form the set
$$
\left\{ \imath\,\begin{pmatrix}
k-j&0\\
0&j
\end{pmatrix}:j=0,\ldots,k\right\},
$$
and we need only take the diagonal part to reach the claimed conclusion.

\end{proof}

Let us now extend the previous considerations 
to a general irreducible representation of $G$. 
More precisely, 
we shall denote by $\mu_{k,l}$ the composition of the representation 
$\det^{\otimes l}\otimes \mathrm{Sym}^k\left(\mathbb{C}^2\right)$ with
the Lie group automorphism $B\mapsto (B^t)^{-1}$:
\begin{equation}
\label{eqn:composition mukl}
(\mu_{k,l})_B(v):=\det(B)^{- l}\,{\mu_k}_{(B^t)^{-1}}(v)\quad
(B\in G, \,v\in W_k\cong
\mathbb{C}^{k+1}).
\end{equation} 

The induced action $\phi_{k,l}$ on $\mathbb{P}^k$ equals $\phi_k$; however,
the change in linearization implies a change in the moment map.
For any $l\in \mathbb{Z}$,
$\mu_{0,l}$ is the representation on $\mathbb{C}$ given by the character 
$\det ^{-l}$. In this case, $\mathbb{P}^0=\{[1]\}$ is just a point, and we shall
take as definition of moment map the function 
$\Phi_{0,l}:[1]\mapsto \imath\,l\,I_2$.
For $k\ge 1$, let us view $\mu_{k,l}$ as a Lie group morphism
$G\rightarrow U(k+1)$.
Then, in place of (\ref{eqn:differential alpha}), we have for $\alpha\in \mathfrak{g}$
\begin{eqnarray}
\label{eqn:differential alphatilde}
\mathrm{d}\mu_{k,l}(\alpha)(E_j)&=&-\sqrt{j\,(k-j+1)}\,\alpha_{21}\,E_{j-1}\nonumber\\
&&-
\big[l\,\mathrm{trace}(\alpha)+(k-j)\,\alpha_{11}+
j\,\alpha_{22}\big]\,E_{j}\nonumber\\
&&-\sqrt{(k-j)\,(j+1)}\,\alpha_{12}\,E_{j+1}.
\end{eqnarray}
It follows that the new moment map, $\Phi_{k,l}:\mathbb{P}^k\rightarrow \mathfrak{g}$,
is given by
\begin{equation}
\label{eqn:Phi'}
\Phi_{k,l}([Z]):=\Phi_k([Z])+\imath\,l\,I_2,
\end{equation}
where $\Phi_k$ is as in (\ref{eqn:PhiFj}). Therefore, 
with the notation of Proposition \ref{prop:Phi(M)},
\begin{equation}
\label{eqn:immaginePhi'}
\Phi_{1,l}\left(\mathbb{P}^1\right)=\imath\,L_1'+\imath\,l\,I_2,\qquad 
\Phi_{k,l}\left(\mathbb{P}^k\right)=\imath\,L_k+\imath\,l\,I_2\quad\forall\,k\ge 2.
\end{equation}

Let us set 
$$\boldsymbol{\zeta}:=\begin{pmatrix}
1&1
\end{pmatrix} ^t,
\qquad 
J_{k,l}:=J_k+l\,\boldsymbol{\zeta}\subset \mathbb{R}^2.
$$ 
Thus $J_{k,l}$ is the segment joining 
$\begin{pmatrix}
k+l&l
\end{pmatrix}$ and $\begin{pmatrix}
l&k+l
\end{pmatrix}$. 
Also, let $\mathcal{C}_{k,l}\subset \mathbb{R}^2\setminus  \{0\}$
be the closed cone through $J_{k,l}$.

Then in place of Corollaries \ref{cor:nonemptyintersection}
and \ref{cor:imagePsi} we have:

\begin{cor}
\label{cor:nonemptyintersection'}
Under the previous assumptions,
$$
\Phi_{1,l}\left(\mathbb{P}^1\right)=\mathcal{O}_{\boldsymbol{\epsilon}_1+l\,\boldsymbol{\zeta}}
=\mathcal{O}_{\boldsymbol{\epsilon}_1}
+l\,I_2,
$$
and for $k\ge 2$
\begin{equation}
\label{eqn:Phi'(M)orbits}
\Phi_{k,l}\left(\mathbb{P}^k\right)=\bigcup_{\boldsymbol{\nu}\in J_k}\mathcal{O}_{\boldsymbol{\nu}+l\,\boldsymbol{\zeta}}
=\bigcup_{\boldsymbol{\nu}\in J_{k+}}\mathcal{O}_{\boldsymbol{\nu}+l\,\boldsymbol{\zeta}}
=\bigcup_{\boldsymbol{\nu}\in J_{k,l}}\mathcal{O}_{\boldsymbol{\nu}}.
\end{equation}
In particular, if $\boldsymbol{\nu}\neq \mathbf{0}$
then 
$\Phi_{k,l}\left(\mathbb{P}^k\right)\cap \mathcal{C}(\mathcal{O}_{\boldsymbol{\nu}})
\neq \emptyset$ if and only if $\boldsymbol{\nu}\in \mathcal{C}_{k,l}$.

\end{cor}

\begin{cor}
\label{cor:imagePsi'}
If $\Psi_{k,l}:\mathbb{P}^k\rightarrow \mathfrak{t}\cong \imath\,\mathbb{R}^2$ is the moment
map for $\psi$ with respect to $\mu_{k,l}$, 
then
\begin{equation}
\label{eqn:immaginePsi'}
\Psi_{k,l}\left(\mathbb{P}^k\right)=\imath\,J_{k,l}.
\end{equation}
Hence $\Psi_{k,l}\left(\mathbb{P}^k\right)\cap \mathbb{R}_+\cdot\boldsymbol{\nu}\neq \emptyset$ 
if and only if $\boldsymbol{\nu}\in \mathcal{C}_{k,l}$.
\end{cor}

The latter Corollary can of course be derived also by the Convexity Theorem in
\cite{atiyah} and \cite{gs cp1}. Let us also remark the following analogue of Corollary \ref{cor:intersetion=projection k}:

\begin{cor}
\label{cor:intersetion=projection kl}
For any $k\ge 2$ and $l\in \mathbb{Z}$,
$\Psi_{k,l} \left(\mathbb{P}^k\right)=\Phi_{k,l} \left(\mathbb{P}^k\right)
\cap \mathfrak{t}$.
\end{cor}

Let us now consider the issue of transversality in this case. 
By Corollary \ref{cor:nonemptyintersection'},
we may assume $k\ge 1$.
Furthermore,
by Proposition \ref{prop:Phi(M)}
and (\ref{eqn:immaginePhi'}), $\Phi_{k,l}\left(\mathbb{P}^k\right)\subset V_{k+2l}$,
where $V_r\subset \mathfrak{g}$ is the affine subspace of skew-Hermitian matrices of trace
$\imath\,r$. If, in particular, $k+2l=0$ then $\Phi_{k,l}\left(\mathbb{P}^k\right)$ 
lies in a proper invariant 
vector subspace (the kernel of the trace), and is therefore not transverse to any cone 
$\mathcal{C}(\mathcal{O})$ in $\mathfrak{g}$
intersecting its image. In fact, if $\mathcal{C}(\mathcal{O})\cap V_0\neq \emptyset$
then by invariance $\mathcal{C}(\mathcal{O})\subset V_0$.
Thus we assume $k+2l\neq 0$.

Let us denote by $\tilde{\phi}_{k,l}$ and $\tilde{\psi}_{k,l}$, respectively, 
the actions of $G$ 
and $T$ on $X_k$ given by the restrictions of the unitary representation 
$\mu_{k,l}$. Let $(X'_k)^G_{\mathcal{O}}$, $(X'_k)^G_{\boldsymbol{\nu}}$ and 
$(X_k')^T_{\boldsymbol{\nu}}$ be defined as $(X_k)^G_{\mathcal{O}}$, 
$(X_k)^G_{\boldsymbol{\nu}}$ and 
$(X_k)^T_{\boldsymbol{\nu}}$, but in terms of the new moment maps $\Phi_{k,l}$ and $\Psi_{k,l}$.
Then, just as before, $\Phi_{k,l}$ is transverse to $\mathcal{C}(\mathcal{O}_{\boldsymbol{\nu}})$
if and only if $\tilde{\phi}_{k,l}$ is locally free at every $Z\in (X_k')^G_{\boldsymbol{\nu}}$,
and $\Psi_{k,l}$ is transverse to $\mathbb{R}_+\cdot \imath\,\boldsymbol{\nu}$ 
if and only if $\tilde{\psi}_{k,l}$
is locally free on $(X'_k)^T_{\boldsymbol{\nu}}$.

Suppose that $Z\in X_k$.
If for some $j=0,\ldots,k$ we have $z_i=0$ for all $i\neq j$, then arguing as in the
proof of Theorem \ref{thm:trasnversality1} one sees that $\tilde{\psi}_{k,l}$ 
is not locally free at $Z$, and therefore neither is $\tilde{\phi}_{k,l}$. 
In this case we 
have, with $\tilde{\Phi}_{k,l}:=\Phi_{k,l}\circ \pi$:
\begin{equation}
\label{eqn:tildePhikl}
\tilde{\Phi}_{k,l}(Z):=\imath\,
\begin{pmatrix}
k-j+l & 0\\
0& j+l
\end{pmatrix}.
\end{equation}

If, conversely, $Z\in X_k$ and $z_l\cdot z_j\neq 0$ for distinct $j,\,l\in \{0,\ldots,k\}$, then
a slight adaptation of the previous arguments shows that
$\tilde{\psi}_{k,l}$ is locally free at $Z$.
Hence we conclude the following variant of Theorem \ref{thm:psi transversality}:

\begin{thm}
\label{thm:transversalitykl}
Suppose that $k\ge 2$ and $k+2\,l\neq 0$.
Let us define
$$
\boldsymbol{\nu}_{k,j,l}:=\begin{pmatrix}
k-j+l & j+l
\end{pmatrix}^t,\qquad 
j=0,\ldots,k.
$$ 
Then $\Psi_{k,l}$ is not transverse 
to $\mathbb{R}_+\,\imath\,\boldsymbol{\nu}$ if and only if 
$\boldsymbol{\nu}\in \mathbb{R}_+\cdot \boldsymbol{\nu}_{k,j,l}$ for some $j=0,\ldots,k$.
\end{thm}

The previous argument clearly also shows that $\Phi_{k,l}$ is not transverse to 
$\mathcal{C}(\mathcal{O}_{\boldsymbol{\nu}_{k,j,l}})\subset \mathfrak{g}$.
In fact, on the one hand if $Z$ is the $j$-th basis vectors, then $\tilde{\psi}_{k,l}$ is not locally
free at $Z$, and therefore \textit{a fortiori} neither is $\tilde{\phi}_{k,l}$.
On the other hand, by (\ref{eqn:tildePhikl}) we also have $Z\in (X'_k)^G_{\boldsymbol{\nu}}$.

Let us assume on the other hand that $\boldsymbol{\nu}\not\in \mathbb{R}_+\boldsymbol{\nu}_{j,l}$ for every
$j$ and that $\nu_1\neq \nu_2$.
If $Z\in (X'_k)^G_{\boldsymbol{\nu}}$,
then there exist $l,j\in \{0,\ldots,k\}$ such that $z_l\cdot z_j\neq 0$.
If $D\in T$ is a diagonal matrix with diagonal entries 
$\begin{pmatrix}
e^{\imath\,\vartheta_1} & e^{\imath\,\vartheta_2}
\end{pmatrix}$ that fixes $Z$, then we need to have
$e^{\imath\,[l(\vartheta_1+\vartheta_2)+(k-a)\,\vartheta_1+a\,\vartheta_2]}=1$ for 
$a=j,\,l$. If $D$ is close to $I_2$, and we assume that $\vartheta_j\sim 0$, we deduce as
in the proof of Theorem \ref{thm:trasnversality1}
that $\vartheta_1=\vartheta_2=0$. Hence $\tilde{\psi}_{k,l}$ 
is locally free on $(X'_k)^G_{\boldsymbol{\nu}}$.
To conclude that $\tilde{\phi}_{k,l}$ is also locally free 
along $(X'_k)^G_{\boldsymbol{\nu}}$, we may now argue
using (\ref{eqn:derivata Phi}) 
as in the proof of Theorem \ref{thm:trasnversality1} (the second summand in (\ref{eqn:Phi'})
does not alter commutators).
Hence we have the following variant of Theorem \ref{thm:trasnversality1}:

\begin{thm}\label{thm:transversality2}
Suppose $k\ge 2$, $k+2l\neq 0$ and $\nu_1\neq \nu_2$. Then $\Phi_{k,l}$ is transverse to $\mathcal{C}(\mathcal{O}_{\boldsymbol{\nu}})$
if and only if $\boldsymbol{\nu}\not\in \mathbb{R}_+\cdot \boldsymbol{\nu}_{k,j,l}$ for every 
$j=0,\ldots,k$. 
\end{thm}

Let us now come to a general representation space of the form
\begin{equation}
\label{eqn:general representation}
W_{\mathbf{L},\mathbf{K}}:=
\bigoplus_{a=1}^r
{\det}^{\otimes l_a}\otimes \mathrm{Sym}^{k_a}\left(\mathbb{C}^2\right),
\end{equation}
where $\mathbf{L}=(l_a)\in \mathbb{Z}^r$, $\mathbf{K}=(k_a)\in \mathbb{N}^r$,
as usual composed with 
the Lie group automorphism $B\mapsto (B^t)^{-1}$ 
(see (\ref{eqn:composition mukl})).
As an abstract vector space, 
$$
W_{\mathbf{L},\mathbf{K}}\cong \bigoplus_{a=1}^r\mathbb{C}^{k_a+1}\cong
\mathbb{C}^{|\mathbf{K}|+r}\quad\Rightarrow\quad
\mathbb{P}(W_{\mathbf{L},\mathbf{K}})\cong \mathbb{P}^{|\mathbf{K}|+r-1},
$$
where $|\mathbf{K}|=\sum_ak_a$. Hence the corresponding morphism
of Lie groups $\mu_{\mathbf{L},\mathbf{K}}:G\rightarrow U(|\mathbf{K}|+r)$
is given by
$$
\mu_{\mathbf{L},\mathbf{K}}(g):=
\begin{pmatrix}
\mu_{l_1,k_1}(g)&& \\
&\ddots& \\
&&\mu_{l_r,k_r}(g)
\end{pmatrix}.
$$
Let us denote by $\phi_{\mathbf{L},\mathbf{K}}$ and $\psi_{\mathbf{L},\mathbf{K}}$,
respectively, the induced Hamiltonian actions of $G$ and $T$ on 
$\mathbb{P}(W_{\mathbf{L},\mathbf{K}})$, and by 
$\Phi_{\mathbf{L},\mathbf{K}}:\mathbb{P}(W_{\mathbf{L},\mathbf{K}})\rightarrow
\mathfrak{g}$, $\Psi_{\mathbf{L},\mathbf{K}}:\mathbb{P}(W_{\mathbf{L},\mathbf{K}})\rightarrow \mathfrak{t}$ their moment maps. If, with abuse of notation, we denote the general $Z\in W_{\mathbf{L},\mathbf{K}}$ as $Z=(Z_a)$, with $Z_a=
\begin{pmatrix}
z_{a,0}&\cdots&z_{a,k_a}
\end{pmatrix}^t\in \mathbb{C}^{k_a+1}$, we have 
\begin{eqnarray}
\label{eqn:moment map general LK}
\lefteqn{\Phi_{\mathbf{L},\mathbf{K}}([Z])}\\
&=&\frac{\imath}{\|Z\|^2•}\,\sum_{a=1}^r
\begin{pmatrix}
\|F_{k_a,1}(Z_a)\|^2+l_a\,\|Z_a\|^2& F_{k_a,2}(Z_a)^t\,\overline{F_{k_a,1}(Z_a)}\\
F_{k_a,1}(Z_a)^t\,\overline{F_{k_a,2}(Z_a)}&\|F_{k_a,2}(Z_a)\|^2+l_a\,\|Z_a\|^2
\end{pmatrix}.\nonumber
\end{eqnarray}

Let us first consider the case where
$\mathbf{K}=\mathbf{1}:=
\begin{pmatrix}
1&\cdots&1
\end{pmatrix}$, $\mathbf{L}=\mathbf{l}:=\begin{pmatrix}
l&\cdots&l
\end{pmatrix}$.
Thus $W_{\mathbf{l},\mathbf{1}}=
\det^{\otimes l}\otimes W_1^{\oplus r}$ is isomorphic to $(\mathbb{C}^2)^r$
as a complex vector space. Then the moment map
$\Phi_{\mathbf{l},\mathbf{1}}:
\mathbb{P}((\mathbb{C}^2)^r)
\rightarrow\mathfrak{g}$ is as follows.
Let us write the general element of
$(\mathbb{C}^2)^r$ as
$Z=\begin{pmatrix}
Z_1&\cdots&Z_r
\end{pmatrix}$ where $Z_j\in \mathbb{C}^2$. Then
\begin{equation}
\label{moment map Phi01rge2}
\Phi_{\mathbf{l},\mathbf{1}}
\big([Z_1:\cdots:Z_r]\big)
=\imath\,\left[{\sum}_{j=1}^{r}\frac{\|Z_j\|^2}{•\|Z\|^2}\,P_{Z_j}+l\,I_2\right],
\end{equation}
where $P_{\mathbf{0}}$ is the null endomorphism of $\mathbb{C}^2$, 
while for $Z\neq \mathbf{0}$ we let
$P_Z$ denote the orthogonal projector of $\mathbb{C}^2$ on
$\mathrm{span}(Z)$.

Let us set $\boldsymbol{\nu}_{1,j,l}:=\begin{pmatrix}
1-j+l&j+l
\end{pmatrix}$, $j=0,1$.

\begin{prop}
\label{prop:Phi1 rge2}
For any $r\ge 2$, the following holds:
\begin{enumerate}
\item $\Phi_{\mathbf{l},\mathbf{1}}\big(\mathbb{P}(W_1^{\oplus r})\big)
=\imath\,L_1+\imath\,l\,I_2$;
\item $\Psi_{\mathbf{l},\mathbf{1}}$ is transverse to $\imath\,\mathbb{R}_+\cdot
\boldsymbol{\nu}$ if and only if $\boldsymbol{\nu}\not\in  \mathbb{R}_+\cdot\boldsymbol{\nu}_{1,j,l}$
for $j=1,2$;
\item $\Phi_{\mathbf{l},\mathbf{1}}$ is transverse to 
$\mathcal{C}(\mathcal{O}_{\boldsymbol{\nu}})$ 
if and only if $\boldsymbol{\nu}\not\in  \mathbb{R}_+\cdot\boldsymbol{\nu}_{1,j,l}$
for $j=1,2$.
\end{enumerate}

\end{prop}

\begin{proof}
[Proof of Proposition \ref{prop:Phi1 rge2}]
Let us assume $l=0$; the general case is similar.
By (\ref{moment map Phi01rge2}), the image of $-\imath\,\Phi_{\mathbf{0},\mathbf{1}}$
consists of all convex linear combinations of $r\ge 2$ orthogonal projectors,
and is therefore contained in $L_1$. Conversely, any matrix in $L_1$ is a 
convex linear combination of two such projectors, and so the reverse implication holds.

To prove the second statement, 
consider $[Z]=[Z_1:\cdots:Z_r]$, with $\|Z\|=1$, such that every $Z_j$
is a scalar multiple of $\boldsymbol{\epsilon}_1:=\begin{pmatrix}
1&0
\end{pmatrix}$.
Then $\Phi_{\mathbf{0},\mathbf{1}}([Z])=\imath\,D_{\boldsymbol{\epsilon}_1}$, and
on the other hand $T$ does not acts locally freely on $S^{4r-1}$ at $Z$.
Hence $\Psi_{\mathbf{0},\mathbf{1}}$ is not transverse to $\mathbb{R}_+\,\imath\,\boldsymbol{\epsilon}_1$,
and $\Phi_{\mathbf{0},\mathbf{1}}$ is not transverse 
to $\mathcal{C}(\mathcal{O}_{\boldsymbol{\epsilon}_1})$.
The argument for $\boldsymbol{\epsilon}_2$ is similar.
If on the other hand the $Z_j$'s are neither all multiples of $\boldsymbol{\epsilon}_1$,
nor all multiples of $\boldsymbol{\epsilon}_2$, then $T$ acts locally freely at $Z$ and arguing as
in the proof of Theorem \ref{thm:trasnversality1} (or applying Lemma 
\ref{lem:trasversality criterion}),
one concludes that the same holds of
$G$. This proves the second and third statement.
\end{proof}

Let us return to (\ref{eqn:general representation}). For the sake of simplicity,
we shall consider a slightly restricted class of representation.

\begin{defn}
\label{defn:restricted class}
A representation $W_{\mathbf{L},\mathbf{K}}$ is \textit{generic}
if it satisfies the following property.
Suppose that for some
$l\in \mathbb{Z}$ the pair $(l,1)$ appears in the sequence 
$(l_1,k_1),\ldots,(l_r,k_r)$. Then there are $1\le a<b\le r$ such 
that $(1,r)=(l_a,k_a)=(l_b,k_b)$. 
\end{defn}

In other words, if 
$\det^{\otimes l}\otimes \mathbb{C}^2$ appears in the isotypical decomposition 
of $W_{\mathbf{L},\mathbf{K}}$, then it does so with multiplicity $\ge 2$.
For example, $W_1$ and $W_1^{\oplus 2}\oplus (\det^{-2}\otimes W_1)\oplus W_2$
are not generic, while $W_1^{\otimes 2}\oplus W_2$ is.

If $Z_a=\mathbf{0}$ for some $a$, then the $a$-th summand in
(\ref{eqn:moment map general LK}) vanishes; therefore, we may restrict the
sum to those $a$'s for which $Z_a\neq \mathbf{0}$, and this restricted sum will be indicated by
a prime. Hence

\begin{eqnarray}
\label{eqn:Phi LK a}
\lefteqn{\Phi_{\mathbf{L},\mathbf{K}}([Z])}\\
&=&\imath\,{\sum}_{a=1}^{\prime \,r}\,\frac{\|Z_a\|^2}{\|Z\|^2•}\cdot 
\frac{1}{•\|Z_a\|^2}\begin{pmatrix}
\|F_{k_a,1}(Z_a)\|^2+l_a\,\|Z_a\|^2& F_{k_a,2}(Z_a)^t\,\overline{F_{k,1}(Z_a)}
\nonumber\\
F_{k_a,1}(Z_a)^t\,\overline{F_{k_a,2}(Z_a)}&\|F_{k_a,2}(Z_a)\|^2+l_a\,\|Z_a\|^2
\end{pmatrix}\\
&=&{\sum}_{a=1}^{\prime \,r}\,\frac{\|Z_a\|^2}{\|Z\|^2•}\,
\Phi_{k_a,l_a}([Z_a])\nonumber.
\end{eqnarray}

\begin{prop}
\label{prop:Phi inviluppo convesso}
Assume that $W_{\mathbf{L},\mathbf{K}}$ is \textit{generic}. Then
$\Phi_{\mathbf{L},\mathbf{K}}\big(\mathbb{P}(W_{\mathbf{L},\mathbf{K}})\big)
\subset \mathfrak{g}$
is the convex hull of the union of the images $\Phi_{k_a,l_a}\left(\mathbb{P}^{k_a}\right)$.
\end{prop}

\begin{proof}
[Proof of Proposition \ref{prop:Phi inviluppo convesso}]
Let us denote by 
$H_{\mathbf{L},\mathbf{K}}\subset \mathfrak{g}$ the convex hull in
point. By (\ref{eqn:moment map general LK}), 
$\Phi_{\mathbf{L},\mathbf{K}}\big(\mathbb{P}(W_{\mathbf{L},\mathbf{K}})\big)
\subseteq H_{\mathbf{L},\mathbf{K}}$. Conversely, suppose $\alpha\in H_{\mathbf{L},\mathbf{K}}$. Then there exist 
$\lambda_a\ge 0$, $a=1,\ldots,r$, such that $\sum_a'\lambda_a=1$, and for each 
$a$ with $\lambda_a>0$ 
there exists $V_a\in \mathbb{C}^{k_a+1}$ of unit norm, such that
$$
\alpha={\sum}_{a=1}^{\prime \,r}\,\lambda_a\,\Phi_{k_a,l_a}([V_a]).
$$
Let us set $Z_a:=\sqrt{•\lambda_a}\,V_a$ if $\lambda_a>0$,
$Z_a=0\in \mathbb{C}^{k_a+1}$ if $\lambda_a=0$, and 
$Z:=\left(Z_a\right)\in \mathbb{C}^{|\mathbf{K}|+r}$. Then 
$\|Z\|=1$ and 
$\Phi_{\mathbf{L},\mathbf{K}}([Z])=\alpha$
by (\ref{eqn:Phi LK a}), hence 
$\alpha\in 
\Phi_{\mathbf{L},\mathbf{K}}\big(\mathbb{P}(W_{\mathbf{L},\mathbf{K}})\big)$.

\end{proof}

We can describe $\Psi_{\mathbf{L},\mathbf{K}}$ in a similar manner, and
deduce the following:
\begin{prop}
\label{prop:Psi inviluppo convesso}
$\Psi_{\mathbf{L},\mathbf{K}}\big(\mathbb{P}(W_{\mathbf{L},\mathbf{K}})\big)
\subset \mathfrak{t}$
is the convex hull of the union of the images $\Psi_{k_a,l_a}\left(\mathbb{P}^{k_a}\right)$.
\end{prop}

On the other hand, $-\imath\,\Psi_{k_a,l_a}\left(\mathbb{P}^{k_a}\right)$
is the segment joining 
$\begin{pmatrix}
k_a+l_a&l_a
\end{pmatrix}^t$ and $\begin{pmatrix}
l_a&k_a+l_a
\end{pmatrix}^t$ for each $a$. Therefore we conclude the following
(which might be also obtained by the Convexity Theorem):

\begin{cor}
\label{cor:PHIKL convex hull}
$-\imath\,\Psi_{\mathbf{L},\mathbf{K}}\big(\mathbb{P}(W_{\mathbf{L},\mathbf{K}})\big)
\subset \mathbb{R}^2
$
is the convex hull of the collection of the points  
 $\begin{pmatrix}
k_a+l_a&l_a
\end{pmatrix}^t$ and $\begin{pmatrix}
l_a&k_a+l_a
\end{pmatrix}^t$, $a=1,\ldots,r$, or equivalently of the segments
$J_{k_a,l_a}$.
\end{cor}

We have the following analogue of Corollaries 
\ref{cor:intersetion=projection k} and
\ref{cor:intersetion=projection kl}:

\begin{cor}
\label{cor:intersetion=projection general}
If $W_{\mathbf{L},\mathbf{K}}$ is generic, then
$\Psi_{\mathbf{L},\mathbf{K}}\big(\mathbb{P}(W_{\mathbf{L},\mathbf{K}})\big)=\Phi_{\mathbf{L},\mathbf{K}}\big(\mathbb{P}(W_{\mathbf{L},\mathbf{K}})\big)
\cap \mathfrak{t}$.
\end{cor}

\begin{prop}
\label{prop:equivalent non zero phi}
Assume that $W_{\mathbf{L},\mathbf{K}}$ is generic.
Then the following conditions are equivalent:
\begin{enumerate}
\item $\mathbf{0}\not\in \Psi_{\mathbf{L},\mathbf{K}}\big(\mathbb{P}(W_{\mathbf{L},\mathbf{K}})\big)$;
\item $\mathbf{0}\not\in \Phi_{\mathbf{L},\mathbf{K}}\big(\mathbb{P}(W_{\mathbf{L},\mathbf{K}})\big)$;
\item either $k_a+2\,l_a>0$ for all $a=1,\ldots,r$, or $k_a+2\,l_a<0$ for all 
$a=1,\ldots,r$.
\end{enumerate}

\end{prop}

\begin{proof}
By Corollary \ref{cor:intersetion=projection general},
1) and 2) are equivalent. Suppose that 2) holds.
By (\ref{eqn:immaginePhi'}), we have 
$\Phi_{k_a,l_a}\left(\mathbb{P}^{k_a}\right)=\imath\,L_{k_a}+\imath\,l_a\,I_2$
for every $a$; if $k_a+2\,l_a=0$ for some $a$, then $l_a\le 0$ and so 
$$
(0)=\imath\,\begin{pmatrix}
-l_a&0\\
0&-l_a
\end{pmatrix}+\imath\,l_a\,I_2\in \Phi_{k_a,l_a}\left(\mathbb{P}^{k_a}\right).
$$
Hence assuming 2) we need to have 
$k_a+2\,l_a\neq 0$ for every $a=1,\ldots,r$.
Suppose that $k_a+2\,l_a>0$ and $k_b+2\,l_b<0$ for some 
$a,\,b=1,\ldots,r$. Then 
$$
\frac{1}{•2}\,(k_a+2\,l_a)\,I_2=\frac{k_a}{2•}\,I_2+l_a\,I_2\in \Phi_{k_a,l_a}\left(\mathbb{P}^{k_a}\right),
$$
and similarly 
$$
\frac{\imath}{•2}\,(k_b+2\,l_b)\,I_2=\imath\,\frac{k_b}{2•}\,I_2+\imath\,l_b\,I_2\in \Phi_{k_b,l_b}\left(\mathbb{P}^{k_b}\right).
$$
Hence by the previous dicussion the segment joining these two matrices is contained in $\Phi_{\mathbf{L},\mathbf{K}}\big(\mathbb{P}(W_{\mathbf{L},\mathbf{K}})\big)$, and it is obvious that 
it meets the origin, absurd. Hence 2) implies 3).

Suppose that 3) holds, say with $>0$. Then for every $a=1,\ldots,r$ and every
$\alpha\in \Psi_{l_a,k_a}\left(\mathbb{P}^{k_a}\right)$ we have 
$-\imath\,\mathrm{trace}(\alpha)= k_a+2\,l_a>0$. 
Since the convex linear combination of matrices with positive trace
has positive trace, 1) also holds by Proposition \ref{prop:Psi inviluppo convesso}.

\end{proof}

\begin{cor}
\label{cor:non zero trace moment}
Assume that $W_{\mathbf{L},\mathbf{K}}$ is generic.
Then
$\mathbf{0}\not\in \Phi_{\mathbf{L},\mathbf{K}}\big(\mathbb{P}(W_{\mathbf{L},\mathbf{K}})\big)$ if and only if 
$\Phi_{\mathbf{L},\mathbf{K}}\big(\mathbb{P}(W_{\mathbf{L},\mathbf{K}})\big)\subset \mathfrak{g}$ is contained in one of the
half-spaces defined by the hyperplane 
$\mathfrak{su}(2)=\ker (\mathrm{trace})\subset \mathfrak{g}$. In particular, if
$\mathbf{0}\not\in \Phi_{\mathbf{L},\mathbf{K}}\big(\mathbb{P}(W_{\mathbf{L},\mathbf{K}})\big)$ and 
$\Phi_{\mathbf{L},\mathbf{K}}\big(\mathbb{P}(W_{\mathbf{L},\mathbf{K}})\big)\cap \mathbb{R}_+\cdot \boldsymbol{\nu}\neq \emptyset$, then $\nu_1+\nu_2\neq 0$. 
\end{cor}

\begin{defn}
\label{defn:uniform rep}
The representation $W_{\mathbf{L},\mathbf{K}}$ will be called \textit{uniform}
if it is generic and $k_a+2\,l_a=k_b+2\,l_b$ for all $a,\,b=1,\ldots,r$.
\end{defn}

The proof of the following Lemma is left to the reader.

\begin{lem}
\label{lem:uniform rep}
The following conditions are equivalent:
\begin{enumerate}
\item $W_{\mathbf{L},\mathbf{K}}$ is uniform;
\item $\phi_{\mathbf{L},\mathbf{K}}$ (equivalently, $\psi_{\mathbf{L},\mathbf{K}}$)
is trivial on $Z(G)$ (the center of $G$).
\end{enumerate}

\end{lem}

Let us now assume that the equivalent conditions in Proposition 
\ref{prop:equivalent non zero phi} are satisfied, and
consider transversality. Let us denote by 
$X_{\mathbf{K}}\subset \mathbb{C}^{|\mathbf{K}|+r}$
the unit sphere, by $\pi_{\mathbf{K}}:X_{\mathbf{K}}\rightarrow
\mathbb{P}^{\mathbf{K}+r-1}$ the Hopf map, and set
$\tilde{\Phi}_{\mathbf{L},\mathbf{K}}=\Phi_{\mathbf{L},\mathbf{K}}\circ \pi_{\mathbf{K}}:
X_{\mathbf{K}}\rightarrow \mathfrak{g}$. 
Also, let $\tilde{\phi}_{\mathbf{L},\mathbf{K}}$ and 
$\tilde{\psi}_{\mathbf{L},\mathbf{K}}$ denote,
respectively, the actions of $G$ and $T$ on $X_{\mathbf{K}}$
by restriction of $\tilde{\phi}_{\mathbf{L},\mathbf{K}}$.
These are liftings of the actions $\phi_{\mathbf{L},\mathbf{K}}$
and $\psi_{\mathbf{L},\mathbf{K}}$ on
$\mathbb{P}(W_{\mathbf{L},\mathbf{K}})$

Let us fix $Z\in X_{\mathbf{K}}$,
and denote by $\mathcal{O}^Z\subset \mathfrak{g}$ the orbit through $\tilde{\Phi}_{\mathbf{L},\mathbf{K}}(Z)$. Perhaps after replacing $Z$ with $(\tilde{\phi}_{\mathbf{L},\mathbf{K}})_g(Z)$ for some $g\in G$, without changing $\mathcal{O}^Z$
we may as well assume
that $\tilde{\Phi}_{\mathbf{L},\mathbf{K}}(Z)\in \mathfrak{t}$.

Suppose that only one component of $Z$ in non-zero, say
$z_{aj}$ for some $a\in \{1,\ldots,r\}$ and $j\in \{0,\ldots,k_a\}$.
Then, as in the case $r=1$, one sees that there is a 1-dimensional torus fixing
$Z$; therefore, neither is $\Phi_{\mathbf{L},\mathbf{K}}$ transverse to 
$\mathcal{C}\left(\mathcal{O}^Z\right)$, nor is 
$\Psi_{\mathbf{L},\mathbf{K}}$ transverse to
$\mathbb{R}_+\,\Psi_{\mathbf{L},\mathbf{K}}(Z)$.
In this case, in view of (\ref{eqn:Phi LK a}) and (\ref{eqn:tildePhikl})
we have
$$
\tilde{\Phi}_{\mathbf{L},\mathbf{K}}(Z)=
\Phi_{k_a,l_a}([Z_a])=
\imath\,
\begin{pmatrix}
k_a-j+l_a & 0\\
0& j+l_a
\end{pmatrix}.
$$

Hence, if we set 
\begin{equation}
\label{eqn:defn di nuaj}
\boldsymbol{\nu}_{k_a,j,l_a}:=\begin{pmatrix}
k_a-j+l_a &j+l_a
\end{pmatrix}^t\qquad (a=1,\ldots,r,\quad j=0,\ldots,k_a),
\end{equation}
we conclude that $\Phi_{\mathbf{L},\mathbf{K}}$ is not transverse to
$\mathcal{C}(\mathcal{O}_{\boldsymbol{\nu}_{k_a,j,l_a}})$ and that
$\Psi_{\mathbf{L},\mathbf{K}}$ is not transverse to
$\mathbb{R}_+\cdot\boldsymbol{\nu}_{k_a,j,l_a}$ for every $a,\,j$.

If, on the other hand, there exist $a\in \{1,\ldots,r\}$
and $j,h\in \{0,\ldots,k_a\}$ with 
$j\neq h$ and $z_{aj}\cdot z_{ah}\neq 0$, then 
the arguments used in the proof of Theorems \ref{thm:trasnversality1},
\ref{thm:transversality2} imply that 
both $\tilde{•\psi}_{\mathbf{L},\mathbf{K}}$ and 
$\tilde{•\phi}_{\mathbf{L},\mathbf{K}}$ are locally free at $Z$.

Thus we reduced to the case where for each $a=1,\ldots,r$ at most one component
of $Z_a$ is non-zero, and $Z_a\neq \mathbf{0}$ for at least two distinct 
values of $a$. We shall make this assumption in the following.

So there exist $a,\,b\in \{1,\ldots,r\}$, $a\neq b$ and $j_a\in 
\{0,\ldots,k_a\}$, $j_b\in \{0,\ldots,k_b\}$ such that 
$z_{a,j_a}\cdot z_{b,j_b}\neq 0$, and 
furthermore $z_{a,j}=0$ if $j\neq j_a$, $z_{b,j}=0$
if $j\neq j_b$.

Consider, as before, a diagonal matrix $D\in T$, 
with diagonal entries $e^{\imath\vartheta_i}$, $i=1,2$, and suppose that
$D$ fixes $Z$. Also, let us assume that $D$ is in a small neighborhood of
$I_2$, so that without loss $\vartheta_j\sim 0$. Then 
the condition $(\tilde{•\phi}_{\mathbf{L},\mathbf{K}})_D(Z)=Z$
implies that
$e^{\imath\,[l_a\,(\vartheta_1+\vartheta_2)+(k_a-j_a)\,\vartheta_1+j_a\,\vartheta_2]}
\,z_{a,j_a}=z_{a,j_a}$ and $e^{\imath\,[l_b\,(\vartheta_1+\vartheta_2)+(k_b-j_b)\,\vartheta_1+j_b\,\vartheta_2]}
\,z_{b,j_b}=z_{b,j_b}$. Since $\vartheta_j\sim 0$, 
this forces
$$
(l_a+k_a-j_a)\,\vartheta_1+(l_a+j_a)\,\vartheta_2=
(l_b+k_b-j_b)\,\vartheta_1+(l_b+j_b)\,\vartheta_2=0.
$$
This system has non-trivial solutions if and only if the vectors 
$\boldsymbol{\nu}_{k_a,j_a,l_a}
$ and $\boldsymbol{\nu}_{k_b,,j_b,l_b}$
are linearly dependent
(see (\ref{eqn:defn di nuaj})); if this is 
the case, then
$\Phi_{k_a,l_a}([Z_a])$ and 
$\Phi_{k_b,l_b}([Z_b])$ are both
scalar multiples of the 
diagonal matrix 
$\imath\,D_{\boldsymbol{\nu}_{k_a,j_a,l_a}}$.

Hence we have the following alternatives. 

Let $I\subseteq \{1,\ldots,r\}$ be the 
non-empty subset of those $a$'s 
such that $Z_a\neq \mathbf{0}$.
If the vectors $\boldsymbol{\nu}_{k_a,j_a,l_a}
$, $a\in I$, are all pairwise linearly dependent, then 
$\tilde{\psi}_{\mathbf{L},\mathbf{K}}$
is not locally free at $Z$, and therefore neither is $\tilde{\phi}_{\mathbf{L},\mathbf{K}}$. Hence,
$\Phi_{\mathbf{L},\mathbf{K}}$ is not transverse to
$\mathcal{C}\left(\mathcal{O}^Z\right)$ at $Z$, and similarly $\Psi_{\mathbf{L},\mathbf{K}}$ is not transverse to
 $\mathbb{R}_+\cdot \Psi_{\mathbf{L},\mathbf{K}}(Z)$ at $Z$.
Furthermore, in this case we also obtain that 
$\Phi_{\mathbf{L},\mathbf{K}}([Z])$
is a multiple of $\imath\,D_{\boldsymbol{\nu}_{k_a,j_a,l_a}}$, and so $\Psi_{\mathbf{L},\mathbf{K}}([Z])$
is a multiple of $\imath\,\boldsymbol{\nu}_{k_a,j_a,l_a}$.

Suppose, on the other hand, that there exist
$a,b\in I$ such that 
$\boldsymbol{\nu}_{k_a,j_a,l_a}
\wedge\boldsymbol{\nu}_{k_b,j_b,l_b}\neq \mathbf{0}$. Then 
$\tilde{\psi}_{\mathbf{L},\mathbf{K}}$ 
is locally free at $Z$. Since we are assuming that $\Phi_{\mathbf{L},\mathbf{K}}([Z])$ is diagonal and non-zero, we can apply the argument used in the proof of Theorem \ref{thm:trasnversality1}, following (\ref{eqn:defn of xieta}), to obtain 
the stronger statement
that 
$\tilde{\phi}_{\mathbf{L},\mathbf{K}}$ 
is also locally free at $Z$,
and so 
$\Phi_{\mathbf{L},\mathbf{K}}$ is transverse to $\mathcal{C}\left(\mathcal{O}^Z\right)$ 
at $Z$. 

The outcome of the previous discussion is the following statement.
Recall that $\boldsymbol{\nu}_{a,j}$ was
defined in
(\ref{eqn:defn di nuaj}).

\begin{thm}
\label{thm:transversality general}
Suppose $\nu_1\neq \nu_2$ and that
the equivalent conditions in
Proposition 
\ref{prop:equivalent non zero phi} are satisfied. Then the following conditions are equivalent:

\begin{enumerate}
\item $\Phi_{\mathbf{L},\mathbf{K}}$ is not transverse to $\mathcal{C}(\mathcal{O}_{\boldsymbol{\nu}})$;
\item $\Psi_{\mathbf{L},\mathbf{K}}$ is not transverse to $\mathbb{R}_+\,\imath\,\boldsymbol{\nu}$;
\item there exist $a\in \{1,\ldots,r\}$ and
$j\in \{0,\ldots,k_a\}$, such that
$\boldsymbol{\nu}=\boldsymbol{\nu}_{k_a,j,l_a}$.
\end{enumerate}
\end{thm}

If $M\subseteq \mathbb{P}(W_{\mathbf{L},\mathbf{K}})$ is a projective submanifold, then the restriction to $M$ of the
Fubini-Study form is a K\"{a}hler form $\omega$ on $M$. If $M$ is $G$-invariant, the induced action of $G$ on $M$ is Hamiltonian with respect to $2\,\omega$, with moment map
$\Phi_M:=\left.\Phi_{\mathbf{L},\mathbf{K}}\right|_M:M\rightarrow \mathfrak{g}$.
Similar considerations apply to the action of $T$ on $M$, which is Hamiltonian with respect to
$2\,\omega$, with moment map $\Psi_M:=\left.\Psi_{\mathbf{L},\mathbf{K}}\right|_M:M\rightarrow\mathfrak{t}$.

For $\boldsymbol{\nu}=\begin{pmatrix}
\nu_1&\nu_2
\end{pmatrix}^t$ with $\nu_j\ge 0$ and $\boldsymbol{\nu}\neq \mathbf{0}$, 
let us denote
by $\mathbb{P}_{\boldsymbol{\nu}}\subseteq	 \mathbb{P}(W_{\mathbf{L},\mathbf{K}})$ the locus 
of those $[Z]=[Z_1:\ldots:Z_r]$, where $Z_a=(z_{aj})\in \mathbb{C}^{k_a+1}$, 
such that $z_{aj}=0$ if $\begin{pmatrix}
k_a-j+l_a&j+l_a
\end{pmatrix}^t$ is not a (positive) multiple of 
$\begin{pmatrix}
\nu_1&\nu_2
\end{pmatrix}^t$. Then $\mathbb{P}_{\boldsymbol{\nu}}=\emptyset$ unless 
$\boldsymbol{\nu}=\boldsymbol{\nu}_{k_a,j,l_a}$ for some $a=1,\ldots,r$ 
and $j=0,\ldots,k_a$, and each $\mathbb{P}_{\boldsymbol{\nu}_{k_a,j,l_a}}$
is a projective subspace. 
For any $(a,j)$ and $(b,j')$, either $\mathbb{P}_{\boldsymbol{\nu}_{k_a,j,l_a}}=
\mathbb{P}_{\boldsymbol{\nu}_{k_b,j',l_{a'}}}$, or else 
$\mathbb{P}_{\boldsymbol{\nu}_{k_a,j,l_a}}\cap 
\mathbb{P}_{\boldsymbol{\nu}_{k_b,j',l_{a'}}}=\emptyset$;
also, the inverse image in $X_{\boldsymbol{K},\boldsymbol{L}}$
of $\bigcup_{a,j}\mathbb{P}_{\boldsymbol{\nu}_{k_a,j,l_a}}$
is the locus over which $\Psi_{\boldsymbol{K},\boldsymbol{L}}$
is not locally free.

\begin{thm}
\label{thm:general manifold}
In the situation of Theorem \ref{thm:transversality general}, suppose that 
$M\subseteq \mathbb{P}(W_{\mathbf{L},\mathbf{K}})$ is a $G$-invariant projective manifold. 
Consider $\boldsymbol{\nu}\in \mathbb{N}^2\setminus \{\mathbf{0}\}$.
Then the following conditions are equivalent:
\begin{description}
\item[1)] $\Psi_M$ is not transverse to $\mathbb{R}_+\cdot \imath\,\boldsymbol{\nu}$;
\item[2)] $\boldsymbol{\nu}=\boldsymbol{\nu}_{k_a,j,l_a}$ for some $(a,j)$, and 
$M\cap \mathbb{P}_{\boldsymbol{\nu}_{k_a,j,l_a}}\neq \emptyset$.
\end{description}

If, in addition, $\nu_1\neq \nu_2$, then 1) and 2) are equivalent 
to
\begin{description}
\item[3)]  $\Phi_M$ is not transverse to $\mathcal{C}(\mathcal{O}_{\boldsymbol{\nu}})$.
\end{description}

\end{thm}

\begin{proof}
[Proof of Theorem \ref{thm:general manifold}]
Let $X_M\subseteq X$ be the inverse image of $M$ in $X_{\mathbf{L},\mathbf{K}}$; thus, $X_M$ is the circle bundle of the induced polarization. Then
$(X_M)_{\boldsymbol{\nu}}^G=(X_{\mathbf{L},\mathbf{K}})^G_{\boldsymbol{\nu}}\cap X_M$ \textit{etc}. 
Let us denote by $\tilde{\phi}_M$ and $\tilde{\psi}_M$, respectively, 
the restrictions of $\tilde{\phi}_{\mathbf{L},\mathbf{K}}$ and 
$\tilde{\psi}_{\mathbf{L},\mathbf{K}}$ to $X_M$.

Let us prove the equivalence of 1) and 2). 

As recalled above,
$\Psi_M$ is not transverse to $\mathbb{R}_+\cdot \imath\,\boldsymbol{\nu}$
if and only if there exists $Z\in (X_M)^T_{\boldsymbol{\nu}}$ such that 
$\tilde{\psi}_M$ is not locally free at $Z$, that is, such that $\tilde{\psi}_{\mathbf{L},\mathbf{K}}$
is not locally free at $Z$. On the other hand, the previous discussion shows that
$\tilde{\psi}_{\mathbf{L},\mathbf{K}}$ is not locally free at $Z$ if and only if
$[Z]\in \mathbb{P}_{\boldsymbol{\nu}_{a,j}}$ for some $(a,j)$, and that
if this happens then $\Psi_M([Z])=\Psi_{\mathbf{K},\mathbf{L}}([Z])$
is a positive multiple of $\imath\,\boldsymbol{\nu}_{a,j}$.

Let us assume that $\nu_1\neq \nu_2$, and prove the equivalence with 3).

Suppose that 2) holds, and suppose $Z\in X_M$, 
$[Z]\in M\cap \mathbb{P}_{a,j}$. 
Then $\tilde{\psi}_M$ is not locally free at $Z$, and therefore
\textit{a fortiori}
neither is $\tilde{\phi}_M$.
Furthermore, 
by the previous discussion $\Phi_M([Z])$ is a positive multiple of
$\imath\,D_{\boldsymbol{\nu}_{a,j}}$, so
$Z\in (X_M)^G_{\boldsymbol{\nu}_{a,j}}$. 
Hence 3) holds.

Conversely, suppose that 3) holds. Then
there exists $Z\in (X_M)^G_{\mathcal{O}_{\boldsymbol{\nu}}}$ such 
that $\tilde{\phi}_M$ is not
locally free at $Z$; perhaps after replacing $Z$ in its orbit, we may assume
without loss that $\tilde{\Phi}_M(Z)$ is diagonal, that is, 
$Z\in (X_M)^G_{\boldsymbol{\nu}}=(X_M)^T_{\boldsymbol{\nu}}\cap 
(X_M)^G_{\mathcal{O}}$. If $\tilde{\psi}_M$ was locally free at $Z$, then an argument
in the proof of Theorem \ref{thm:trasnversality1}
(see (\ref{eqn:derivata Phi}) and (\ref{eqn:xieta Phi-derivative}))
would imply that $\tilde{\phi}_M$ is itself locally free at $Z$, absurd.
Hence $\tilde{\psi}$ is not locally free at $Z$, and therefore
$[Z]\in \mathbb{P}_{a,j}$ for some $a,j$, and $\Phi_M([Z])$ is ap positive multiple of $\imath\,D_{\boldsymbol{\nu}_{a,j}}$. Hence 1) and 2) hold.

\end{proof}

\section{$\overline{M}^T_{\boldsymbol{\nu}}$}

We shall assume in this section that
$\mathbf{0}\not\in \Psi (M)$, and that 
both $\Psi$ and $\Phi$ are transverse to $\mathbb{R}_+\cdot \imath\,\boldsymbol{\nu}$,
where $\nu_1>\nu_2$.
Then
$M^T_{\boldsymbol{\nu}}\subset M$ is a
smooth compact connected $T$-invariant hypersurface; 
furthermore, $M^G_{\boldsymbol{\nu}}:=\Phi^{-1}(\mathbb{R}_+\cdot \imath\,\boldsymbol{\nu})\subset M^T_{\boldsymbol{\nu}}$ is a smooth, compact and connected 
$T$-invariant submanifold 
of real codimension two (three in $M$) \cite{gp}. 
In \S \ref{sctn:kahler structure}, $M$ is not assumed to
be projective.

\subsection{The K\"{a}hler structure of $\overline{M}^T_{\boldsymbol{\nu}}$}
\label{sctn:kahler structure}


The $1$-parameter subgroup 
\begin{equation}
\label{eqn:T1nu parametrized}
T^1_{\boldsymbol{\nu}_\perp}:=\left\{\,\kappa_{\boldsymbol{\nu}}\left(
e^{\imath\vartheta}\right)\,:\,e^{\imath\,\vartheta}\in S^1\right\},
\quad
\kappa_{\boldsymbol{\nu}}\left(
e^{\imath\vartheta}\right):=
\begin{pmatrix}
e^{-\imath\,\nu_2\,\vartheta}&0\\
0&e^{\imath\,\nu_1\,\vartheta}
\end{pmatrix}
\end{equation}
acts locally freely on $M^T_{\boldsymbol{\nu}}$; its orbits are the
leaves of the null foliation of $M^T_{\boldsymbol{\nu}}$. 
If $\nu_1$ and $\nu_2$ are coprime, as we may assume
without loss, 
$\kappa_{\boldsymbol{\nu}}:S^1\rightarrow T^1_{\boldsymbol{\nu}_\perp}$ in (\ref{eqn:T1nu parametrized}) 
is a Lie group isomorphism. 

Let us set
\begin{equation}
\label{eqn:MTnu MGnu}
\overline{M}^T_{\boldsymbol{\nu}}:=M^T_{\boldsymbol{\nu}}/T^1_{\boldsymbol{\nu}_\perp},
\qquad\overline{M}^G_{\boldsymbol{\nu}}:=
M^G_{\boldsymbol{\nu}}/T^1_{\boldsymbol{\nu}_\perp}\subset \overline{M}^T_{\boldsymbol{\nu}}.
\end{equation}
Then $\overline{M}^T_{\boldsymbol{\nu}}$
is an orbifold of (real) dimension $2\,(d-1)$, 
and $\overline{M}^G_{\boldsymbol{\nu}}\subset \overline{M}^T_{\boldsymbol{\nu}}$ is a suborbifold of real codimension two,
meaning that the topological embedding $\overline{M}^G_{\boldsymbol{\nu}}\subset \overline{M}^T_{\boldsymbol{\nu}}$ can be lifted to an embedding of local slices.
We shall let $q_{\boldsymbol{\nu}}:M^T_{\boldsymbol{\nu}}\rightarrow\overline{M}^T_{\boldsymbol{\nu}}$
denote the projection.

\begin{defn}
\label{defn:psi nu perp}
$\psi_{\boldsymbol{\nu}_\perp}$ is the action of 
$T^1_{\boldsymbol{\nu}_\perp}$ on $M$ given by restriction of $\psi$.
\end{defn}
 
By means of $\kappa_{\boldsymbol{\nu}}$, we shall view 
$\psi_{\boldsymbol{\nu}_\perp}$ as a Hamiltonian $S^1$-action,
with moment map
$\Psi_{\boldsymbol{\nu}_\perp}:=\langle\Psi,\boldsymbol{\nu}_\perp\rangle$.
The proof of the following is left to the reader:

\begin{lem}
\label{lem:0 regular value nuperp} 
Given that $\Psi$ is transverse to $\mathbb{R}_+\cdot \imath\,\boldsymbol{\nu}$, 
$0$ is a regular value of $\Psi_{\boldsymbol{\nu}_\perp}$, and
$M^T_{\boldsymbol{\nu}}={\Psi_{\boldsymbol{\nu}_\perp}}^{-1}(0)$.
\end{lem}

As an orbifold, $\overline{M}^T_{\boldsymbol{\nu}}$ coincides with the symplectic quotient 
(symplectic reduction at $\mathbf{0}$)
$M//T^1_{\boldsymbol{\nu}_\perp}$. Hence it inherits a reduced
K\"{a}hler orbifold structure $\left(\overline{M}^T_{\boldsymbol{\nu}},\,J_{\overline{M}^T_{\boldsymbol{\nu}}},\,\omega_{\overline{M}^T_{\boldsymbol{\nu}}}\right) $.

As mentioned in the introduction, $\overline{M}^G_{\boldsymbol{\nu}}$ may also be viewed as a symplectic quotient, namely $\overline{M}^G_{\boldsymbol{\nu}}=Y//T^1_{\boldsymbol{\nu}_\perp}$, 
where $Y\subset M$ is the \lq symplectic cross section\rq\, discussed in \cite{gs cp1}. 
Hence $\overline{M}^G_{\boldsymbol{\nu}}$ also carries a symplectic orbifold structure $(\overline{M}^G_{\boldsymbol{\nu}},\,\omega_{\overline{M}^G_{\boldsymbol{\nu}}})$.
Since both $\omega_{\overline{M}^G_{\boldsymbol{\nu}}}$
and $\omega_{\overline{M}^T_{\boldsymbol{\nu}}}$ are both induced from $\omega$, 
$(\overline{M}^G_{\boldsymbol{\nu}},\,\omega_{\overline{M}^G_{\boldsymbol{\nu}}})$ is a symplectic suborbifold
of 
$(\overline{M}^T_{\boldsymbol{\nu}},\,\omega_{\overline{M}^T_{\boldsymbol{\nu}}})$. 

The $T$-invariant direct sum decomposition $\mathfrak{g}=\mathfrak{t}\oplus \mathfrak{a}$
determines a splitting $\Phi=\Psi \oplus \Upsilon':M\rightarrow \mathfrak{g}$,
where both $\Psi:M\rightarrow \mathfrak{t}$ and 
$\Upsilon':M\rightarrow \mathfrak{a}$
are $T$-equivariant (notation is as in (\ref{eqn:defn of xieta})).
By restriction we obtain a $T$-equivariant smooth map
\begin{equation}
\label{eqn:Upsilon}
\Upsilon:=\left.\Upsilon'\right|_{•M^T_{\boldsymbol{\nu}}}:M^T_{\boldsymbol{\nu}}\rightarrow \mathfrak{a}.
\end{equation}
Since
\begin{equation}
\label{eqn:action Ta}
\begin{pmatrix}
e^{\imath\,\vartheta_1}&0\\
0&e^{\imath\,\vartheta_2}
\end{pmatrix}  \,\imath\,
\begin{pmatrix}
a&z\\
\overline{z}&b
\end{pmatrix}\,
\begin{pmatrix}
e^{-\imath\,\vartheta_1}&0\\
0&e^{-\imath\,\vartheta_2}
\end{pmatrix}  
=
\imath\,\begin{pmatrix}
a&e^{\imath\,(\vartheta_1-\vartheta_2)}\,z \\
e^{-\imath\,(\vartheta_1-\vartheta_2)}\,\overline{z}&b
\end{pmatrix},
\end{equation}
identifying $\mathfrak{a}\cong \mathbb{C}$ by the parameter
$z$ in (\ref{eqn:action Ta}), we may interpret $\Upsilon$ as a 
map $M^T_{\boldsymbol{\nu}}\rightarrow \mathbb{C}$ with the equivariance property
\begin{equation}
\label{eqn:covariance upsilon}
\Upsilon\circ \psi_{D(\vartheta_1,\vartheta_2)^{-1}}
=e^{-\imath\,(\vartheta_1-\vartheta_2)}\,
\Upsilon,
\end{equation}
where $D(\vartheta_1,\vartheta_2)\in T$ is the diagonal matrix 
with entries $e^{\imath\,\vartheta_j}$.

By Theorem 1.2 of \cite{gp},
$M^T_{\boldsymbol{\nu}}\cap M^G_{\mathcal{O}}=
M^G_{\boldsymbol{\nu}}$, and the intersection is tangential, that is,
$T_mM^T_{\boldsymbol{\nu}}=T_mM^G_{\mathcal{O}}\subset T_mM$ if $m\in M^G_{\boldsymbol{\nu}}$.
Since $M^G_{\mathcal{O}}$ is $G$-invariant, for any $\beta\in \mathfrak{g}$
the vector field $\beta_M\in \mathfrak{X}(M)$ induced by $\beta$ is tangent to
$M^G_{\mathcal{O}}$. 
Hence, if $m\in M^G_{\boldsymbol{\nu}}$ then
$\beta_M(m)\in T_mM^T_{\boldsymbol{\nu}}$.
Therefore,
$\mathfrak{a}_M(m)\subset T_mM^T_{\boldsymbol{\nu}}$ for any $m\in M^G_{\boldsymbol{\nu}}$. 
The argument used for (\ref{eqn:xieta Phi-derivative}), and the remark that 
$M^G_{\boldsymbol{\nu}}=\Upsilon^{-1}(0)$,
imply the following.

\begin{lem}
\label{lem:Upsilon}
Under the previous assumptions, we have:
\begin{enumerate}
\item $\mathrm{d}_m\Upsilon\big(\mathfrak{a}_M(m)\big)=\mathfrak{a}$,
$\forall \,m\in M^T_{\boldsymbol{\nu}}$;
\item $0$ is a regular value of
$\Upsilon$;
\item we have a $T$-equivariant direct sum decomposition
\begin{equation}
\label{eqn:direct sum splitting}
T_mM^T_{\boldsymbol{\nu}}=T_mM^G_{\boldsymbol{\nu}}\oplus \mathfrak{a}_M(m),\quad
\forall\,m\in M^G_{\boldsymbol{\nu}}.
\end{equation}
\end{enumerate}

\end{lem}

\begin{lem}
\label{lem:symplectic direct sum}
The summands on the right hand side of (\ref{eqn:direct sum splitting}) are symplectically orthogonal.
\end{lem}

\begin{proof}
[Proof of Lemma \ref{lem:symplectic direct sum}]
Let us consider the Hamiltonian functions $\Phi^\eta :=\langle\Phi,\eta\rangle$
and $\Phi^\xi:=
\langle \Phi,\xi\rangle$.
Explicitly, if 
$$
\Phi =\imath\,\begin{pmatrix}
a&z\\
\overline{z}&b
\end{pmatrix},
$$
where $a,\,b:M\rightarrow \mathbb{C}$ and $z:M\rightarrow \mathbb{C}$ are 
$\mathcal{C}^\infty$, then 
$
\Phi^\eta=-2\,\Im (z)$, $\Phi^\xi = 2\,\Re (z)$.

By definition of $M^G_{\boldsymbol{\nu}}$, $z$ vanishes identically on $M^G_{\boldsymbol{\nu}}$;
therefore, for any $(m,v)\in M^G_{\boldsymbol{\nu}}$ we have
$$
0=\mathrm{d}_m\Phi^\eta \big(  v  \big)=\omega_m\big(\eta_M(m),v\big),
$$
and similarly for $\xi$.
\end{proof}

\begin{cor}
\label{cor:aM symplectic}
$\mathfrak{a}_M(m)\subseteq T_mM$ is a symplectic vector subspace, 
$\forall\,m\in M^G_{\boldsymbol{\nu}}$.

\end{cor}

\begin{proof}
[Proof of Corollary \ref{cor:aM symplectic}]
This follows immediately from Lemma \ref{lem:symplectic direct sum}.
Alternatively, we  need to show that
$\omega_m\big(\eta_M(m),\xi_M(m)\big)\neq 0$. For
$m\in M^G_{\boldsymbol{\nu}}$, we have
$\Phi (m)=\imath\,\lambda (m)\,D_{\boldsymbol{\nu}}$ where
$\lambda (m)>0$.
Arguing as for (\ref{eqn:xieta Phi-derivative}) we obtain
\begin{equation}
\label{eqn:xieta Phi-derivative1}
\omega_m\big(\eta_M(m),\xi_M(m)\big)=\big\langle\mathrm{d}_m\Phi\big(\xi_{M}(m)\big),
\eta\big\rangle 
=\lambda (m)\,(\nu_1-\nu_2)\,\langle \eta,\eta\rangle >0.
\end{equation}
\end{proof}

\begin{defn}
\label{defn:stabilizzatori}
If $m\in M^T_{\boldsymbol{\nu}}$, 
$F_m \leqslant T^1_{\boldsymbol{\nu}^\perp}$ denotes its stabilizer subgroup
for $\psi_{\boldsymbol{\nu}_\perp}$ (Definition \ref{defn:psi nu perp}).
Furthermore, $F_{\boldsymbol{\nu}}\leqslant T^1_{\boldsymbol{\nu}^\perp}$ 
denotes the stabilizer for $\psi_{\boldsymbol{\nu}_\perp}$ 
of a general $m\in M^T_{\boldsymbol{\nu}^\perp}$.
\end{defn}

Hence, 
$F_{\boldsymbol{\nu}}\leqslant F_m$,
$\forall\,m\in M^T_{\boldsymbol{\nu}}$.

\begin{lem}
\label{lem:generic stabilizer}
If $m\in M^T_{\boldsymbol{\nu}}\setminus M^G_{\boldsymbol{\nu}}$,
then $F_m\leqslant T^1_{\boldsymbol{\nu}^\perp}\cap Z(G)$.
In particular, 
$F_{\boldsymbol{\nu}}\leqslant T^1_{\boldsymbol{\nu}^\perp}\cap Z(G)$.
\end{lem}

\begin{proof}
[Proof of Lemma \ref{lem:generic stabilizer}]
By equivariance, if $\phi_g(m)=m$, then 
$\mathrm{Ad}_g\big(\Phi(m)\big)=\Phi(m)\in \mathfrak{g}$ where $\mathrm{Ad}$ is 
the adjoint action. 
If $m\in M^T_{\boldsymbol{\nu}}\setminus M^G_{\boldsymbol{\nu}}$
then $\Phi (m)$ is not diagonal. The claim then follows from 
by (\ref{eqn:action Ta}).

\end{proof}

\begin{rem} 
For a uniform representation
$F_{\boldsymbol{\nu}}= T^1_{\boldsymbol{\nu}^\perp}\cap Z(G)$,
since $Z(G)$ acts trivially on $M$ (Definition \ref{defn:uniform rep}).
\end{rem}

Let us introduce the quotients (isomorphic to $S^1$)
\begin{equation}
\label{eqn:quotient groups}
S^1(\boldsymbol{\nu}):=T^1_{\boldsymbol{\nu}^\perp}/F_{\boldsymbol{\nu}},\qquad
T^1(\boldsymbol{\nu}):=T^1_{\boldsymbol{\nu}^\perp}/
\big(T^1_{\boldsymbol{\nu}^\perp}\cap Z(G)\big).
\end{equation}
The induced action $\overline{\psi}_{\boldsymbol{\nu}_\perp}:S^1(\boldsymbol{\nu})\times 
M^T_{\boldsymbol{\nu}}\rightarrow M^T_{\boldsymbol{\nu}}$ is locally free
and generically free, hence effective.
If $\left(M^T_{\boldsymbol{\nu}}\right)_{sm}\subseteq
M^T_{\boldsymbol{\nu}}$ is the dense open set where 
$F_m=F_{\boldsymbol{\nu}}$, then $\left(M^T_{\boldsymbol{\nu}}\right)_{sm}$
is a principal $S^1(\boldsymbol{\nu})$-bundle over its image
$(\overline{M}^T_{\boldsymbol{\nu}})_{sm}$.

Given a character $\chi:S^1(\boldsymbol{\nu})\rightarrow \mathbb{C}^*$
we obtain an Hermitian orbifold line bundle
$L_\chi$. Given the CR structure on $M^T_{\boldsymbol{\nu}}$, $L_\chi$ is in fact 
an holomorphic orbifold line bundle on $\overline{M}^T_{\boldsymbol{\nu}}$.
A smooth function 
$\Sigma:M^T_{\boldsymbol{\nu}}\rightarrow \mathbb{C}$ such that
$\Sigma\circ (\overline{\psi}_{\boldsymbol{\nu}_\perp})_{g^{-1}}
=\chi (g)\,\Sigma$ for any $g\in S^1(\boldsymbol{\nu})$ determines a 
smooth section $\sigma_\Sigma$ of $L_\chi$.

By Lemma \ref{lem:generic stabilizer}, 
we have a short exact sequence 
$$
0\rightarrow  \left(T^1_{\boldsymbol{\nu}^\perp}\cap Z(G)\right)/F_{\boldsymbol{\nu}}
\rightarrow S^1(\boldsymbol{\nu})\rightarrow T^1(\boldsymbol{\nu})
\rightarrow 0;
$$
therefore, any character of $T^1(\boldsymbol{\nu})$ yields a character 
of $S^1(\boldsymbol{\nu})$. In particular, we obtain a character
of $S^1(\boldsymbol{\nu})$ from any character of $T$ with kernel
$Z(G)$, whence from the character $e^{-\imath\,(\theta_1-\theta_2)}$ appearing
in (\ref{eqn:action Ta}). Explicitly, evaluating the latter on 
$T^1_{\boldsymbol{\nu}}\cong S^1$ we obtain the character 
$e^{\imath\,(\nu_1+\nu_2)\,\vartheta}$. 
We shall denote by $\chi$ the corresponding character of $S^1(\boldsymbol{\nu}_\perp)$.

By (\ref{eqn:covariance upsilon}), $\Upsilon$ determines a section
$\sigma_\Upsilon$ of $L_\chi$.
By Lemma \ref{lem:Upsilon} we conclude the following.

\begin{thm}
\label{thm:orbifold MTnu}
The symplectically embedded orbifold 
$\overline{M}^G_{\boldsymbol{\nu}}\subset \overline{M}^T_{\boldsymbol{\nu}}$ 
is the zero locus of the transverse section 
$\sigma_{\Upsilon}$
of $L_{\chi}$. 
If $\overline{•\iota}_T: 
\overline{M}^G_{\boldsymbol{\nu}}\subset \overline{M}^T_{\boldsymbol{\nu}}$
is the inclusion, there is a direct sum decomposition of orbifold
vector bundles
$$
\overline{•\iota}_T^*\left(T\overline{M•}^T_{\boldsymbol{\nu}}\right)=
T\overline{•M}^G_{\boldsymbol{\nu}}\oplus
\overline{•\iota}^*_T\left( L_{\mathfrak{a}} \right).
$$
\end{thm}

\subsection{The case of $\mathbb{P}(W_{\mathbf{L},\mathbf{K}})$}

We aim to classify the D-H reductions 
$(\overline{M}^T_{\boldsymbol{\nu}},J_0',\Omega_0')$
when $M=\mathbb{P}(W_{\mathbf{L},\mathbf{K}})$, assuming that
$W_{\mathbf{L},\mathbf{K}}$ is generic (Definition \ref{defn:restricted class}).
In particular, we shall interpret each such K\"{a}hler orbifold as
a weighted projective variety, related to certain explicit 
combinatorial data associated to
$\mathbf{L},\,\mathbf{K},\,\boldsymbol{\nu}$.

\subsubsection{From Hamiltonian circle actions to orbifolds}
\label{sctn:general construction}

The object of this section is to review and slightly extend 
a general construction from \cite{pao-lower},  providing a K\"{a}hler orbifold
from a Hamiltonian circle action with positive moment map.
This construction generalizes the
one of weighted projective spaces.

Let $R$ be an $r$-dimensional connected projective manifold, with
complex structure $J_R$, and
let $(B,h)$ be a positive holomorphic line bundle on $R$, with $\nabla$
the unique compatible covariant derivative.
Also, let $Y\subset B^\vee$ be the unit circle bundle, with projection $\pi:Y\rightarrow R$; 
let $\alpha\in \Omega^1(Y)$
the connection form corresponding to $\nabla$. Hence (by the positivity of $(B,h)$)
$\mathrm{d}\alpha=2\,\pi^*(\omega_R)$, where $\omega_R$ is a Hodge form on $R$.
Thus $(R,J_R,2\,\omega_R)$ is a K\"{a}hler manifold.

Suppose that there is an holomorphic 
and Hamiltonian circle action $\mu:T^1\times R\rightarrow R$
on $(R,J_R,2\,\omega_R)$, with (normalized) moment map
$\mathcal{M}: R\rightarrow \mathbb{R}$. 
Then there is an infinitesimal \lq action\rq\, 
$\mathrm{d}\mu:\mathfrak{t}^1\rightarrow\mathfrak{X}(R)$
at Lie algebra level.
These Hamiltonian data determine 
an infinitesimal contact CR action of $T^1$ on $Y$, lifting $\mathrm{d}\mu$
\cite{k}: if $\xi=\partial/\partial r\in 
Lie(T^1)\cong \mathbb{R}$
then 
\begin{equation}
\label{eqn:defn of xiY}
\xi_Y:=\xi_R^\sharp-\mathcal{M}\,\partial_\theta\in \mathfrak{X}(Y)
\end{equation}
is a contact vector field.
Here 
$\upsilon^\sharp\in \mathfrak{X}(Y)$ is the horizontal lift of
the vector field
$\upsilon\in \mathfrak{X}(R)$ with respect to $\alpha$, and
$\partial_\theta$ is the generator of the
structure circle action on $Y$ (fiber rotation). Furthermore, we write
$\mathcal{M}$ for $\mathcal{M}\circ \pi:Y\rightarrow \mathbb{R}$.

Let us make the
stronger hypothesis that that there is an actual group action 
$\tilde{\mu}:T^1\times Y\rightarrow Y$ lifting $\mu$ associated to this
infinitesimal lift; that is, $\mathrm{d}\tilde{\mu}(\xi)=\xi_Y$. 
Let us suppose also that $\mathcal{M}>0$. 
Then, in view of (\ref{eqn:defn of xiY}), $\xi_Y(y)\neq 0$ at every $y\in Y$;
thus $\tilde{\mu}$ is locally free.
Perhaps passing to a quotient group if necessary, we may assume 
that $\tilde{\mu}$ is effective, whence generically free.
Therefore the orbit space $R':=Y/\tilde{\mu}$ is naturally an orbifold, and
the projection 
$\pi':Y\rightarrow R'$ is an orbifold circle bundle on $R'$.

On $Y$, we have the following distributions: 
\begin{enumerate}
\item the vertical tangent space for
$\pi$, $V(\pi):=\ker(\mathrm{d}\pi)=\mathrm{span}(\partial_\theta)$;
\item the horizontal tangent space for 
$\alpha$, $H=\ker(\alpha)$;
\item the vertical tangent space for 
$\pi'$, $V(\pi'):=\ker (\mathrm{d}\pi')=\mathrm{span}(\xi_Y)$.
\end{enumerate}
 
For every $y\in Y$, $V(\pi)_y\subset T_yY$ is the tangent space to
the $S^1$-orbit (we denote the circle by $S^1$ when it acts on $Y$ by the structure
rotation action), $V(\pi')_y\subset T_yY$ is the tangent space to
the $T^1$-orbit, and
$H(y)$ is
isomorphic to $T_{\pi(y)}R$ via $\mathrm{d}_y\pi$, and to
the uniformized tangent space $T_{\pi'(y)}R'$ via $\mathrm{d}_y\pi'$.
The tangent bundle of $Y$ splits as
\begin{equation}
\label{eqn:double splitting Y}
TY=V(\pi)\oplus H=V(\pi')\oplus H.
\end{equation}

Let $J_H$ be the complex structure on the vector bundle $H$ given by pull-back of $J$. 
Then $(H,J_H)$ is a $\tilde{\mu}$-invariant 
CR structure on $Y$, and it descends to an orbifold complex 
structure $J_{R'}$ on
$R'$ (the arguments in \cite{pao-lower} were formulated over the smooth
locus, but they can be extended to the orbifold case). 
Thus $(R',J_{R'})$ is a complex orbifold.

Let us set $\beta:=\alpha/\mathcal{M}\in \Omega^1(Y)$; then $H=\ker(\beta)$,
$\beta$ is $\tilde{\mu}$-invariant and 
$\beta(\xi_Y)=-1$. Hence $\beta$ is a
connection form for $q$. Thus there exists 
$\omega_{R'}\in \Omega^2(R')$ such that $\mathrm{d}\beta=2\,(\pi')^*(\omega_{R'})$.
Since 
$$
\mathrm{d}\beta=-\frac{1}{\mathcal{M}^2•}\,\mathrm{d}\mathcal{M}\wedge \alpha
+\frac{2}{\mathcal{M}•}\,\pi^*(\omega_R),
$$
$\mathrm{d}\beta$ restricts on each $H(y)$ to a linear symplectic structure compatible
with $J_H(y)$; therefore $\omega_{R'}$ is an orbifold K\"{a}hler form on 
$(R',J_{R'})$ (see \S 2.2 of \cite{pao-lower}).
%

\begin{rem}
The two orbifold fibrations 
$R \stackrel{\pi}{\leftarrow} Y \stackrel{\pi'}{\rightarrow} R'$
are dual to each other, meaning that $(R')'=R$ as K\"{a}hler orbifolds.
More precisely, the $S^1$-action $r$ on $Y$ given by counterclockise 
fiber rotation
descends to an Hamiltonian action $\mu'$ on $(R',\omega_{R'})$, with moment map
$1/\mathcal{M}$ (interpreted as a function on $R'$), of which it is the contact lift.
Applying the same procedure to $(R',J_{R'},\omega_{R'},\mu')$
we return to $(R,J_R,\omega_R,\mu)$ (see \S 2.3 of \cite{pao-lower}). 
In principle, one would need to phrase the previous discussion assuming that $R$ itself is an orbifold, but this won't be needed in the
following.

\end{rem}

A special case of this construction is given by
weighted projective spaces.
Let $\mathbf{a}
=\begin{pmatrix}
a_0&\cdots&a_k
\end{pmatrix}$ be a string of positive integers, 
and 
consider the action $\mu^{\mathbf{a}}$
of $T^1$ on $\mathbb{P}^k$ given by
\begin{equation}
\label{eqn:muatheta}
\mu^{\mathbf{a}}_{\vartheta}:[z_0:\cdots:z_k]\mapsto \left[e^{-\imath\,a_0\,\vartheta}\,z_0:\cdots:
e^{-\imath\,a_k\,\vartheta}\,z_k\right].
\end{equation}

Then $\mu^{\mathbf{a}}$ is Hamiltonian with respect to 
$2\,\omega_{FS}$, with normalized moment map 
\begin{equation}
\label{eqn:moment map'a}
\Phi^{\mathbf{a}}({Z}):=\frac{1}{\|Z\|^2•}\,\sum_{j=0}^k a_j\,|z_j|^2.
\end{equation}

Let $H_k=\mathcal{O}_{\mathbb{P}^k}(1)$ be the hyperplane line bundle on $\mathbb{P}^k$,
endowed with the standard Hermitian metric; its dual
$H_k^\vee$ is the tautological line bundle, and 
the unit circle bundle in $H_k^\vee$ is the unit sphere 
$S^{2k+1}\subset \mathbb{C}^{k+1}$, with projection the
Hopf map $\pi:S^{2k+1}\rightarrow \mathbb{P}^k$. 
The contact lift of $\mu^{\mathbf{a}}$ 
is the restriction to $S^{2k+1}$ of the
unitary representation
\begin{equation}
\label{eqn:muatheta1}
\tilde{\mu}^{\mathbf{a}}_{\vartheta}:(z_0,\cdots,z_k)\mapsto \left(e^{-\imath\,a_0\,\vartheta}\,z_0,\cdots,
e^{-\imath\,a_k\,\vartheta}\,z_k\right).
\end{equation}
We shall use the same symbol $\tilde{\mu}^{\mathbf{a}}_{\vartheta}$ for both the
unitary representation and its restriction to $S^{2k+1}$.
$\tilde{\mu}^{\mathbf{a}}$ is generically free
if the $a_j$'s are coprime.
The quotient $S^{2k+1}/\mu^{\mathbf{a}}$ is the weighted 
projective space $\mathbb{P}(\mathbf{a})$.
Let $\pi':S^{2k+1}\rightarrow\mathbb{P}(\mathbf{a})$
denote the projection.

The induced orbifold K\"{a}hler structure 
$\eta^{\mathbf{a}}\in \Omega^2\big(\mathbb{P}(\mathbf{a})\big)$
is as follows. The vector field generating 
(\ref{eqn:muatheta1}) is $-\tilde{V_{\mathbf{a}}}$, where 
\begin{equation}
\label{eqn:Vatilde}
\tilde{V_{\mathbf{a}}}=\imath\,\sum_{j=0}^ka_j\,\left(
z_j\,\frac{\partial}{•\partial z_j}
-\overline{z}_j\,\frac{\partial}{•\partial\overline{z}_j}\right),
\end{equation}
viewed as a vector field on $S^{2k+1}$.
$\tilde{V}_{\mathbf{a}}$ is the contact lift of $V_{\mathbf{a}}$, where 
$-V_{\mathbf{a}}$ is the vector field generating 
(\ref{eqn:muatheta}). The moment map (\ref{eqn:moment map'a}) can be obtained by
pairing $\tilde{V}_{\mathbf{a}}$ with the connection form
$$
\alpha=\frac{\imath}{•2}\,\sum_{j=0}^k\,
\left(z_j\,\mathrm{d}\overline{z}_j-\overline{z}_j\,\mathrm{d}z_j\right).
$$
Hence $\beta^{\mathbf{a}}:=\alpha/\Phi^{\mathbf{a}}$ is a connection form for
the action generated by $V_{\mathbf{a}}$ on $S^{2k+1}$
(as usual, we write $\Phi^{\mathbf{a}}$ for $\Phi^{\mathbf{a}}\circ \pi$). 
Then
$\eta^{\mathbf{a}}$ is determined
by the relation $2\,{\pi'}^*(\eta^{\mathbf{a}})=\mathrm{d}\beta^{\mathbf{a}}$.


The K\"{a}hler structures
on $\mathbb{P}^k$ and $\mathbb{P}(\mathbf{a})$
can be changed by modifying the
Hermitian product on $\mathbb{C}^{k+1}$. Let $\mathbf{d}=(d_0,\ldots,d_k)$
be a string of positive integers, and
set
\begin{equation}
\label{eqn:a sympl str}
h_{\mathbf{d}}\left(Z,Z'\right):=\sum_{j=0}^k d_j\,z_j\,\overline{z_j'},
\quad \tilde{\omega}_{\mathbf{d}}:=-\Im (h_{\mathbf{d}})=
\frac{\imath}{•2}\,\sum_{j=0}^k d_j\,\mathrm
dz_j\wedge\mathrm{d}\overline{z}_j.
\end{equation}
The action $r_{-\vartheta}:Z\mapsto e^{-\imath\,\vartheta}\,Z$
of $S^1$ on $(\mathbb{C}^k,2\,\tilde{•\omega}_{\mathbf{d}})$ is Hamiltonian,
with normalized moment map
$$
N_{\mathbf{d}}(Z):=\sum_{j=0}^k d_j\,|z_j|^2.
$$

Let $S^{2k+1}_{\mathbf{d}}:=N_{\mathbf{d}}^{-1}(1)\subset \mathbb{C}^{k+1}$
be the unit sphere for $h_{\mathbf{d}}$. Thus
$S^{2k+1}_{\mathbf{d}}$ is the unit circle bundle
in $H^\vee_k$ with respect to the line bundle metric induced by $h_{\mathbf{d}}$. 
The quotient $S^{2k+1}_{\mathbf{d}}/r$
is again $\mathbb{P}^k$, with a new K\"{a}hler structure $\omega_{\mathbf{d}}$
(the symplectic reduction of $\tilde{\omega}_{\mathbf{d}}$).
More explicitly, let $\pi_{\mathbf{d}}:S^{2k+1}_{\mathbf{d}}
\rightarrow \mathbb{P}^k$ be the projection, 
$\iota_{\mathbf{d}}:S^{2k+1}_{\mathbf{d}}
\rightarrow \mathbb{C}^{k+1}$ the inclusion, 
and set
$$
\alpha_{\mathbf{d}}:=\iota_{\mathbf{d}}^*\left(\frac{\imath}{•2}\,\sum_{j=0}^kd_j\,\left(
z_j\,\mathrm{d}\overline{z}_j-\overline{z}_j\,\mathrm{d}z_j\right)\right).
$$
Then $\alpha_{\mathbf{d}}$ is the connection 1-form on
$S^{2k+1}_{\mathbf{d}}$ for $\pi_{\mathbf{d}}$, and 
$$\mathrm{d}\alpha_{\mathbf{d}}
=2\,\pi_{\mathbf{d}}^*(\omega_{\mathbf{d}})=
2\,\iota_{\mathbf{d}}^*(\tilde{•\omega}_{\mathbf{d}}).$$

The action $\mu^{\mathbf{a}}$ in
(\ref{eqn:muatheta}) is Hamiltonian on 
$\left(\mathbb{P}^k, 2\,\omega_{\mathbf{d}}\right)$,
with normalized moment map
\begin{equation}
\label{eqn:Phi a}
\Phi^{\mathbf{a}}_{\mathbf{d}}([Z]):=
\frac{\sum_{j=0}^k a_j\cdot d_j\,|z_j|^2}{\sum_{j=0}^k d_j\,|z_j|^2}.
\end{equation}
%
 %
%
%
The contact lift of $\mu^{\mathbf{a}}$ to $S^{2k+1}_{\mathbf{d}}$ is 
again functionally given by (\ref{eqn:muatheta1});
we still have $S^{2k+1}_{\mathbf{d}}/\tilde{\mu}^{\mathbf{a}}=\mathbb{P}(\boldsymbol{a})$, 
but with a new K\"{a}hler
form $\eta^{\mathbf{a}}_{\mathbf{d}}$.
Namely, $\beta^{\mathbf{a}}_{\mathbf{d}}:=\alpha_{\mathbf{d}}/\Phi^{\mathbf{a}}_{\mathbf{d}}$
is a connection form for $\tilde{\mu}^{\mathbf{a}}$ on $S^{2k+1}_{\mathbf{d}}$,
and 
$\eta^{\mathbf{a}}_{\mathbf{d}}$ is determined by the condition
\begin{equation}
\label{eqn:eta a d}
\mathrm{d}\beta^{\mathbf{a}}_{\mathbf{d}}=
2\,{q^{\mathbf{a}}_{\mathbf{d}}}^*(\eta^{\mathbf{a}}_{\mathbf{d}}),
\end{equation}
where $q^{\mathbf{a}}_{\mathbf{d}}:S^{2k+1}_{\mathbf{d}}\rightarrow \mathbb{P}(\boldsymbol{a})$
is the projection.
The linear automorphism $\tilde{f}_{\mathbf{d}}:
\mathbb{C}^{k+1}\rightarrow\mathbb{C}^{k+1}$ given by $(z_j)\mapsto (\sqrt{d}_j\,z_j)$ descends to automorphisms
$f_\mathbf{d}:\mathbb{P}^k\rightarrow\mathbb{P}^k$ and 
$f_\mathbf{d}^\mathbf{a}:\mathbb{P}(\mathbf{a})\rightarrow\mathbb{P}(\mathbf{a})$, satisfying $f_\mathbf{d}^*(\omega_{FS})=\omega_\mathbf{d}$ and
${f_\mathbf{d}^\mathbf{a}}^*(\eta'_{\mathbf{a}})
=\eta_\mathbf{d}^\mathbf{a}$.

Let us remark in passing the following homogeneity property.

\begin{lem}
\label{lem:invarianza omega a}
For any string of positive integers $\mathbf{d}=\begin{pmatrix}
d_0&\cdots&d_k
\end{pmatrix}$ and $r=1,2,\ldots$, we have 
$\omega_{r\,\mathbf{d}}=\omega_{\mathbf{d}}\in \Omega^2(\mathbb{P}^k)$.
\end{lem}

\begin{proof}
[Proof of Lemma \ref{lem:invarianza omega a}]
Let $\pi_{\mathbf{d}}:S^{2k+1}_\mathbf{d}\rightarrow \mathbb{P}^k$,
$\pi_{r\,\mathbf{d}}:S^{2k+1}_{r\,\mathbf{d}}\rightarrow \mathbb{P}^k$
be the the Hopf maps. We have, by definition, 
$h_{r\mathbf{a}}=r\,h_{\mathbf{a}}$; therefore, $S^{2k+1}_{r\,\mathbf{a}}
=\delta_{\frac{1}{\sqrt{r}}}\left(S^{2k+1}_{\mathbf{a}}\right)$, where
$\delta_s (Z)=s\,Z$. Since
$ \pi_{\mathbf{a}}=\pi_{r\mathbf{a}}\circ \delta_{\frac{1}{\sqrt{r}}}$, we have 
$$
\pi_{\mathbf{a}}^*(\omega_{r\,\mathbf{a}})=
\delta_{\frac{1}{\sqrt{r}}}^*\left(\pi_{r\mathbf{a}}^*(\omega_{r\,\mathbf{a}})\right)
=\delta_{\frac{1}{\sqrt{r}}}^*\left(\tilde{\omega}_{r\,\mathbf{a}}\right)
=\tilde{\omega}_{\mathbf{a}}=\pi_{\mathbf{a}}^*(\omega_{\mathbf{a}}).
$$

\end{proof}

\begin{cor}\label{cor:FS}
If $r=1,2,\ldots$ and $\mathbf{r}=\begin{pmatrix}
r&\cdots&r
\end{pmatrix}$, then $\omega_{\mathbf{r}}=\omega_{FS}$
(the standard Fubini-Study form).
\end{cor}

\begin{proof}
$\omega_{FS}$ corresponds to
$\mathbf{1}=\begin{pmatrix}
1&\cdots&1
\end{pmatrix}
$. 

\end{proof}

The following variant yields a class of weighted projective varieties.
Let 
$\mathbf{b}
=\begin{pmatrix}
b_0&\cdots&b_l
\end{pmatrix}$ be another string of positive integers. On
$\mathbb{P}^k\times \mathbb{P}^l$, consider the
K\"{a}hler structure 
$\omega_{\mathbf{a},\mathbf{b}}:=\omega_{\mathbf{a}}+ \omega_{\mathbf{b}}$
(symbols of pull-back are omitted). 
$\omega_{\mathbf{a},\mathbf{b}}$ is the Hodge form associated to 
$H_{k,l}:=\mathcal{O}_{\mathbb{P}^k}(1)\boxtimes 
\mathcal{O}_{\mathbb{P}^l}(1)$ and the tensor product of the 
Hermitian products
$h_{\mathbf{a}}$, $h_{\mathbf{b}}$.
The corresponding unit circle bundle 
$X_{\mathbf{a},\mathbf{b}}\subset H_{k,l}^\vee$
can be identified with the image
$S^{2k+1}_\mathbf{a}\otimes_{k,l}S^{2l+1}_\mathbf{b}\subset \mathbb{C}^{k+1}\times \mathbb{C}^{l+1}$
of the map 
\begin{equation}
\label{eqn:tau a b}
\tau_{\mathbf{a},\mathbf{b}}:(Z,W)\in 
S^{2k+1}_{\mathbf{a}}\times S^{2l+1}_{\mathbf{b}}\mapsto Z\otimes _{k,l}W\in 
\mathbb{C}^{k+1}\otimes \mathbb{C}^{l+1};
\end{equation}
we have denoted by $\otimes_{k,l}:\mathbb{C}^{k+1}\times \mathbb{C}^{l+1}
\rightarrow\mathbb{C}^{k+1}\otimes \mathbb{C}^{l+1}$ the tensor product operation.
Equivalently, $X_{\mathbf{a},\mathbf{b}}$ is the quotient of $S^{2k+1}_{\mathbf{a}}\times S^{2l+1}_{\mathbf{b}}$ by 
the $S^1$-action $(Z,W)\mapsto \left(e^{\imath\,\theta}\,Z,e^{-\imath\,\theta}\,W\right)$.
The $S^1$-action on $X_{\mathbf{a},\mathbf{b}}$ given by scalar multiplication
(clockwise rotation) is 
$r_{e^{\imath\,\vartheta}}(Z\otimes _{k,l}W):=
e^{\imath\,\vartheta}\,Z\otimes _{k,l}W$. The projection 
$\pi_{\mathbf{a},\mathbf{b}}:X_{\mathbf{a},\mathbf{b}}\rightarrow\mathbb{P}^k\times \mathbb{P}^l$
is $\pi_{\mathbf{a},\mathbf{b}}(Z\otimes _{k,l}W):=([Z],[W])$.

Let $\iota_{\mathbf{a},\mathbf{b}}:S^{2k+1}_{\mathbf{a}}\times S^{2l+1}_{\mathbf{b}}
\hookrightarrow \mathbb{C}^{k+1}\times \mathbb{C}^{l+1}$ be the inclusion. 
The connection 1-form $\alpha_{\mathbf{a},\mathbf{b}}$ 
on $X_{\mathbf{a},\mathbf{b}}$ is determined by the relation
\begin{equation}
\label{eqn:connection form ab}
\tau_{\mathbf{a},\mathbf{b}}^*(\alpha_{\mathbf{a},\mathbf{b}})
=\iota_{\mathbf{a},\mathbf{b}}^*\left(\tilde{\alpha}_{\mathbf{a},\mathbf{b}}\right),
\end{equation}
where 
\begin{equation}
\label{eqn:alpha ab}
\tilde{\alpha}_{\mathbf{a},\mathbf{b}}:=
\frac{\imath}{2•}\,\left[\sum_{j=0}^ka_j\,
\left(z_j\,\mathrm{d}\overline{z}_j-\overline{z}_j\,\mathrm{d}z_j\right)+
\sum_{j=0}^l b_j\,
\left(w_j\,\mathrm{d}\overline{w}_j-\overline{w}_j\,\mathrm{d}w_j\right)\right].
\end{equation}
Furthermore,
$\mathrm{d}\alpha_{\mathbf{a},\mathbf{b}}=2\,\pi_{\mathbf{a},\mathbf{b}}^*(
\omega_{\mathbf{a},\mathbf{b}})$.

The product $T^1$-action 
\begin{eqnarray}
\label{eqn:ham flow}
\mu^{\mathbf{a},\mathbf{b}}_{\vartheta}\big([Z],\,[W])&=&
\left(\left[e^{-\imath\,a_0\,\vartheta}\,z_0:\cdots:e^{-\imath\,a_k\,\vartheta}\,z_k\right],
\left[e^{-\imath\,b_0\,\vartheta}\,w_0:\cdots:e^{-\imath\,b_l\,\vartheta}\,w_l\right]\right)\nonumber\\
&=&\left(\mu^{\mathbf{a}}_{\vartheta}\big([Z]),\,\mu^{\mathbf{b}}_{\vartheta}\big([W])\right)
\end{eqnarray}
is clearly Hamiltonian on $\left(\mathbb{P}^k\times \mathbb{P}^l, 2\,\omega_{\mathbf{a},\mathbf{b}}\right)$, with normalized moment map
\begin{equation}
\label{eqn:ham flow Phi ab}
\Phi_{\mathbf{a},\mathbf{b}}([Z],[W]):=\Phi_{\mathbf{a}}^{\mathbf{a}}([Z])
+\Phi_{\mathbf{b}}^{\mathbf{b}}([W]),
\end{equation}
where $\Phi_{\mathbf{a}}^{\mathbf{a}}$ and $\Phi_{\mathbf{b}}^{\mathbf{b}}$ are as in 
(\ref{eqn:Phi a}).
Its contact lift $\tilde{\mu}^{\mathbf{a},\mathbf{b}}$  is the 
restriction 
to $X_{\mathbf{a},\mathbf{b}}=S^{2k+1}_{\mathbf{a}}\otimes _{k,l} S^{2l+1}_{\mathbf{b}}$ of the tensor product representation
$\tilde{\mu}^{\mathbf{a}}\otimes 
\tilde{\mu}^{\mathbf{b}}$
on $\mathbb{C}^{k+1}\otimes \mathbb{C}^{l+1}$.
The latter is the unitary representation $\tilde{\mu}_{\vartheta}^{\mathbf{c}}:(X_{ij})\mapsto \left(e^{-\imath\,c_{ij}\,\vartheta}\,
X_{ij}\right)$ associated to the
string $\mathbf{c}=(c_{ij})$, with $c_{ij}:=a_i+b_j>0$.

We shall set 
$$\mathbb{P}(\mathbf{a},\mathbf{b}):=X_{\mathbf{a},\mathbf{b}}/\tilde{\mu}^{\mathbf{a},\mathbf{b}},$$ 
with projection
$\pi'_{\mathbf{a},\mathbf{b}} :X_{\mathbf{a}, \mathbf{b} }
\rightarrow\mathbb{P}(\mathbf{a},\mathbf{b})$, orbifold complex structure
$K_{\mathbf{a},\mathbf{b}}$, and K\"{a}hler form $\eta_{\mathbf{a},\mathbf{b}}$.
Explicitly,
$\beta_{\mathbf{a}, \mathbf{b} }:=\alpha_{\mathbf{a}, \mathbf{b} }/
\Phi_{\mathbf{a}, \mathbf{b} }$ is a connection form for 
$ \pi'_{\mathbf{a},\mathbf{b}}$, and $\eta_{\mathbf{a}, \mathbf{b} }$ is determined by the
relation
\begin{equation}
\label{eqn:eta ab}
2\,{\pi'}_{\mathbf{a},\mathbf{b}}^*(\eta_{\mathbf{a}, \mathbf{b} })=
\mathrm{d}\beta_{\mathbf{a}, \mathbf{b} }.
\end{equation}

We can interpret $\mathbb{P}(\mathbf{a},\mathbf{b})$
as a weighted projective variety, 
as follows.
Consider the Segre embedding 
$$\sigma_{k,l}:([Z],[W])\in\mathbb{P}^k\times \mathbb{P}^l
\mapsto [Z\otimes_{k,l}W]\in \mathbb{P}\left(
\mathbb{C}^{k+1}\otimes \mathbb{C}^{l+1}\right)\cong \mathbb{P}^{kl+k+l}.
$$
In coordinates, this is given by 
$T_{ij}=Z_i\,W_j$.
Let $\mathcal{C}_{k,l}\subset \mathbb{C}^{k+1}\otimes \mathbb{C}^{l+1}$
be the affine cone over $\sigma_{k,l}(\mathbb{P}^k\times \mathbb{P}^l)$; its ideal 
$I(\mathcal{C}_{k,l})\trianglelefteq
\mathbb{K}[X_{ij}]$ 
is generated by the quadratic
polynomials $T_{ij}\,T_{ab}-T_{ib}\,T_{aj}$ ($0\le i,a\le k$,
$0\le j,b\le l$). 

Let us denote by 
$\tilde{\mu}^{\mathbf{c}}_{\mathbb{C}^*}$ the extension of $\tilde{\mu}^{\mathbf{c}}$ to
$\mathbb{C}^*$, and consider
the weighted projective space
$$
\mathbb{P}(\mathbf{c}):=
\left(\mathbb{C}^{k+1}\otimes \mathbb{C}^{l+1}\setminus\{0\}\right)/
\tilde{\mu}^{\mathbf{c}}_{\mathbb{C}^*}.
$$
The weighted projective subvarieties of $\mathbb{P}(\mathbf{c})$
are in one-to-one correspondence with the prime ideals of
$\mathbb{K}[T_{ij}]$ that are homogeous with respect to the grading 
$\mathrm{deg}_{\mathbf{c}}(T_{ij})=c_{ij}$.
Since $I(\mathcal{C}_{k,l})$ is generated by $\mathrm{deg}_{\mathbf{c}}$-homogenous elements, it determines a
weighted projective subvariety
$$
P(\mathcal{C}_{k,l};\mathbf{c}):=
\mathcal{C}_{k,l}/\tilde{\mu}^{\mathbf{c}}_{\mathbb{C}^*}\subset \mathbb{P}(\mathbf{c}).$$
%

Let $\mathbf{d}=(d_{ij})$ be any positive sequence, 
and let
$S^{2(kl+k+l)+1}_{\mathbf{d}}\subset
\mathbb{C}^{k+1}\otimes \mathbb{C}^{l+1}$
be the unit sphere for the Hermitian product
$h_{\mathbf{d}}$. Then 
$S^{2(kl+k+l)+1}_{\mathbf{d}}$ is 
$\tilde{\mu}^{\mathbf{c}}$-invariant, and 
$\mathbb{P}(\mathbf{c})=S^{2(kl+k+l)+1}_{\mathbf{d}}/\tilde{\mu}^{\mathbf{c}}$.
With this description, 
$\mathbb{P}(\mathbf{c})$
inherits the orbifold K\"{a}hler structure $\eta^{\mathbf{c}}_{\mathbf{d}}$.
Explicitly, let 
$\iota_{\mathbf{d}}:S^{2(kl+k+l)+1}_{\mathbf{d}}\hookrightarrow \mathbb{C}^{k+1}\otimes
\mathbb{C}^{l+1}$ be the inclusion, and
set
\begin{equation}
\label{eqn:alpha d product}
\alpha_{\mathbf{d}}:=\iota_{\mathbf{d}}^*\left(\frac{\imath}{•2}\,
\sum_{i,j}d_{ij}\,\left[T_{ij}\,\mathrm{d}\overline{T}_{ij}-
\overline{T}_{ij}\,\mathrm{d}T_{ij}\right]\right),
\end{equation}
\begin{equation}
\label{eqn:Phi cd product}
\Phi^{\mathbf{c}}_{\mathbf{d}}([T]):=
\frac{\sum_{i,j}c_{ij}\cdot d_{ij}\,|T_{ij}|^2}{•\sum_{i,j}d_{ij}\,|T_{ij}|^2}
\qquad ([T]\in \mathbb{P}\left(\mathbb{C}^{k+1}\otimes
\mathbb{C}^{l+1}\right)),
\end{equation}
\begin{equation}
\label{eqn:connection cd}
\beta^{\mathbf{c}}_{\mathbf{d}}:=\frac{1}{•\Phi^{\mathbf{c}}_{\mathbf{d}}}\,
\alpha_{\mathbf{d}},
\end{equation}
where in the latter relation $\Phi^{\mathbf{c}}_{\mathbf{d}}$ is viewed as a function on
$S^{2(kl+k+l)+1}_{\mathbf{d}}$. 
Then $\beta^{\mathbf{c}}_{\mathbf{d}}$ is a connection 1-form for 
the projection $q^{\mathbf{c}}_{\mathbf{d}}:S^{2(kl+k+l)+1}_{\mathbf{d}}\rightarrow
\mathbb{P}(\mathbf{c})$, and $\eta^{\mathbf{c}}_{\mathbf{d}}$ satisfies
\begin{equation}
\label{eqn:eta dc product}
2\,{q^{\mathbf{c}}_{\mathbf{d}}}^*\left(\eta^{\mathbf{c}}_{\mathbf{d}}\right)=
\mathrm{d}\beta^{\mathbf{c}}_{\mathbf{d}}
\end{equation}
(recall (\ref{eqn:eta a d}) and (\ref{eqn:Phi a})).
Hence,
$\eta^{\mathbf{c}}_{\mathbf{d}}$ restricts to an orbifold K\"{a}hler structure
on the complex suborbifold $P(\mathcal{C}_{k,l};\mathbf{c})\subset 
\mathbb{P}(\mathbf{c})$.

The following is left to the reader:

\begin{lem}
If $d_{ij}=a_i\cdot b_j$, then 
$\mathcal{C}_{k,l}\cap S^{2(kl+k+l)+1}_{\mathbf{d}}=X_{\mathbf{a},\mathbf{b}}$.
Hence $P(\mathcal{C}_{k,l};\mathbf{c})=\mathbb{P}(\mathbf{a},\mathbf{b})$.

\end{lem}

\begin{lem}
\label{lem:symplectic embedding}
Assume $c_{ij}=a_i+b_j$, $d_{ij}=a_i\,b_j$. 
Let $\jmath:\mathbb{P}(\mathbf{a},\mathbf{b})\hookrightarrow \mathbb{P}(\mathbf{c})$
be the inclusion, and let $\eta_{\mathbf{a},\mathbf{b}}$ be as in (\ref{eqn:eta ab}).
Then $\jmath^*(\eta^{\mathbf{c}}_{\mathbf{d}})=\eta_{\mathbf{a},\mathbf{b}}$.

\end{lem}

\begin{proof}
[Proof of Lemma \ref{lem:symplectic embedding}]
In view of (\ref{eqn:eta ab}), (\ref{eqn:connection cd}) and (\ref{eqn:eta dc product}),
we need only prove that 
$\alpha_{\mathbf{d}}$ and $\Phi_{\mathbf{d}}^{\mathbf{c}}$
pull back on $X_{\mathbf{a},\mathbf{b}}$ to, respectively,
$\alpha_{\mathbf{a},\mathbf{b}}$ in (\ref{eqn:connection form ab}) 
and
$\Phi_{\mathbf{a},\mathbf{b}}$ in (\ref{eqn:ham flow Phi ab}).
This follows from a straighforward computation
by setting $T_{ij}=Z_i\,W_j$ in (\ref{eqn:alpha d product}) and
(\ref{eqn:Phi cd product}).
\end{proof}

Summing up, we have proved the following .

\begin{prop}
\label{prop:weighted proj}

Let $\mathbf{a}=(a_0,\ldots,a_k)$,
$\mathbf{b}=(b_0,\ldots,b_l)$ be sequences of
positive integers, and set $c_{ij}:=a_i+b_j$.
Define a grading on $\mathbb{K}[T_{ij}]$
by setting $\mathrm{deg}_{\mathbf{c}}(T_{ij})=c_{ij}$.
Then the ideal $I\trianglelefteq\mathbb{K}[T_{ij}]$ with generators
$T_{ij}\,T_{ab}-T_{ib}\,T_{aj}$ is 
$\mathrm{deg}_{\mathbf{c}}$-homogenous, and
$\mathbb{P}(\mathbf{a},\mathbf{b})\subset
\mathbb{P}(\mathbf{c})$ is the corresponding
weighted projective variety.
Furthermore, if $d_{ij}:=a_i\,b_j$
then $(\mathbb{P}(\mathbf{a},\mathbf{b}),\eta_{\mathbf{a},\mathbf{b}})$
is a K\"{a}hler suborbifold of $(\mathbb{P}(\mathbf{c}),\eta^{\mathbf{c}}_{\mathbf{d}})$.

\end{prop}

The $T^1$-action on $\mathbb{P}^k\times \mathbb{P}^l$
\begin{eqnarray}
\label{eqn:ham flow Phi a b bar}
\mu^{\mathbf{a},-\mathbf{b}}_{\vartheta}\big([Z],\,[W])&:=&
\left(\left[e^{-\imath\,a_0\,\vartheta}\,z_0:\cdots:e^{-\imath\,a_k\,\vartheta}\,z_k\right],
\left[e^{\imath\,b_0\,\vartheta}\,w_0:\cdots:e^{\imath\,b_l\,\vartheta}\,w_l\right]\right)\nonumber\\
&=&\left(\mu^{\mathbf{a}}_{\vartheta}\big([Z]),\,\mu^{\mathbf{b}}_{-\vartheta}\big([W])\right)
\end{eqnarray}
can be interpreted in terms of the previous case
by passing to the opposite K\"{a}hler structure on $\mathbb{P}^l$,
and noting that $e^{\imath\,b_j\,\vartheta}\,e_j=e^{-\imath\,b_j\,\vartheta}\bullet e_j$,
where $(e_j)$ is the standard basis and
 $\bullet$ denotes scalar multiplication in $\overline{•\mathbb{C}^{l+1}}$.
Namely, let us consider $\mathbb{P}^k\times \overline{•\mathbb{P}^l}$,
endowed with the K\"{a}hler form 
$\omega_{\mathbf{a},-\mathbf{b}}:=\omega_{\mathbf{a}}- \omega_{\mathbf{b}}$.
The latter is the Hodge form associated to the holomorphic line bundle
$H_{k,\overline{l}}:=\mathcal{O}_{\mathbb{P}^k}(1)\boxtimes 
\mathcal{O}_{\overline{•\mathbb{P}^l}}(1)$ and the positive metric on it
given by the tensor product of the 
Hermitian metrics induced by
$h_{\mathbf{a}}$ on 
$\mathbb{C}^{k+1}$
and $\overline{h_{\mathbf{b}}}$ on $\overline{•\mathbb{C}^{l+1}}$.
The corresponding unit circle bundle 
$X_{\mathbf{a},-\mathbf{b}}=S^{2k+1}_{\mathbf{a}}\otimes_{k,\overline{l}}S^{2l+1}_{\mathbf{b}}$
is the image of the map 
$$
\tau_{\mathbf{a},-\mathbf{b}}:(Z,W)\in 
S^{2k+1}_{\mathbf{a}}\times S^{2l+1}_{\mathbf{b}}\mapsto 
Z\otimes _{k,\overline{l}}W\in \mathbb{C}^{k+1}\otimes\overline{•\mathbb{C}^{l+1}};$$
we have denoted by 
$\otimes_{k,\overline{l}}:\mathbb{C}^{k+1}\times \mathbb{C}^{l+1}
\rightarrow\mathbb{C}^{k+1}\otimes \overline{\mathbb{C}^{l+1}}$ 
the tensor product operation.
Thus componentwise $(Z_i)\otimes_{k,\overline{l}}(W_j)=(Z_i\,\overline{W}_j)$.
Equivalently, it is the quotient of $S^{2k+1}_{\mathbf{a}}\times S^{2l+1}_{\mathbf{b}}$ by
the $S^1$-action $(Z,W)\mapsto \left(e^{\imath\,\theta}\,Z,e^{\imath\,\theta}\,W\right)$.
The projection 
$\pi_{\mathbf{a},-\mathbf{b} }:X_{ \mathbf{a},  -\mathbf{b}    }
\rightarrow\mathbb{P}^k\times \overline{\mathbb{P}^l}$ is
$Z\otimes _{k,\overline{l}}W\mapsto ([Z],[W])$, and the connection form
$\alpha_{\mathbf{a},-\mathbf{b}}$ is determined by obvious variants 
of (\ref{eqn:connection form ab}) and (\ref{eqn:alpha ab}).
We have $2\,\pi^*_{\mathbf{a},-\mathbf{b} }(\omega_{\mathbf{a},-\mathbf{b} })
=\mathrm{d}\alpha_{\mathbf{a},-\mathbf{b} }$.

Then $\mu^{\mathbf{a},-\mathbf{b}}$ in 
(\ref{eqn:ham flow Phi a b bar}) is Hamiltonian 
with respect to $2\,\omega_{\mathbf{a},-\mathbf{b}}$, with normalized moment map 
$\Phi_{\mathbf{a}, \mathbf{b}}$ 
in (\ref{eqn:ham flow Phi ab}).
Its contact lift $\tilde{•\mu}^{\mathbf{a},-\mathbf{b}}$ to $X_{\mathbf{a},-\mathbf{b}}$ 
is the tensor product (for $\otimes_{k,\overline{l}}$) 
of the flows
$\tilde{\mu}^{\mathbf{a}}_{\vartheta}$ and $\tilde{\mu}^{-\mathbf{b}}_{\vartheta}$.
We shall set $\mathbb{P}(\mathbf{a},-\mathbf{b}):=
X_{\mathbf{a},-\mathbf{b}}/\tilde{•\mu}^{\mathbf{a},-\mathbf{b}}$, 
with projection $q_{\mathbf{a},-\mathbf{b} }:X_{\mathbf{a},-\mathbf{b}}\rightarrow
\mathbb{P}(\mathbf{a},-\mathbf{b})$,
and denote by $\eta_{ \mathbf{a},-\mathbf{b} } $
and $K_{\mathbf{a},-\mathbf{b} }$ its 
(orbifold) symplectic and complex structures, respectively.
Thus 
\begin{equation}
\label{eqn:q a -b}
2\,q_{\mathbf{a},-\mathbf{b} }^*(\eta_{ \mathbf{a},-\mathbf{b} }) 
=\mathrm{d}\beta_{\mathbf{a},-\mathbf{b} },
\quad 
\text{where}\quad 
\beta_{\mathbf{a},-\mathbf{b} }:=\alpha_{\mathbf{a},-\mathbf{b} }/
\Phi_{\mathbf{a},\mathbf{b}}.
\end{equation}

The Segre embedding
$$
\sigma_{k,\overline{l}}:([Z],[W])\in \mathbb{P}^k\times \mathbb{P}^l=\mathbb{P}^k\times 
\overline{\mathbb{P}^l}
\mapsto [Z\otimes_{k,\overline{l}}W]\in 
\mathbb{P}\left(\mathbb{C}^{k+1}\otimes \overline{\mathbb{C}}^{l+1}\right),
$$
given in coordinates by $T_{i,\overline{j}}=Z_i\,\overline{W}_j$, intertwines 
$\mu^{\mathbf{a}}\times \mu^{-\mathbf{b}}$ with 
$\mu^{\mathbf{a}}\otimes_{k,\overline{l}} \mu^{-\mathbf{b}}
=\mu^{\mathbf{c}}$, where $c_{ij}=a_i+b_j$. 
The unitary representation $\tilde{\mu}^{\mathbf{c}}$ on $\mathbb{C}^{k+1}\otimes
\overline{•\mathbb{C}^{l+1}}$ is defined in terms of the identification
$\mathbb{C}^{k+1}\otimes
\overline{•\mathbb{C}^{l+1}}\cong \mathbb{C}^{k\,l+k+l+1}$ given by the basis
$e_{i\overline{j}}:=e^{k}_i\otimes_{k,\overline{l}}e^l_j$,
where $(e^k_i)_{i=0}^k$ and $(e^l_j)_{j=0}^l$ are, respectively, the standard basis
of $\mathbb{C}^{k+1}$ and $\mathbb{C}^{l+1}$.
Coordinatewise, $\mu^{\mathbf{c}}_{\vartheta}([T_{i,\overline{j}}])=
\left[e^{-\imath\,c_{ij}\vartheta}\,T_{i,\overline{j}}\right]$.
The same argument used above realizes $\mathbb{P}(\mathbf{a},-\mathbf{b})$ as
the weighted projective variety associated to the cone 
$\mathcal{C}_{k,\overline{l}}\subset \mathbb{C}^{k+1}\otimes \overline{\mathbb{C}}^{l+1}$ over 
$\sigma_{k,\overline{l}}(\mathbb{P}^k\times \overline{\mathbb{P}^l})$
and the weighting $\mathbf{c}$, with induced orbifold
K\"{a}hler structure $\eta_{\mathbf{a},-\mathbf{b}}$.

The latter case is equivalent to the previous one, once we use the standard basis
to induce a unitary isomorphism 
$\overline{•\mathbb{C}^{l+1}}\cong \mathbb{C}^{l+1}$. The reason for emphasizing the
coexistence of the complex structures on $\mathbb{P}^l$ and $\overline{\mathbb{P}^l}$
is the following. 
Being the quotient of $S^{2k+1}_{\mathbf{a}}\times S^{2l+1}_{\mathbf{b}}$ by
the $S^1$-action $(Z,W)\mapsto \left(e^{\imath\,\theta}\,Z,e^{\imath\,\theta}\,W\right)$,
$X_{ \mathbf{a},  -\mathbf{b}}$ is diffeomorphic to the submanifold
$Y_{ \mathbf{a},  -\mathbf{b}}\subset \mathbb{P}^{k+l+1}$ given by
\begin{equation}
\label{eqn:defn Yab}
Y_{ \mathbf{a},  -\mathbf{b}}:=\left\{[Z:W]\in \mathbb{P}^{k+l+1}\,:\,
\|Z\|_{\mathbf{a}}=\|W\|_{\mathbf{b}}\right\}.
\end{equation}
Explicitly, the diffeomorphism
\begin{equation}
\label{eqn:defn fab}
f_{ \mathbf{a},  -\mathbf{b}}:
[Z:W]\in Y_{ \mathbf{a},-\mathbf{b}}\mapsto
\frac{Z}{•\|Z\|_{\mathbf{a}}} 
\otimes_{k,\overline{l}}\frac{W}{•\|W\|_{\mathbf{b}}}\in 
 X_{ \mathbf{a},-\mathbf{b}}
\end{equation}
intertwines the $S^1$-action 
\begin{equation}
\label{eqn:action r}
r:\left(e^{\imath\,\vartheta},[Z:W]\right)\in S^1\times Y_{ \mathbf{a},-\mathbf{b}}\mapsto 
\left[e^{\imath\vartheta/2}\,Z:e^{-\imath\vartheta/2}\,W\right]\in 
Y_{ \mathbf{a},-\mathbf{b}}
\end{equation}
with the structure bundle action on 
$X_{ \mathbf{a},-\mathbf{b}}$
given by scalar multiplication.

As a hypersurface in $\mathbb{P}^{k+l+1}$, $Y_{ \mathbf{a},-\mathbf{b}}$ 
inherits an alternative CR structure.
To interpret the latter, notice that $Y_{ \mathbf{a},  -\mathbf{b}}$ may be 
identified with the unit circle bundle $Z_{ \mathbf{a},  -\mathbf{b}}\subset 
\mathcal{O}_{\mathbb{P}^k}(-1)
\boxtimes \mathcal{O}_{\mathbb{P}^l}(1)$. To make this explicit,
given a one-dimensional complex vector space $L$ and $\ell\in L$, $\ell \neq 0$,
let $\ell^*\in L^\vee$ be the uniquely determined element 
such that $\ell^*(\ell)=1$. Then 
the diffeomorphism 
\begin{equation}
\label{eqn:defn di gab}
g_{ \mathbf{a},  -\mathbf{b}}:
[Z:W]\in Y_{ \mathbf{a},-\mathbf{b}}
\mapsto \frac{Z}{•\|Z\|_{\mathbf{a}}} 
\otimes_{k,l}\left(\frac{W}{•\|W\|_{\mathbf{b}}}\right)^*
\in Z_{ \mathbf{a},  -\mathbf{b}}
\end{equation}
intertwines the action (\ref{eqn:action r}) with the structure bundle action on
$Z_{ \mathbf{a},  -\mathbf{b}}$ given by scalar multiplication. 
Thus we have two $S^1$-equivariant diffeomorphisms
$X_{ \mathbf{a},  -\mathbf{b}}
\stackrel{f_{ \mathbf{a},  -\mathbf{b}}}{\longleftarrow}
Y_{ \mathbf{a},  -\mathbf{b}}\stackrel{g_{ \mathbf{a},  -\mathbf{b}}}{\longrightarrow}
Z_{ \mathbf{a},  -\mathbf{b}}$ and the composition 
$f_{ \mathbf{a},  -\mathbf{b}}\circ g_{ \mathbf{a},  -\mathbf{b}}^{-1}:
Z_{ \mathbf{a},  -\mathbf{b}}\rightarrow X_{ \mathbf{a},  -\mathbf{b}}$
covers the identity $\mathbb{P}^k\times \mathbb{P}^l\rightarrow \mathbb{P}^k\times 
\overline{\mathbb{P}^l}$.

\subsubsection{Application to symplectic reductions}
\label{sctn:symp reduct appl}

Let be given an Hamiltonian action 
$\beta:S^1\times N\rightarrow N$ on a symplectic manifold
$(N,\Omega)$, with normalized moment map 
$\mathfrak{B}:N\rightarrow \mathbb{R}$, such that
$0$ is a regular value of $\mathfrak{B}$.
Then the quotient $N_0:=\mathfrak{B}^{-1}(0)/\beta$ is an
orbifold. 

Let $\pi:\mathfrak{B}^{-1}(0)\rightarrow N_0$ be the projection, and
$\iota:\mathfrak{B}^{-1}(0)\hookrightarrow N$ be the inclusion. 
The reduced orbifold symplectic structure $\Omega_0$ is determined by
the condition $\iota^*(\Omega)=\pi^*(\Omega_0)$.

One the other hand,
given a connection $1$-form $\alpha$ for the
$S^1$-action on $\mathfrak{B}^{-1}(0)$, a closed form $\Omega_{0}'$ on $N_0$
is determined by the condition $\mathrm{d}\alpha=2\,\pi^*(\Omega_0')$ 
\cite{dh}. $[\Omega_0']\in H^2(N_0,\mathbb{R})$ 
is the Chern class of a principal
$S^1$-bundle naturally associated to $\pi$ (see \cite{dh}, \cite{w} 
for a precise discussion). 

Let $J$ be a complex structure on
$N$ compatible with $\Omega$, so that $(N,J,\Omega)$ is a K\"{a}hler manifold,
and such that $\beta$ is holomorphic (i.e., $\beta_g:M\rightarrow M$ is 
$J$-holomorphic
for every $g\in S^1$);
then $J$ descends to an orbifold complex structure 
$J_{0}$ on $N_0$ compatible with $\Omega_{0}$. Thus
$(N_0,J_0,\Omega_{0})$ is a K\"{a}hler orbifold. 
On the other hand, even if $\Omega_{0}'$ turns out to be symplectic, 
$J_{0}$ needn't be compatible with $\Omega_{0}'$.

We shall apply the considerations in \S \ref{sctn:general construction}
to describe a class of Hamiltonian circle actions for which  
$\Omega_{0}'$ is a symplectic form; furthermore, there is
a natural alternative choice of a complex structure $J'_0$ on $N_0$, compatible
with $\Omega_{0}'$. Therefore, in this situation the triple
$(N_0, J_0',\Omega_{0}')$ is a K\"{a}hler orbifold, generally different from
$(N_0,J_0,\Omega_{0})$.
Since $[\Omega_0']\in H^2(N_0,\mathbb{R})$ is the class appearing in the
Duistermaat-Heckman Theorem on the variation of cohomology in symplectic reduction
\cite{dh}, we shall call $(N_0,J_0',\Omega_{0}')$ the \textit{DH-reduction} of $(N,J,\Omega)$ 
under $\beta$.

Given integers $k,\,l\ge 1$, let $\mathbf{a}=\begin{pmatrix}
a_0&\cdots&a_k
\end{pmatrix} ,\,\mathbf{b}= 
\begin{pmatrix}
b_0&\cdots&b_l
\end{pmatrix}
$ be strings of positive integers, and consider the holomorphic
action of $T^1$ on $\mathbb{P}^{k+l+1}$ given by
\begin{eqnarray}
\label{eqn:actiongamma}
\lefteqn{
\gamma^{\mathbf{a},-\mathbf{b}}_{e^{\imath\,\vartheta}}
\big([z_0:\cdots:z_k:w_0:\cdots:w_l]\big)}\\
&=&\left[e^{-\imath\,a_0\,\vartheta}\,z_0:\cdots:e^{-\imath\,a_k\,\vartheta}\,z_k:
e^{\imath\,b_0\,\vartheta}\,w_0:\cdots:e^{\imath\,b_k\,\vartheta}\,w_l\right].
\nonumber
\end{eqnarray}
Then $\gamma^{\mathbf{a},-\mathbf{b}}$ is Hamiltonian
with respect to $\Omega=2\,\omega_{FS}$, with normalized moment map
\begin{equation}
\label{eqn:moment map}
\Gamma_{\mathbf{a},-\mathbf{b}} \big([Z:W]\big):=
\frac{1}{\|Z\|^2+\|W\|^2•}\,\left(\sum_{j=0}^k\,a_j\,|z_j|^2
-\sum_{j=0}^l\,b_j\,|w|_j^2\right).
\end{equation}
Hence 
$\Gamma_{\mathbf{a},-\mathbf{b}}^{-1}(0)=
Y_{\mathbf{a},-\mathbf{b}}$ (see (\ref{eqn:defn Yab})), and 
$0$ is a regular value of 
$\Gamma_{\mathbf{a},-\mathbf{b}}$ \cite{gs cp1}. In fact, the diffeomorphism
$f_{ \mathbf{a},  -\mathbf{b}}$ in (\ref{eqn:defn fab}) intertwines 
$\gamma^{\mathbf{a},-\mathbf{b}}$ and $\tilde{•\mu}^{\mathbf{a},-\mathbf{b}}$.
Therefore, the K\"{a}hler orbifold 
$(N_0,\Omega_{0}', J_0')$ is in this case isomorphic to
$\big(\mathbb{P}(\mathbf{a},-\mathbf{b}),\eta_{\mathbf{a},-\mathbf{b}}\big)$
(hence abstractly to $\big(\mathbb{P}(\mathbf{a},\mathbf{b}),\eta_{\mathbf{a},\mathbf{b}}\big)$).

We can relate the complex structures $J_0$ and $J_0'$ pointwise,
as follows. Let 
$\pi':=q_{\mathbf{a},-\mathbf{b} }\circ f_{ \mathbf{a},  -\mathbf{b}}:
Y_{\mathbf{a},-\mathbf{b}}\rightarrow \mathbb{P}(\mathbf{a},-\mathbf{b})$
be the projection, and consider 
$[Z:W]\in Y_{\mathbf{a},-\mathbf{b}}$. We may assume
$\|Z\|_{\mathbf{a}}=\|W\|_{\mathbf{b}}=1$, i.e.
$Z\in S^{2k+1}_{\mathbf{a}}$, $W\in S^{2l+1}_{\mathbf{b}}$.
Let $H_Z(S^{2k+1}_{\mathbf{a}})\subset T_ZS^{2k+1}_{\mathbf{a}}$
and $H_W(S^{2l+1}_{\mathbf{b}})\subset T_W S^{2l+1}_{\mathbf{b}}$ be
the maximal complex subspaces (with respect to the complex structures 
of $\mathbb{C}^{k+1}$ and $\mathbb{C}^{l+1}$, respectively), with respective
complex structures $K_Z$ and $L_W$.
Then the uniformized tangent space of  $\mathbb{P}(\mathbf{a},-\mathbf{b})$
at $\pi'([Z:W])$ is canonically isomorphic to 
$H_Z(S^{2k+1}_{\mathbf{a}})\times H_W(S^{2l+1}_{\mathbf{b}})$ as a real vector space.
The complex structures $J_0$ and $J_0'$ at $\pi'([Z:W])$ 
correspond to $K_Z\times L_W$ and $K_Z\times (-L_W)$, respectively.

The previous considerations extend to the cases $k=0$, $l>0$, and
$k>0$, $l=0$.
Consider an action $\gamma$ of $T^1$ on $\mathbb{P}^{l+1}$ of the
form
$$
\gamma_{e^{\imath\vartheta}}\big([z_0:\cdots:z_k:w_0]\big)
:=\left[e^{-\imath\,a_0\,\vartheta}\,z_0:\cdots:
e^{-\imath\,a_k\,\vartheta}\,z_k:
e^{\imath\,b_0\vartheta}\,w_0\right],
$$
with moment map
$$
\Gamma:[z_0:\cdots:z_k:w_0]\mapsto 
\frac{1}{•\|Z\|^2+|w_0|^2}\,\left[\sum_{j=0}^k a_j\,|z_j|^2-b_0\,|w_0|^2\right].
$$
Hence $Y:=\Gamma^{-1}(0)$ is entirely contained in the affine open set
where $w_0\neq 0$; explicitly, 
$$
Y=\left\{\left[z_0:\cdots:z_k:\frac{1}{\sqrt{b_0}•}\right]\,:\,\sum_{j=0}^k a_j\,|z_j|^2=1\right\}\cong
S^{2k+1}_{\mathbf{a}}.
$$
The diffeomorphism $[\mathbf{z}:1/\sqrt{b}_0]\in Y\mapsto\mathbf{z}\in S^{2k+1}_{\mathbf{a}}$
intertwines $\gamma$ with the action $\hat{\gamma}_{e^{\imath\,\vartheta}}:
(z_j)\mapsto \left(e^{-\imath\,(a_j+b_0)\,\vartheta}\,z_j\right)$. Assuming, say, that
the integers $a_j+b_0$ are coprime,
$Y/\tilde{•\gamma}$ may be identified with the weighted projective space $\mathbb{P}(b_0+a_0,
\ldots,b_0+a_k)$, and $\Omega_0'$ with the K\"{a}hler form $\eta_{(b_0+a_j)}$.
In this case, $J_0=J_0'$.

\subsection{The DH-reduction of $\mathbb{P}(W_{\mathbf{L},\mathbf{K}})$}

We aim to describe the DH-reductions of a general $\mathbb{P}(W_{\mathbf{L},\mathbf{K}})$ with respect to $T^1_{\boldsymbol{\nu}_\perp}$, when 
$\boldsymbol{\nu}$ varies in $\mathbb{Z}^2$.
We shall call this as the \textit{$\boldsymbol{\nu}$-th DH-reduction }of 
$\mathbb{P}(W_{\mathbf{L},\mathbf{K}})$. 
Recall that this is the triple $(N_0,J_0',\Omega_0')$ (in the notation
in the preample of \S \ref{sctn:symp reduct appl}) when $N=\mathbb{P}(W_{\mathbf{L},\mathbf{K}})$ and $\beta=\psi_{\boldsymbol{\nu}}$ (the restriction of
$\Phi_{\mathbf{L},\mathbf{K}}$ to $T^1_{•\boldsymbol{\nu}_\perp}\cong S^1$
- see (\ref{eqn:T1nu parametrized})).

By way of example, let us start with two special cases.
%
\begin{exmp}
\label{exmp:k=1case}
Consider the
representation $\mu_1^{\oplus r}$ of $G$ on $W_1^{\oplus r}$, for some $r\ge 1$,
as usual composed with 
the Lie group automorphism $B\mapsto (B^t)^{-1}$.  
This corresponds to (\ref{eqn:general representation}) with
$\mathbf{K}=\mathbf{1}:=\begin{pmatrix}
1&\cdots&1
\end{pmatrix}$, $\mathbf{L•}=\mathbf{0}$. 
Let us assume $\nu_1,\nu_2>0$.

By (\ref{eqn:F1}) and
(\ref{eqn:F2}), $F_{1,j}:\mathbb{C}^2\rightarrow \mathbb{C}$ for $j=1,2$
are given by 
$F_{1,1}(Z)=z_0$ and $F_{1,2}(Z)=z_1$, where $Z=\begin{pmatrix}
z_0&z_1
\end{pmatrix}$. Hence by (\ref{eqn:moment map general LK}) 
the moment map 
$\Phi_{\mathbf{1},\mathbf{0}}:\mathbb{P}\left(W_1^{\oplus r}\right)\rightarrow \mathfrak{g}$
is 
\begin{equation}
\label{eqn:K0L0}
\Phi_{\mathbf{1},\mathbf{0}}([Z])=\frac{\imath}{\|Z\|^2}\,
\begin{pmatrix}
\sum_{a=1}^r|z_{a,0}|^2 & \sum_{a=1}^r z_{a,1}\,\overline{z_{a,0}}\\
 \sum_{a=1}^r z_{a,0}\,\overline{z_{a,1}} & \sum_{a=1}^r|z_{a,1}|^2
\end{pmatrix}.
\end{equation}
Here $Z=(Z_1,\ldots,Z_r)\in (\mathbb{C}^2)^r\cong \mathbb{C}^{2r}$, and for each $a$
$Z_a=\begin{pmatrix}
z_{a,0}&z_{a,1}
\end{pmatrix}$.
Therefore, with $M=\mathbb{P}\left(W_1^{\oplus r}\right)$,
$$
M^T_{ \boldsymbol{\nu}}
=\left\{[Z]\,:\,\nu_2\,\sum_{a=1}^r|z_{a,0}|^2=\nu_1\,\sum_{a=1}^r|z_{a,1}|^2\right\}.
$$
Let us define $S_j:(\mathbb{C}^2)^r\rightarrow \mathbb{C}^r$ by setting
$S_j(Z):=
\begin{pmatrix}
z_{1j}&\cdots&z_{rj}
\end{pmatrix}$ for $j=0,1$.
With the unitary change of coordinates
$Z\in \mathbb{C}^{2r}\mapsto \big(S_1(Z),S_0(Z)\big)\in \mathbb{C}^{2r}$, 
we can identify $M^T_{\boldsymbol{\nu}}$ with
$$
{M'}^T_{ \boldsymbol{\nu}}
=\left\{[S_1:S_0]\in \mathbb{P}^{2r-1}\,:\,\nu_1\,\|S_1\|^2=\nu_2\,\|S_0\|^2\right\}.
$$
Let us identify $T^1_{\boldsymbol{\nu}^\perp}$ with 
$S^1$ as in (\ref{eqn:T1nu parametrized}). Then the action
$\psi_{\boldsymbol{\nu}_\perp}$ of
$T^1_{\boldsymbol{\nu}^\perp}$ on $\mathbb{P}^{2r-1}$
corresponds to the circle action given by
\begin{equation}
\label{eqn:gamma on M'}
\gamma_{e^{\imath\,\vartheta}}:[S_1:S_0]\mapsto 
\left[e^{-\imath\,\nu_1\,\vartheta}\,S_1:e^{\imath\,\nu_2\,\vartheta}\,S_0
\right].
\end{equation}
Hence if we set $\boldsymbol{\nu}_2:=\begin{pmatrix}
\nu_2&\cdots&\nu_2
\end{pmatrix}, \,\boldsymbol{\nu}_1:=\begin{pmatrix}
\nu_1&\cdots&\nu_1
\end{pmatrix}\in \mathbb{Z}^r$ then 
$\gamma=\gamma^{\boldsymbol{\nu}_1,-\boldsymbol{\nu}_2}$, 
where notation is as in
(\ref{eqn:actiongamma}).

We can use $f_{\boldsymbol{\nu}_1,-\boldsymbol{\nu}_2}$
in (\ref{eqn:defn fab}) to identify ${M'}^T_{ \boldsymbol{\nu}}\cong 
{M}^T_{ \boldsymbol{\nu}}$ with the unit circle bundle $X_{\boldsymbol{\nu}_1,-\boldsymbol{\nu}_2}$ over $\mathbb{P}^{r-1}\times\overline{\mathbb{P}^{r-1}}$,
with projection 
$\pi_{\boldsymbol{\nu}_1,-\boldsymbol{\nu}_2}:[S_1:S_0]\mapsto ([S_1],[S_0])$.
Since $\gamma$ covers the trivial action on $\mathbb{P}^{r-1}\times\overline{\mathbb{P}^{r-1}}$,  
$\mathbb{P}(\boldsymbol{\nu}_1,-\boldsymbol{\nu}_2)=\mathbb{P}^{r-1}\times\overline{\mathbb{P}^{r-1}}$. 

The connection form $\alpha_{\boldsymbol{\nu}_1,-\boldsymbol{\nu}_2}$
on ${M'}^T_{ \boldsymbol{\nu}}\cong X_{\boldsymbol{\nu}_1,-\boldsymbol{\nu}_2}$,
as unit circle bundle in 
$\mathcal{O}_{\mathbb{P}^{r-1}}(-1)\boxtimes 
\mathcal{O}_{\overline{•\mathbb{P}^{r-1}}}(-1)$, is as follows. Let 
$$
\Xi:(Z,W)\in S^{2r-1}_{\boldsymbol{\nu}_1}\times
S^{2r-1}_{\boldsymbol{\nu}_2}\mapsto [Z:W]\in {M'}^T_{ \boldsymbol{\nu}}
$$ 
and let $\jmath:S^{2r-1}_{\boldsymbol{\nu}_1}\times
S^{2r-1}_{\boldsymbol{\nu}_2}\hookrightarrow \mathbb{C}^r\times \mathbb{C}^r$
be the inclusion; clearly, $S^{2r-1}_{\boldsymbol{\nu}_1}=S^{2r-1}(1/\sqrt{\nu_1})$
and $S^{2r-1}_{\boldsymbol{\nu}_2}=S^{2r-1}(1/\sqrt{\nu_2})$ where 
$S^{2r-1}(r)$ is the sphere centered at the origin of radius $r>0$. Then 
$
\Xi^*(\alpha_{\boldsymbol{\nu}_1,-\boldsymbol{\nu}_2})=
\jmath^*\left(\tilde{\alpha}_{\boldsymbol{\nu}_1,-\boldsymbol{\nu}_2}\right)$,
where
$$
\tilde{\alpha}_{\boldsymbol{\nu}_1,-\boldsymbol{\nu}_2}:=
\frac{\imath}{•2}\,
\left[\nu_1\,
\sum_{j=1}^r\left(z_{j1}\,\mathrm{d}\overline{z}_{j1}-\overline{z}_{j1}\,\mathrm{d}z_{j1}\right)-
\nu_2\,\sum_{j=1}^r \left(z_{j0}\,\mathrm{d}\overline{z}_{j0}
-\overline{z}_{j0}\,\mathrm{d}z_{j0}\right)
 \right].
$$
The corresponding K\"{a}hler structure $\omega$ on $\mathbb{P}^{r-1}\times\overline{\mathbb{P}^{r-1}}$ is then uniquely determined by the condition that
$$
2\,\Xi^*\left(\pi_{\boldsymbol{\nu}_1,-\boldsymbol{\nu}_2}^*(\omega)\right)
=2\,\jmath^*\left(\mathrm{d}\tilde{\alpha}_{\boldsymbol{\nu}_1,-\boldsymbol{\nu}_2}\right).
$$
Either by direct inspection, or by appealing to Corollary
\ref{cor:FS}, one can verify that $\omega=\pi_1^*(\omega_{FS})-\pi_2^*(\omega_{FS})$
($\pi_j$ is the projection of $\mathbb{P}^{r-1}\times\overline{\mathbb{P}^{r-1}}$
onto the $j$-th factor). Furthermore, by (\ref{eqn:ham flow Phi ab}) we have
$\Phi_{\boldsymbol{\nu}_1,\boldsymbol{\nu}_2}=\nu_1+\nu_2$
(constant) and so by 
(\ref{eqn:q a -b}) we conclude that 
$\eta_{\boldsymbol{\nu}_1,-\boldsymbol{\nu}_2}=(\nu_1+\nu_2)^{-1}\,\omega$.

It is evident
from (\ref{eqn:K0L0}) that $\sigma_\Upsilon$ 
(see Theorem \ref{thm:orbifold MTnu})
is the section of 
$\mathcal{O}_{\mathbb{P}^r}(1)\boxtimes \mathcal{O}_{\overline{\mathbb{P}^r}}(1)$
given by the bi-homogeneous polynomial $S_1\cdot \overline{S}_0$. 
Hence $\overline{M}^G_{\boldsymbol{\nu}}\subset \mathbb{P}^r \times\overline{\mathbb{P}^r}$ 
is a (holomorphic) $(1,1)$-divisor.

\end{exmp}

\begin{exmp}
Let us consider the representation $\mu_2^{\oplus r}$ on $W_2^{\oplus r}$; thus 
$\mathbf{K}=\mathbf{2}:=\begin{pmatrix}
2&\cdots&2
\end{pmatrix}$, $\mathbf{L}=\mathbf{0}$
in (\ref{eqn:general representation}).
The functions $F_{2,j}:\mathbb{C}^3\rightarrow \mathbb{C}^2$ in (\ref{eqn:F1}) and
(\ref{eqn:F2}) are given by
$$F_{2,1}:\begin{pmatrix}
z_0&z_1&z_2
\end{pmatrix}
\mapsto \begin{pmatrix}
\sqrt{2}\,z_0&z_1
\end{pmatrix},\quad F_{2,2}:\begin{pmatrix}
z_0&z_1&z_2
\end{pmatrix}
\mapsto \begin{pmatrix}
z_1&\sqrt{2}\,z_2
\end{pmatrix}.$$
For $j=0,1,2$ let us define $S_j:\left(\mathbb{C}^3\right)^r
\rightarrow \mathbb{C}^r$ by setting
$$
S_j(Z_1,\ldots,Z_r):=\begin{pmatrix}
z_{1,j}&\cdots&z_{r,j}
\end{pmatrix};
$$
then
by (\ref{eqn:moment map general LK})
\begin{eqnarray}
\label{eqn:moment map 20}
\lefteqn{\Phi_{\mathbf{2},\mathbf{0}}([Z])}\\
&=&\frac{\imath}{\|Z\|^2•}\,
\begin{pmatrix}
2\,\|S_0(Z)\|^2+\|S_1(Z)\|^2& 
\sqrt{2}\,\left(S_1(Z)^t\,\overline{S_0(Z)}+S_2(Z)^t\,\overline{S_1(Z)}\right)\\
\sqrt{2}\,\left(S_0(Z)^t\,\overline{S_1(Z)}+S_1(Z)^t\,\overline{S_2(Z)}\right)&\|S_1(Z)\|^2+2\,\|S_2(Z)\|^2
\end{pmatrix}.\nonumber
\end{eqnarray}
Assume $\nu_1>\nu_2>0$.
With the unitary change of coordinates
$$Z\in \left(\mathbb{C}^3\right)^r\mapsto 
\begin{pmatrix}
S_1(Z)&S_2(Z)&S_0(Z)
\end{pmatrix}\in \left(\mathbb{C}^r\right)^3,$$ 
$M^T_{\boldsymbol{\nu}}$ may be identified with
\begin{eqnarray*}
{M'}^T_{\boldsymbol{\nu}}&:=&
\left\{[S_1:S_2:S_0]\in \mathbb{P}^{3r-1}=\mathbb{P}\left(\mathbb{C}^r\oplus\mathbb{C}^r
\oplus \mathbb{C}^r\right)\right.\\
&&:\left.(\nu_1-\nu_2)\,\|S_1\|^2+2\,\nu_1\,\|S_2\|^2
=2\,\nu_2\,\|S_0\|^2\right\}.
\end{eqnarray*}

Furthermore, if we identify $T^1_{\boldsymbol{\nu}^\perp}$
with $S^1$ as in (\ref{eqn:T1nu parametrized}), its action
on ${M'}^T_{\boldsymbol{\nu}}$ corresponds to
\begin{eqnarray}
\label{eqn:gamma mu2 r}
\gamma_{e^{\imath\,\vartheta}}\big(
   [S_0:S_1:S_2]\big):= \left[
e^{-\imath\,(\nu_1-\nu_2)\,\vartheta}\,S_1:e^{-2\,\imath\,\nu_1\,\vartheta}\,S_2:e^{2\imath\,\nu_2\,\vartheta}\,S_0\right].
\end{eqnarray}
Let us define $\mathbf{a}_{\boldsymbol{\nu}}\in \mathbb{N}^{2r}$ and
$\mathbf{b}_{\boldsymbol{\nu}}\in \mathbb{N}^{r}$ by setting
$$
\mathbf{a}_{\boldsymbol{\nu}}:=\begin{pmatrix}
\nu_1-\nu_2&\cdots&\nu_1-\nu_2&2\,\nu_1&\cdots&2\,\nu_1
\end{pmatrix},
\qquad\mathbf{b}_{\boldsymbol{\nu}}:=\begin{pmatrix}
2\,\nu_2&\cdots 2\,\nu_2
\end{pmatrix},
$$
where $\nu_1-\nu_2$ and $2\,\nu_1$ are repeated $r$ times.
Then by (\ref{eqn:gamma mu2 r}) we have
$\gamma=\gamma^{\mathbf{a}_{\boldsymbol{\nu}},-\mathbf{b}_{\boldsymbol{\nu}}}$
(see (\ref{eqn:actiongamma})).
By means of $f_{\mathbf{a}_{\boldsymbol{\nu}},-\mathbf{b}_{\boldsymbol{\nu}}}$,
we can identify ${M'}^T_{\boldsymbol{\nu}}$ with the unit circle bundle
$$
X_{\mathbf{a}_{\boldsymbol{\nu}},-\mathbf{b}_{\boldsymbol{\nu}}}
\subset  
\mathcal{O}_{\mathbb{P}^{2r-1}}(-1)\boxtimes
\mathcal{O}_{\overline{•\mathbb{P}^{r-1}}}(-1),
$$
with respect to the Hermitian metric
induced by $h_{\mathbf{a}_{\boldsymbol{\nu}}}$ 
and $h_{\mathbf{b}_{\boldsymbol{\nu}}}$, with projection
$\pi_{\mathbf{a}_{\boldsymbol{\nu}},-\mathbf{b}_{\boldsymbol{\nu}}}:
[S_1:S_2:S_0]\mapsto ([S_1:S_2],[S_0])$.
The structure $S^1$-action given by clockwise fibre rotation is
$$
r_{e^{-{\imath \,\vartheta}}}: [S_1:S_2:S_0]\mapsto 
\left[e^{-\imath\,\vartheta/2}\,S_1:
e^{-\imath\,\vartheta/2}\,S_2:e^{\imath\,\vartheta/2}\,S_0\right].
$$

Thus $\gamma$ may be identified with the contact lift $\tilde{\mu}^{\mathbf{a}_{\boldsymbol{\nu}},-\mathbf{b}_{\boldsymbol{\nu}}}$ to
$X_{\mathbf{a}_{\boldsymbol{\nu}},-\mathbf{b}_{\boldsymbol{\nu}}}$
of the Hamiltonian $S^1$-action
$\mu^{\mathbf{a}_{\boldsymbol{\nu}},-\mathbf{b}_{\boldsymbol{\nu}}}$ on $(\mathbb{P}^{2r-1}\times 
\overline{•\mathbb{P}^{r-1}},2\,\omega_{\mathbf{a}_{\boldsymbol{\nu}},-\mathbf{b}_{\boldsymbol{\nu}}})$ having moment map 
$\Phi_{\mathbf{a}_{\boldsymbol{\nu}},\mathbf{b}_{\boldsymbol{\nu}}}$
(see the discussion following (\ref{eqn:ham flow Phi a b bar}).
Hence $(N_0,J_0',\Omega_0')$ in \S \ref{sctn:symp reduct appl}
with $N=M$ and $S^1\cong T^1_{\boldsymbol{\nu}_\perp}$ is in this case
$\big(\mathbb{P}(\mathbf{a}_{\boldsymbol{\nu}},-\mathbf{b}_{\boldsymbol{\nu}}),\eta_{\mathbf{a}_{\boldsymbol{\nu}},-\mathbf{b}_{\boldsymbol{\nu}}}\big)$.

We can rewrite (\ref{eqn:gamma mu2 r}) as
\begin{eqnarray}
\label{eqn:gamma mu2 r1}
\gamma_{e^{\imath\,\vartheta}}\big(
   [S_0:S_1:S_2]\big):= \left[
e^{-\imath\,(\nu_1+\nu_2)\,\vartheta}\,S_1:e^{-2\,\imath\,(\nu_1+\nu_2)\,\vartheta}\,S_2:S_0\right].
\end{eqnarray}
Passing to the quotient
group $T^1(\boldsymbol{\nu})$ in (\ref{eqn:quotient groups}), this is the 
action
$\overline{\gamma}_{e^{\imath\,\vartheta}}:
[S_1:S_2:S_0]\in {M'}^T_{\boldsymbol{\nu}}\mapsto \left[
e^{-\imath\,\vartheta}\,S_1:e^{-2\,\imath\,\vartheta}\,S_2:S_0\right]\in {M'}^T_{\boldsymbol{\nu}}$.
The latter is functionally independent of $\boldsymbol{\nu}_\perp$, and it
follows that the quotients $\mathbb{P}(\mathbf{a}_{\boldsymbol{\nu}},-\mathbf{b}_{\boldsymbol{\nu}})$ are all isomorphic as complex orbifolds when
$\nu_1>\nu_2>0$.

\end{exmp}

Let us come to a general 
representation $W_{\mathbf{L},\mathbf{K}}$. Let us introduce some terminology.

\begin{defn}
If $W_{\mathbf{K},\mathbf{L}}$ is a representation
fullfilling the equivalent conditions of Proposition 
\ref{prop:equivalent non zero phi}, let 
$$
\mathcal{I}(\mathbf{L},\mathbf{K}):=\big\{(a,j)\,:\,a\in \{1,\ldots,r\},\,
j\in \{0,\ldots,k_a\}\}.
$$
Given $\boldsymbol{\nu}=\begin{pmatrix}
\nu_1&\nu_2
\end{pmatrix} \in \mathbb{Z}^2$, let us define
$n_{\boldsymbol{\nu}}:\mathcal{I}(\mathbf{L},\mathbf{K})
\rightarrow \mathbb{Z}$ by setting
\begin{equation}
\label{eqn:n-aj-a}
n_{\boldsymbol{\nu}}(a,j):=-\nu_2\,(k_a-j+l_a)+\nu_1\,(l_a+j).
\end{equation}
\end{defn}

Let us assume that $\Phi_{\mathbf{L},\mathbf{K}}\big(\mathbb{P}(W_{\mathbf{L},\mathbf{K}})\big)\cap \mathbb{R}_+\cdot \imath\,\boldsymbol{\nu}\neq \emptyset$,
and that $\Phi_{\mathbf{L},\mathbf{K}}$ 
is transverse to $\mathbb{R}_+\cdot \imath\,\boldsymbol{\nu}$.
Then, by Proposition \ref{prop:Phi inviluppo convesso} and Theorem \ref{thm:transversality general}, $\boldsymbol{\nu}$ lies in the interior of one of the wedges cut out by the rays through the 
integral vectors $\boldsymbol{\nu}_{k_a,j_a,l_a}$ 
defined in (\ref{eqn:defn di nuaj}). It follows that:
\begin{enumerate}
\item $n_{\boldsymbol{\nu}}(a,j)\neq 0$ for every $(a,j)\in \mathcal{I}(\mathbf{L},\mathbf{K})$;
\item there exist $(a,j),\,(b,h)\in \mathcal{I}(\mathbf{L},\mathbf{K})$ 
such that $n_{\boldsymbol{\nu}}(a,j)\cdot n_{\boldsymbol{\nu}}(b,h)<0$.
\end{enumerate}

\begin{defn}
Under the previous assumptions, let us define
\begin{equation}
\label{eqn:mathcalP}
\mathcal{P}_{\boldsymbol{\nu}}(\mathbf{L},\mathbf{K}):=\big\{
(a,j)\in \mathcal{I}(\mathbf{L},\mathbf{K})\,:\,
n_{\boldsymbol{\nu}}(a,j)>0\big\},
\end{equation}
\begin{equation}
\label{eqn:mathcalN}
\mathcal{N}_{\boldsymbol{\nu}}(\mathbf{L},\mathbf{K}):=\big\{
(a,j)\in \mathcal{I}(\mathbf{L},\mathbf{K})\,:\,
n_{\boldsymbol{\nu}}(a,j)<0\big\}.
\end{equation}
Then $\mathcal{I}(\mathbf{L},\mathbf{K})$ is the disjoint union of
$\mathcal{P}_{\boldsymbol{\nu}}(\mathbf{L},\mathbf{K})$ and 
$\mathcal{N}_{\boldsymbol{\nu}}(\mathbf{L},\mathbf{K})$, both of which are non-empty.
Furthermore, let us define 
$$
\mathbf{a}_{\boldsymbol{\nu}}(\mathbf{L},\mathbf{K}):=
\big( |n_{\boldsymbol{\nu}}(a,j)| \big)_{(a,j)\in \mathcal{P}_{\boldsymbol{\nu}}(\mathbf{L},\mathbf{K})}\in \mathbb{N}^{|\mathcal{P}_{\boldsymbol{\nu}}(\mathbf{L},\mathbf{K})|},
$$
$$
\mathbf{b}_{\boldsymbol{\nu}}(\mathbf{L},\mathbf{K}):=
\big( |n_{\boldsymbol{\nu}}(a,j)| \big)_{(a,j)\in \mathcal{N}_{\boldsymbol{\nu}}(\mathbf{L},\mathbf{K})}\in \mathbb{N}^{|\mathcal{N}_{\boldsymbol{\nu}}(\mathbf{L},\mathbf{K})|}.
$$
\end{defn}

\begin{thm}
\label{thm:general MTnu}
Let $W_{\mathbf{L},\mathbf{K}}$ be a representation
fullfilling the equivalent conditions of Proposition \ref{prop:equivalent non zero phi}. Suppose that $\boldsymbol{\nu}=\begin{pmatrix}
\nu_1&\nu_2
\end{pmatrix}$, $\nu_1\neq \nu_2$, and that 
\begin{enumerate}
\item $\Phi_{\mathbf{L},\mathbf{K}}\big(\mathbb{P}(W_{\mathbf{L},\mathbf{K}})\big)\cap \mathbb{R}_+\cdot \imath\,\boldsymbol{\nu}\neq \emptyset$;
\item $\Phi_{\mathbf{L},\mathbf{K}}$ 
is transverse to $\mathbb{R}_+\cdot \imath\,\boldsymbol{\nu}$.

\end{enumerate}
Then the $\boldsymbol{\nu}$-th DH-reduction of 
$\mathbb{P}(W_{\mathbf{L},\mathbf{K}})$ is 
\begin{equation}
\label{eqn:nuth orbifold0}
\left(
\mathbb{P}\left(\mathbf{a}_{\boldsymbol{\nu}}(\mathbf{L},\mathbf{K}),
-\mathbf{b}_{\boldsymbol{\nu}}(\mathbf{L},\mathbf{K})
\right)
,
\eta_{
•\mathbf{a}_{\boldsymbol{\nu}}(\mathbf{L},\mathbf{K}),
-•\mathbf{b}_{\boldsymbol{\nu}}(\mathbf{L},\mathbf{K})}
\right).
\end{equation}
Furthermore, if $W_{\mathbf{L},\mathbf{K}}$ is a uniform representation
(Definition \ref{defn:uniform rep}) then the complex orbifold  $\mathbb{P}\left(\mathbf{a}_{\boldsymbol{\nu}}(\mathbf{L},\mathbf{K}),
-\mathbf{b}_{\boldsymbol{\nu}}(\mathbf{L},\mathbf{K})
\right)$ remains constant as $\boldsymbol{\nu}$ 
ranges in the interior of one of the wedges cut out by the rays through the 
$\boldsymbol{\nu}_{k_a,j_a,l_a}$'s. 

\end{thm}

\begin{rem}
As discussed in \S \ref{sctn:general construction}, (\ref{eqn:nuth orbifold0}) is 
a weighted projective subvariety and a K\"{a}hler suborbifold of the weighted projective space
$$
\left(
\mathbb{P}\left(\mathbf{c}_{\boldsymbol{\nu}}(\mathbf{L},\mathbf{K})
\right),\eta^{
\mathbf{c}_{\boldsymbol{\nu}}(\mathbf{L},\mathbf{K})
}_{
\mathbf{d}_{\boldsymbol{\nu}}(\mathbf{L},\mathbf{K})
}\right),
$$
where 
$$
\mathbf{c}_{\boldsymbol{\nu}}(\mathbf{L},\mathbf{K})_{ij}:=
\mathbf{a}_{\boldsymbol{\nu}}(\mathbf{L},\mathbf{K})_i+
\mathbf{b}_{\boldsymbol{\nu}}(\mathbf{L},\mathbf{K})_j,
\quad \mathbf{d}_{\boldsymbol{\nu}}(\mathbf{L},\mathbf{K})_{ij}:=
\mathbf{a}_{\boldsymbol{\nu}}(\mathbf{L},\mathbf{K})_i\cdot
\mathbf{b}_{\boldsymbol{\nu}}(\mathbf{L},\mathbf{K})_j.
$$

\end{rem}
\begin{proof}
[Proof of Theorem \ref{thm:general MTnu}]
By (\ref{eqn:moment map general LK})
we have with $M=\mathbb{P}(W_{\mathbf{L},\mathbf{K}})$
\begin{eqnarray}
\label{eqn:MTnu KL}
M^T_{\boldsymbol{\nu}}&=&
\left\{[Z]:\nu_2\,\sum_{a=1}^r\left(\|F_{k_a,1}(Z_a)\|^2+l_a\,\|Z_a\|^2\right)\right.
\nonumber\\
&&\left.
=\nu_1\,\sum_{a=1}^r\left(\|F_{k_a,2}(Z_a)\|^2+l_a\,\|Z_a\|^2\right)
\right\}.
\end{eqnarray}

In view of (\ref{eqn:F1}) and (\ref{eqn:F2}), 
the relation in (\ref{eqn:MTnu KL}) may be rewritten 
\begin{eqnarray}
\label{eqn:explicit KL}
0&=&\sum_{(a,j)\in \mathcal{I}(\mathbf{L},\mathbf{K})}n_{\boldsymbol{\nu}}(a,j)\,|z_{a,j_a}|^2\\
&=&
\sum_{(a,j)\in \mathcal{P}_{\boldsymbol{\nu}}(\mathbf{L},\mathbf{K})}|n_{\boldsymbol{\nu}}(a,j)|\,|z_{a,j_a}|^2
-\sum_{(a,j)\in \mathcal{N}_{\boldsymbol{\nu}}(\mathbf{L},\mathbf{K})}|n_{\boldsymbol{\nu}}(a,j)|\,|z_{a,j_a}|^2.
\nonumber
\end{eqnarray}
This can be reformulated as follows. Let us 
consider $\mathbb{C}^{|\mathcal{P}_{\boldsymbol{\nu}}(\mathbf{L},\mathbf{K})|}$  and 
$\mathbb{C}^{|\mathcal{N}_{\boldsymbol{\nu}}(\mathbf{L},\mathbf{K})|}$,
with coordinates
$Z=(z_{a,j})_{(a,j)\in \mathcal{P}_{\boldsymbol{\nu}}(\mathbf{L},\mathbf{K})}$,
$W=(w_{a,j})_{(a,j)\in \mathcal{N}_{\boldsymbol{\nu}}(\mathbf{L},\mathbf{K})}$,
respectively. 
On 
$\mathbb{C}^{|\mathcal{P}_{\boldsymbol{\nu}}(\mathbf{L},\mathbf{K})|}$ and 
$\mathbb{C}^{|\mathcal{N}_{\boldsymbol{\nu}}(\mathbf{L},\mathbf{K})|}$ 
we have the positive definite Hermitian products given by
$$
h_{\mathbf{a}_{\boldsymbol{\nu}}(\mathbf{L},\mathbf{K})}(Z,Z')=
\sum_{(a,j)\in \mathcal{P}_{\boldsymbol{\nu}}(\mathbf{L},\mathbf{K})}
|n_{\boldsymbol{\nu}}(a,j)|
\,z_{a,j}\,\overline{z_{a,j}'},
$$
$$
h_{\mathbf{b}_{\boldsymbol{\nu}}(\mathbf{L},\mathbf{K})}(W,W')
=\sum_{(a,j)\in \mathcal{N}_{\boldsymbol{\nu}}(\mathbf{L},\mathbf{K})}
|n_{\boldsymbol{\nu}}(a,j)|
\,w_{a,j}\,\overline{w_{a,j}'},
$$
and so by (\ref{eqn:MTnu KL}) 
\begin{eqnarray}
\label{eqn:MTnu compact}
M^T_{\boldsymbol{\nu}}\cong {M'}^T_{\boldsymbol{\nu}}&:=&
\left\{[Z:W]\in \mathbb{P}
\left(\mathbb{C}^{|\mathcal{P}_{\boldsymbol{\nu}}(\mathbf{L},\mathbf{K})|}
\oplus \mathbb{C}^{|\mathcal{N}_{\boldsymbol{\nu}}(\mathbf{L},\mathbf{K})|}\right)\,:\right.\nonumber\\
&&\left.
h_{\mathbf{a}_{\boldsymbol{\nu}}(\mathbf{L},\mathbf{K})}(Z,Z)=h_{\mathbf{b}_{\boldsymbol{\nu}}(\mathbf{L},\mathbf{K})}(W,W)
\right\}.
\end{eqnarray}
Therefore
$M^T_{\boldsymbol{\nu}}$ may be identified
by $f_{\mathbf{a}_{\boldsymbol{\nu}}(\mathbf{L},\mathbf{K}),-\mathbf{b}_{\boldsymbol{\nu}}(\mathbf{L},\mathbf{K})•}$ 
in (\ref{eqn:defn fab})
with the unit circle bundle in  
$$
X_{\mathbf{a}_{\boldsymbol{\nu}}(\mathbf{L},\mathbf{K}),-\mathbf{b}_{\boldsymbol{\nu}}(\mathbf{L},\mathbf{K})•}\subset \mathcal{O}_{\mathbb{P}^{|\mathcal{P}_{\boldsymbol{\nu}}(\mathbf{L},\mathbf{K})|-1}}(-1)
\boxtimes \mathcal{O}_{\overline{
\mathbb{P}^{
|\mathcal{N}_{\boldsymbol{\nu}}(\mathbf{L},\mathbf{K})|-1}}}(-1),
$$
relative to the Hermitian metric induced by $h_{\mathbf{a}_{\boldsymbol{\nu}}(\mathbf{L},\mathbf{K})}$
and $h_{\mathbf{b}_{\boldsymbol{\nu}}(\mathbf{L},\mathbf{K})}$; the bundle projection is
$\pi:[Z:W]\mapsto ([Z],[W])$.

In the notation 
(\ref{eqn:actiongamma}),
the action of $T^1_{\boldsymbol{\nu}^\perp}$ on ${M'}^T_{\boldsymbol{\nu}}$
given by restriction of $\phi_{\mathbf{L},\mathbf{K}}$ is 
\begin{eqnarray}
\label{eqn:actionKLnu}
\lefteqn{
\gamma^{\mathbf{a}_{\boldsymbol{\nu}}(\mathbf{L},\mathbf{K}),
-\mathbf{b}_{\boldsymbol{\nu}}(\mathbf{L},\mathbf{K})}_{e^{\imath\vartheta}}\big(
\big[(z_{a,j}): (w_{a,j})   \big])}\\
&=&\left[\left( e^{-\imath\,n_{\boldsymbol{\nu}}(a,j)\,\vartheta} \,z_{a,j} \right)   
: \left( e^{-\imath\,n_{\boldsymbol{\nu}}(a,j)\,\vartheta} \,w_{a,j} \right)\right]\nonumber\\
&=&\left[\left( e^{-\imath\,|n_{\boldsymbol{\nu}}(a,j)|\,\vartheta} \,z_{a,j} \right)   
: \left( e^{\imath\,|n_{\boldsymbol{\nu}}(a,j)|\,\vartheta} \,w_{a,j} 
\right)\right].\nonumber
\end{eqnarray}
 $\gamma^{\mathbf{a}_{\boldsymbol{\nu}}(\mathbf{L},\mathbf{K}),
-\mathbf{b}_{\boldsymbol{\nu}}(\mathbf{L},\mathbf{K})}$ corresponds, under the previous identification, to the contact lift 
$\tilde{\mu}^{\mathbf{a}_{\boldsymbol{\nu}}(\mathbf{L},\mathbf{K}),
-\mathbf{b}_{\boldsymbol{\nu}}(\mathbf{L},\mathbf{K})}$
of the Hamiltonian action $\mu^{\mathbf{a}_{\boldsymbol{\nu}}(\mathbf{L},\mathbf{K}),
-\mathbf{b}_{\boldsymbol{\nu}}(\mathbf{L},\mathbf{K})}$ 
(see (\ref{eqn:ham flow Phi a b bar}))
on 
 $$\left(\mathbb{P}^{|\mathcal{P}_{\boldsymbol{\nu}}(\mathbf{L},\mathbf{K})|-1}\times 
 \overline{\mathbb{P}^{|\mathcal{N}_{\boldsymbol{\nu}}(\mathbf{L},\mathbf{K})|-1}•}, 
 2\,\omega_{\mathbf{a}_{\boldsymbol{\nu}}(\mathbf{L},\mathbf{K}),
-\mathbf{b}_{\boldsymbol{\nu}}(\mathbf{L},\mathbf{K})}  \right),$$ 
with moment map 
$\Phi_{\mathbf{a}_{\boldsymbol{\nu}}(\mathbf{L},\mathbf{K}),\mathbf{b}_{\boldsymbol{\nu}}(\mathbf{L},\mathbf{K})}$
(recall (\ref{eqn:ham flow Phi ab})).
Thte first statement of the Theorem follows from this.

Let us assume that $W_{\mathbf{L},\mathbf{K}}$ is a uniform representation.
By definition, there is $s\in \mathbb{Z}$ (independent of $a$) 
such that $k_a+2\,l_a=s$ for
$a=1,\ldots,r$. Then (\ref{eqn:n-aj-a}) may be rewritten
\begin{equation}
\label{eqn:n-aj-a uniform}
n_{\boldsymbol{\nu}}(a,j)
=-\nu_2\,s+(\nu_1+\nu_2)\,(l_a+j).
\end{equation}
Therefore, (\ref{eqn:actionKLnu}) may be rewritten
\begin{eqnarray}
\label{eqn:actionKLnuuniform}
\lefteqn{
\gamma^{\mathbf{a}_{\boldsymbol{\nu}}(\mathbf{L},\mathbf{K}),
-\mathbf{b}_{\boldsymbol{\nu}}(\mathbf{L},\mathbf{K})}_{e^{\imath\vartheta}}\big(
\big[(z_{a,j}): (w_{a,j})   \big])}\\
&=&\left[\left( e^{\imath\,[\nu_2\,s-(l_a+j)\,(\nu_1+\nu_2)\,\vartheta]} \,z_{a,j} \right)   
: \left( e^{\imath\,[\nu_2\,s-(\nu_1+\nu_2)\,(l_a+j)\,\vartheta]} \,w_{a,j} \right)\right]
\nonumber\\
&=&\left[\left( e^{-\imath\,(\nu_1+\nu_2)\,(l_a+j)\,\vartheta]} \,z_{a,j} \right)   
: \left( e^{-\imath\,(\nu_1+\nu_2)\,(l_a+j)\,\vartheta} \,w_{a,j} \right)\right].
\nonumber
\end{eqnarray}
After passing to the quotient group
$T^1(\boldsymbol{\nu})$ in (\ref{eqn:quotient groups}), 
we obtain the action
$\big[(z_{a,j}): (w_{a,j})   \big]\mapsto \left[\left( e^{-\imath\,(l_a+j)\,\vartheta]} \,z_{a,j} \right)   
: \left( e^{-\imath\,\,(l_a+j)\,\vartheta]} \,w_{a,j} \right)\right]$, which is functionally independent
of $\boldsymbol{\nu}$, and the claim can be readily deduced from this.

\end{proof}

\subsection{The case of $\mu_k$ and $\nu_1\gg \nu_2$}

Let us focus on the special case of $\mu_k$, for $k\ge 2$ and 
$\boldsymbol{\nu}$ in the in the range $\nu_1\gg \nu_2$.
For any positive sequence $\boldsymbol{a}=
\begin{pmatrix}
a_1&\cdots&a_k
\end{pmatrix}$, the quotient of the sphere 
$S^{2k-1}_{\boldsymbol{a}}\subset \mathbb{C}^k$ by the
circle action
with weights $\begin{pmatrix}
1&2&\cdots&k
\end{pmatrix}$ is $\mathbb{P}(1,2,\ldots,k)$; the image in 
$\mathbb{P}(1,2,\ldots,k)$ of $S^{2k-1}_{\boldsymbol{a}}\cap (z_1=0)$
is a canonically embedded copy of $\mathbb{P}(2,\ldots,k)$,
independent of $\mathbf{a}$. We shall denote by
$\jmath:\mathbb{P}(2,\ldots,k)\hookrightarrow \mathbb{P}(1,2,\ldots,k)$
the inclusion, which is a holomorphic orbifold embedding.

\begin{thm}
\label{thm:isotopy}
Under the previous assumptions, suppose $\nu_1\gg \nu_2$.
Then $\overline{M}^T_{\boldsymbol{\nu}}\cong \mathbb{P}(1,2,\ldots,k)$.
Furthermore,
there is a smooth isotopy of orbifold embeddings
$$
J:[0,1]\times \mathbb{P}(2,\ldots,k)\rightarrow \mathbb{P}(1,2,\ldots,k)
$$
such that:
\begin{enumerate}
\item $J_0=\jmath$;
\item $J_1\big(\mathbb{P}(2,\ldots,k)\big)=\overline{M}^G_{\boldsymbol{\nu}}$;
\item $J_t\big(\mathbb{P}(2,\ldots,k)\big)$ is a symplectically embedded orbifold in
$(\overline{M}^T_{\boldsymbol{\nu}},\,\Omega_0)$ for every $t\in [0,1]$;
\end{enumerate}
In particular, $\overline{M}^G_{\boldsymbol{\nu}}$ is diffeomorphic to
$\mathbb{P}(2,\ldots,k)$. 
\end{thm}

The following argument will produce $J_t\big(\mathbb{P}(2,\ldots,k)\big)$
as the zero locus of a smoothly varying family of transverse sections of the
orbifold line bundle in Theorem \ref{thm:orbifold MTnu}.

\begin{proof}
[Proof of Theorem \ref{thm:isotopy}]
We have $M=\mathbb{P}^k=\mathbb{P}(W_k)$. By (\ref{eqn:F1}),
(\ref{eqn:F2}) and 
(\ref{eqn:PhiFj}), $M^T_{\boldsymbol{\nu}}$ is contained in the affine open set
$\mathbb{A}^k_0=(z_0\neq 0)$. More explicitly, let us define
$\mathbf{a}_{\boldsymbol{\nu}}
(k)\in \mathbb{N}^k$ by setting $$\mathbf{a}_{\boldsymbol{\nu}}
(k)_j:=\nu_1\,j-\nu_2\,(k-j);$$ thus $\mathbf{a}_{\boldsymbol{\nu}}
(k)_j>0$ for $j=1,\ldots,k$ if $\nu_1>(k-1)\,\nu_2$. Then
\begin{equation}
\label{eqn:MTnu S}
M^T_{\boldsymbol{\nu}}=\left\{
\left[\frac{1}{•\sqrt{•k\,\nu_2} } :v_1:\cdots:v_k\right]\in \mathbb{P}^k\,:\,\sum_{j=1}^k 
\mathbf{a}_{\boldsymbol{\nu}}
(k)_j\,|v_j|^2=1\right\}\cong S^{2k-1}_{•\mathbf{a}_{\boldsymbol{\nu}}
(k)}.
\end{equation}
The following is left to the reader:

\begin{lem}
\label{lem:stima V norma}
Let $\|\cdot\|:\mathbb{C}^k\rightarrow \mathbb{R}$ be the standard Euclidean norm.
If $\nu_1\ge 2\,(k-1)\,\nu_2$, then
$\|V\|\le \sqrt{•2/\nu_1}$ for all $V\in S^{2k-1}_{\mathbf{a}_{\boldsymbol{\nu}}
(k)}$.
\end{lem}

Being irreducible, $\mu_k$ is uniform, 
hence $T^1(\boldsymbol{\nu})=S^1(\boldsymbol{\nu})$ in
(\ref{eqn:quotient groups}). 
Under the isomorphism $\kappa_{\boldsymbol{\nu}}:
S^1\cong T^1_{\boldsymbol{\nu}_\perp}$
in (\ref{eqn:T1nu parametrized}), 
$T^1_{\boldsymbol{\nu}}\cap Z(G)$ corresponds to 
the subgroup of $S^1$ of $(\nu_1+\nu_2)$-th roots of unity; thus the quotient map 
$T^1_{\boldsymbol{\nu}_\perp}\rightarrow T^1(\boldsymbol{\nu})$
corresponds to the Lie group epimorphism 
$e^{\imath\,\vartheta}\in S^1\mapsto e^{\imath\,(\nu_1+\nu_2)\,\vartheta}\in S^1$.

Identified with $S^1$ as in (\ref{eqn:T1nu parametrized}),
$T^1_{\boldsymbol{\nu}^\perp}$
acts on $M^T_{\boldsymbol{\nu}}$ as
\begin{eqnarray}
\label{eqn:T1action muk case}
\lefteqn{\gamma_{e^{\imath\vartheta}}\left(  \left[\frac{1}{•\sqrt{•k\,\nu_2}} :v_1:\cdots:v_k\right]\right)}\\
&=& 
\left[\frac{1}{•\sqrt{•k\,\nu_2}}:e^{-\imath\,\vartheta\,(\nu_1+\nu_2)}\,v_1:\cdots:
e^{-\imath\,j\,\vartheta\,(\nu_1+\nu_2)}\,v_j:
\cdots:e^{-\imath\,k\,\vartheta\,(\nu_1+\nu_2)}\,v_k\right].\nonumber
\end{eqnarray}
Passing to the action $\overline{\gamma}$ 
of the quotient group $T^1(\boldsymbol{\nu})\cong S^1$, 
we conclude that $J_0=J_0'$,
and $\overline{M}^T_{\boldsymbol{\nu}}\cong \mathbb{P}(1,2,\ldots,k)$.
Furthermore, the intersection $S^{2k-1}_{\mathbf{a}_{\boldsymbol{\nu}}
(k)}\cap (v_1=0)$ is clearly $\overline{\gamma}$-invariant, and it
projects down to $\mathbb{P}(2,\ldots,k)\subset \mathbb{P}(1,2,\ldots,k)$.

As $\overline{\gamma}$ 
is effective, any character $\chi$ of $T^1(\boldsymbol{\nu})$
defines an orbifold line bundle $L_\chi$ on $\overline{M}^T_{\boldsymbol{\nu}}$.
We shall
write $L=L_{1}$ if $\chi=\chi_1$ corresponds to the identity of $S^1$.
Any function $f:S^{2k-1}_{•\mathbf{a}_{\boldsymbol{\nu}}
(k)}\rightarrow \mathbb{C}$ which is the restriction of 
a $\mathcal{C}^\infty$ (respectively, holomorphic) 
function on $\mathbb{C}^k$
and satisfies $f\circ \overline{\gamma}_{e^{-\imath\vartheta}}=
e^{\imath\vartheta}\,f$ determines a $\mathcal{C}^\infty$ 
(respectively, holomorphic) section $\sigma_f$ of $L_{\mathfrak{a}}$. 

With abuse of notation, 
in view of (\ref{eqn:MTnu S})
let us regard $\Phi_{12}$ as defined on 
$S^{2k-1}_{•\mathbf{a}_{\boldsymbol{\nu}}
(k)}$; by (\ref{eqn:Phikji}),
\begin{equation}
\label{eqn:Phi12 V}
\Phi_{12}(V)=
\frac{\imath}{(k\,\nu_2)^{-1}+\|V\|^2•}\,\left[\frac{1}{•\sqrt{•\nu_2}}\, v_1+
\sum_{j=1}^{k-1}\sqrt{(k-j)\,(j+1)}\,v_{j+1}\,\overline{v_j}\right].
\end{equation}

Let us consider the continuous function $\Lambda:[0,1]\times 
S^{2k-1}_{•\mathbf{a}_{\boldsymbol{\nu}}
(k)}
\rightarrow \mathbb{C}$ given by 
\begin{eqnarray}
\label{eqn:Upsilon t}
\lefteqn{\Lambda: (t,V)}\\
&\mapsto& \frac{\imath}{(k\,\nu_2)^{-1}+t\,\|V\|^2}\,
\left[\frac{1}{•\sqrt{•\nu_2}}\,\,v_1+t\,
\sum_{j=1}^{k-1}\sqrt{(k-j)\,(j+1)}\,v_{j+1}\,\overline{v_j}\right];\nonumber
\end{eqnarray}
we shall write $\Lambda_t(V):=\Lambda (t,V)$.

Let $(e_1,\ldots,e_k)$ denote the standard basis of $\mathbb{C}^k$, and let 
$(e_1^*,\ldots,e_k^*)$ be the dual basis. 
Then 
\begin{equation}
\label{eqn:equivariance property}
\Lambda_0=
\left.\imath\,k\,\sqrt{\nu_2}\,e_1^*\right|_{S^{2k-1}_{•\mathbf{a}_{\boldsymbol{\nu}}
(k)}},\quad 
\Lambda_1=\Phi_{12},\quad 
\Lambda_t\circ \overline{\gamma}_{e^{-\imath\vartheta}}=e^{\imath\,\vartheta}\,
\Lambda_t,\,\forall\,t\in [0,1];
\end{equation}
in particular,
$\Lambda_t$ corresponds to a $\mathcal{C}^\infty$ section 
$\sigma_{{\Lambda}_t}$
of $L_{1}$.

Using
(\ref{eqn:Upsilon t}) and Lemma \ref{lem:stima V norma}, 
one can prove the following two Lemmas.

\begin{lem}
\label{lem:dtildeUpsilon}
Let us set $\tilde{\Lambda}_t:=-\imath\,(k\,\sqrt{\nu}_2)^{-1}\,\Lambda_t$,
and let us view $\tilde{\Lambda}_t$ as defined on $\mathbb{C}^k$ (by the same functional
equation). Then, uniformly in $V\in S^{2k-1}_{•\mathbf{a}_{\boldsymbol{\nu}}
(k)}$ we have 
$$
\mathrm{d}_V\tilde{\Lambda}_t=e_1^*+O\left(\sqrt{\frac{\nu_2}{•\nu_1}}\right).
$$
\end{lem}

\begin{lem}
\label{lem:stima su v1}
There exists $C>0$ (independent of $k$, $t$ and
$\boldsymbol{\nu}$)
such that if $V\in S^{2k-1}_{•\mathbf{a}_{\boldsymbol{\nu}}
(k)}$ and
$\Lambda_t(V)=0$ for some $t\in [0,1]$, then
$|v_1|\le C\,k\,(\sqrt{•\nu_2}/\nu_1)$.
\end{lem}

The general $V\in S^{2k-1}_{•\mathbf{a}_{\boldsymbol{\nu}}
(k)}$ has the form
\begin{equation}
\label{eqn:Vbetar}
V=\sum_{j=1}^k\,\frac{r_j}{\sqrt{\mathbf{a}_{\boldsymbol{\nu}}
(k)_j}•}\,e_j,\qquad\text{where}\qquad
r_j\in \mathbb{C},\quad \sum_{j=1}^k|r_j|^2=1.
\end{equation}
Lemma \ref{lem:stima su v1} and (\ref{eqn:Vbetar}) imply
that 
 if $V\in S^{2k-1}_{\mathbf{a}_{\boldsymbol{\nu}}
(k)}$ and
$\Lambda_t(V)=0$ for some $t\in [0,1]$, then
$v_1=r_1/\sqrt{\mathbf{a}_{\boldsymbol{\nu}}
(k)_1}$ where $r_1\in \mathbb{C}$ satisfies
\begin{equation}
\label{eqn:stime r1}
|r_1|\le C\,k\,\frac{\sqrt{•\nu_2}}{\nu_1•}\,\sqrt{\mathbf{a}_{\boldsymbol{\nu}}
(k)_1}\le C\,k\,\sqrt{\frac{\nu_2}{•\nu_1}}.
\end{equation}

Hence, if $R'=R'(V):=\sum_{j=2}^k r_j\,e_j$ then 
$$
\nu_1/\nu_2>2\,C^2\,k^2\quad\Rightarrow\quad\|R'\|^2=1-|r_1|^2\ge 1-C^2\,k^2\,(\nu_2/\nu_1)\ge 1/2.
$$ 
Hence there exists 
$j\in \{2,\ldots,k\}$ such that $|r_j|\ge 1/\sqrt{•2\,k} $. Perhaps after renumbering,
we may assume that
$j=2$. 

Therefore, we can draw the following conclusion.

\begin{lem}
\label{lem:bound on r2}
Suppose $\nu_1/\nu_2\gg 0$. If 
$V\in S^{2k-1}_{\mathbf{a}_{\boldsymbol{\nu}}
(k)}$ and
$\Lambda_t(V)=0$ for some $t\in [0,1]$ then,
perhaps after a 
renumbering of $(2,\ldots,k)$ we have
\begin{equation}
\label{eqn:form of V zero locus}
V=\frac{r_1}{\sqrt{\mathbf{a}_{\boldsymbol{\nu}}
(k)_1}}\,e_1	+	\frac{r_2}{\sqrt{\mathbf{a}_{\boldsymbol{\nu}}
(k)_2}}\,e_2 +S(V),
\end{equation}
where $S(V)\in \mathrm{span}_{\mathbb{C}}(e_3,\ldots,e_k)$, $r_1$
satisfies (\ref{eqn:stime r1}) and $|r_2|\ge 1/\sqrt{•2\,k} $.
\end{lem}

Let us set
\begin{eqnarray}
\label{eqn:normal vector}
N_V&:=&\sqrt{\nu_1}\,
\left[-\frac{1}{•\sqrt{•\mathbf{a}_{\boldsymbol{\nu}}
(k)_1}}\,\overline{r}_2\,e_1+\frac{1}{•\sqrt{•\mathbf{a}_{\boldsymbol{\nu}}
(k)_2}}\,\overline{r}_1\,e_2\right]\nonumber\\
&=&-\frac{\overline{r_2}}{\sqrt{1-(k-1)\frac{\nu_2}{\nu_1•}}•}\,e_1+
\frac{\overline{r_1}}{\sqrt{2-(k-2)\frac{\nu_2}{\nu_1•}}•}\,e_2.
\end{eqnarray}
Then 
$\mathrm{span}_{\mathbb{C}}( N_V)\subseteq T_VS^{2k-1}_{\mathbf{a}_{\boldsymbol{\nu}}
(k)}$ and 
$\|N\|> 1/(2k)$ by Lemma \ref{lem:bound on r2}. In view of Lemma \ref{lem:dtildeUpsilon},
we obtain for every $e^{\imath\theta}\in S^1$
\begin{equation}
\label{eqn:diff on V'}
\mathrm{d}_V\tilde{\Lambda}_t\left(e^{\imath\theta}\,N_V\right)=-\frac{e^{\imath\theta}\,
\overline{r_2}}{\sqrt{1-(k-1)\frac{\nu_2}{\nu_1•}}•}+
O\left(\sqrt{\frac{\nu_2}{•\nu_1}}\right).
\end{equation}
It follows that 
$\mathrm{d}_V\tilde{\Lambda}_t$ restricts to a surjective $\mathbb{R}$-linear
map $\mathrm{span}_{\mathbb{C}}(N_V)\rightarrow \mathbb{C}$; therefore 
the same is true
\textit{a fortiori} of the restriction of $\mathrm{d}_V\Lambda_t$ to $T_VS^{2k-1}_{\mathbf{a}_{\boldsymbol{\nu}}
(k)}$.

Thus we conclude the following:

\begin{lem}
\label{lem:zero locus smooth}
Suppose $\nu_1/\nu_2\gg 0$, $V\in S^{2k-1}_{\mathbf{a}_{\boldsymbol{\nu}}
(k)}$, $t\in [0,1]$, and $\Upsilon_t(V)=0$.
Then 
$\left.\mathrm{d}_V\Lambda_t\right|_{T_VS^{2k-1}_{\mathbf{a}_{\boldsymbol{\nu}}
(k)}}\rightarrow \mathbb{C}$ is a surjective
$\mathbb{R}$-linear map.
\end{lem}

Lemma \ref{lem:zero locus smooth} has the following consequences:

\begin{cor}
\label{cor:zero locus smooth}
In the situation of Lemma \ref{lem:zero locus smooth}, 
$Z_t:=\Lambda_t^{-1}(0)\subset S^{2k-1}_{\mathbf{a}_{\boldsymbol{\nu}}
(k)}$ is a smooth
$\gamma$-invariant submanifold of $S^{2k-1}_{\mathbf{a}_{\boldsymbol{\nu}}
(k)}$, 
of (real) codimension $2$. 
\end{cor}

\begin{cor}
$\overline{Z}_t:=Z_t/\gamma\subset
\overline{M}^T_{\boldsymbol{\nu}}$ is a smoothly embedded orbifold
of real codimension $2$.
\end{cor}  

\begin{cor}
\label{cor:manifold with boundary}
Let $\mathcal{Z}:=
\Lambda^{-1}(0)
\subset [0,1]\times S^{2k-1}_{\mathbf{a}_{\boldsymbol{\nu}}
(k)}$. 
Then:
\begin{enumerate}
\item $\mathcal{Z}$ is a
submanifold (with boundary) of codimension $2$
of $[0,1]\times S^{2k-1}_{\mathbf{a}_{\boldsymbol{\nu}}
(k)}$;
\item the projection $p:\mathcal{Z}\rightarrow 
[0,1]$
is a submersion;
\item $Z_t=p^{-1}(t)$ for every $t$.
\end{enumerate}

\end{cor}

$T^1(\boldsymbol{\nu})$ acts on
$[0,1]\times S^{2k-1}_{\mathbf{a}_{\boldsymbol{\nu}}
(k)}$
trivially on the first factor and via $\overline{\gamma}$ on the second, and this action preserves $\mathcal{Z}$ in view of (\ref{eqn:equivariance property}).
The product metric on $[0,1]\times S^{2k-1}_{\mathbf{a}_{\boldsymbol{\nu}}
(k)}$
restricts to an invariant Riemannian metric $g_{\mathcal{Z}}$ on $\mathcal{Z}$.
By $g_{\mathcal{Z}}$, we can define an invariant horizontal distribution for $p$, 
whence an invariant horizontal vector field, whose integral curves are the horizontal lifts of  
$[0,1]$ for $g_{\mathcal{Z}}$. These horizontal lifts 
define an invariant family 
$\psi_p$ of paths, one for each $p\in Z_0$;
for each $t$, the assignment $\psi^t:p\in Z_0\mapsto \psi_p(t)\in Z_t$ is a $\overline{\gamma}$-equivariant diffeomorphism. 
Therefore, $\psi^t$ descends to a smoothly varying family of orbifold diffeomorphisms 
$\overline{\psi}^t:
\overline{Z}_0\rightarrow
\overline{Z}_t$. 
In particular, $\overline{Z}_0$ is diffeomorphic to $\overline{Z}_1$.

Let 
$\mathbf{a}_{\boldsymbol{\nu}}
(k)':=\big(\mathbf{a}_{\boldsymbol{\nu}}
(k)_2,\ldots,\mathbf{a}_{\boldsymbol{\nu}}
(k)_k\big)$.
Then in view of (\ref{eqn:equivariance property})
\begin{equation}
\label{eqn:S0}
Z_0=\{v_1=0\}\cap S^{2k-1}_{\mathbf{a}_{\boldsymbol{\nu}}
(k)'}=\{0\}\times S^{2k-3}_{\mathbf{a}_{\boldsymbol{\nu}}
(k)'};
\end{equation}
by (\ref{eqn:T1action muk case}), $\overline{Z}_0\cong \mathbb{P}(2,3,\ldots,k)$.
Thus every 
$\overline{Z}_t\subset\overline{M}^T_{\boldsymbol{\nu}}$ is diffeomorphic to
$\mathbb{P}(2,3,\ldots,k)$.

Let us show that every $\overline{Z}_t$ is symplectically embedded 
in $(\overline{M}^T_{\boldsymbol{\nu}}, \Omega_0 )$.
By construction, 
$S^{2k-1}_{\mathbf{a}_{\boldsymbol{\nu}}
(k)}\cong M^T_{\boldsymbol{\nu}}=\Psi_{\boldsymbol{\nu}_\perp}^{-1}(0)$
($\Psi_{\boldsymbol{\nu}_\perp}$ 
is as in Lemma \ref{lem:0 regular value nuperp}).
Let $q:S^{2k-1}_{\mathbf{a}_{\boldsymbol{\nu}}
(k)}\rightarrow \overline{M}^T_{\boldsymbol{\nu}}$
be the projection,
and let $\iota:S^{2k-1}_{\mathbf{a}_{\boldsymbol{\nu}}
(k)}\hookrightarrow \mathbb{C}^k\cong\mathbb{A}^k_0\subset \mathbb{P}^k$
be the inclusion; then $q^*(\Omega_0)=\iota^*(\omega_{FS})$.

Let $\omega_0:=(\imath/2)\,\sum_{j=1}^k\mathrm{d}v_j\wedge \mathrm{d}\overline{v}_j$
be the standard symplectic structure on $\mathbb{C}^k$.
Expressing $\omega_{FS}$ in affine coordinates,
by a standard computation we obtain on $\mathbb{A}_0^k$
\begin{equation}
\label{eqn:approximate fs0}
\omega_{FS}=\omega_0+R_2(V),
\end{equation}
where $R_2$ is a differential form vanishing to second order at the origin.
By Lemma \ref{lem:stima V norma},
along $S^{2k-1}_{\mathbf{a}_{\boldsymbol{\nu}}
(k)}$ we have
$\|V\|^2\le 2/\nu_1\le 2\,\nu_2/\nu_1$; hence (\ref{eqn:approximate fs0}) implies
that $\omega_{FS}=\omega_0+O(\nu_2/\nu_1)$ 
on $S^{2k-1}_{\mathbf{a}_{\boldsymbol{\nu}}(k)}$. 
Therefore,
\begin{equation}
\label{eqn:qstar eta red}
q^*(\Omega_0)=\iota^*(\omega_{FS})=
\iota^*(\omega_0)+O\left(\frac{\nu_2}{•\nu_1}\right).
\end{equation}
With $\tilde{•\Lambda}_t:\mathbb{A}^n_0\cong\mathbb{C}^k\rightarrow\mathbb{C}$ 
be as in Lemma \ref{lem:dtildeUpsilon}, let us set 
$\tilde{Z}_t:=\tilde{•\Lambda}_t^{-1}(0)$; thus
$Z_t=\tilde{Z}_t\cap S^{2k-1}_{\mathbf{a}_{\boldsymbol{\nu}}(k)}$.

Let $(\epsilon_1,\epsilon_2,\ldots,\epsilon_{2k-1},\epsilon_{2k})$
be the real basis $(e_1,\imath\,e_1,\ldots,e_k,\imath\,e_k)$
of $\mathbb{C}^k\cong \mathbb{R}^{2k}$. 
Then by Lemma \ref{lem:dtildeUpsilon}
\begin{eqnarray}
\label{eqn:dPhi12 V}
\mathrm{d}_V\tilde{•\Lambda}_t&=&\epsilon_1^*+\imath\,\epsilon_2^*+O\left(\sqrt{\frac{\nu_2}{\nu_1}}\right)\qquad
(V\in S^{2k-1}_{\mathbf{a}_{\boldsymbol{\nu}}(k)}),
\end{eqnarray}
and this implies that if $\nu_1/\nu_2\gg 0$ then 
$\ker(\mathrm{d}_V\tilde{•\Lambda}_t)$ is a (real) 
symplectic
vector subspace of $(\mathbb{C}^k,\omega_0)$ 
whenever $V\in S^{2k-1}_{\mathbf{a}_{\boldsymbol{\nu}}(k)}$
and $t\in [0,1]$.
Given this and (\ref{eqn:qstar eta red}),
we conclude the following:

\begin{lem}
\label{lem:symplectic locus ZV}
If $\nu_1/\nu_2\gg 0$, then the following holds. For every $t\in [0,1]$ and 
$V\in S^{2k-1}_{\mathbf{a}_{\boldsymbol{\nu}}(k)}$ 
such that $\Lambda_t(V)=0$, the tangent space 
$T_V\tilde{Z}_t$ \`{e} un sottospazio vettoriale
simplettico di $(\mathbb{C}^k,\omega_{FS})$.

\end{lem}

\begin{cor}
\label{cor:symplectic locus}
If $\nu_1/\nu_2\gg 0$, 
there exists a $\overline{\gamma}$-invariant
open neighborhood $U\subseteq \mathbb{C}^k$ of 
$S^{2k-1}_{\mathbf{a}_{\boldsymbol{\nu}}(k)}$, 
such that $\tilde{Z}_t':=\tilde{Z}_t\cap U$
is a symplectic submanifold of real codimension $2$
of $(\mathbb{C}^k,\omega_{FS})$,
for every $t\in [0,1]$.
\end{cor}

Let $\jmath_t:\tilde{Z}_t'\hookrightarrow \mathbb{C}^k$ be the inclusion, and set $\omega_t:=\jmath_t^*(\omega_{FS})$.
The restriction $\psi_t:=\Psi_{\boldsymbol{\nu}_\perp}\circ \jmath_t$
is the moment map for the action of $T^1_{\boldsymbol{\nu}^\perp}$ on 
$(\tilde{Z}_t',\omega_t)$, and $Z_t=\psi_t^{-1}(0)$. 
Hence $\overline{Z}_t:=\tilde{Z}_t'/\gamma$, with the reduced symplectic structure
$\overline{\omega}_t$, is the symplectic 
reduction of $(\tilde{Z}_t',\omega_t)$, and
as such it is a symplectic suborbifold of 
$(\overline{M}^T_{\boldsymbol{\nu}},\Omega_0)$.

\end{proof}

\section{$M^G_{\mathcal{O}}$}
\label{sctn:MGO case I}

We shall assume throughout that $\mathbf{0}\not\in \Phi(M)$ and that 
$\Phi$ is transverse to $\mathcal{C}(\mathcal{O})$, and focus on 
$\overline{M}^G_{\mathcal{O}}$ and its relation to
$\overline{M}^G_{\boldsymbol{\nu}}$.
We do not assume that $M$ be projective. 

Given that 
$\Phi$ is transverse to $\mathcal{C}(\mathcal{O})$, $\phi$
has rank $\ge 3$ along $M^G_{\mathcal{O}}$, meaning that for every 
$m\in M_{\mathcal{O}}^{G}$ the evaluation map 
$\mathrm{val}_m:\xi\in \mathfrak{g}\mapsto \xi_M(m)\in T_mM_{\mathcal{O}}^{G}$ has rank $\ge 3$ 
\cite{pao-tori}, \cite{gp}.
Let us give a direct proof for the reader's convenience.

\begin{prop}
Given that $\Phi$ is transverse to $\mathcal{C}(\mathcal{O})$,
for any $m\in M^G_{\mathcal{O}}$ the evaluation map
$\mathrm{val}_m:\mathfrak{g}\rightarrow T_mM$ is injective
on $\ker\big(\Phi(m)\big)$.

\end{prop}

\begin{proof}
If $m\in M^G_{\mathcal{O}}$, then
by equivariance $\Phi$ is transverse to $\mathcal{C}(\mathcal{O})$ at $m$ if and
only if it is transverse to the ray $\mathbb{R}_+\,\Phi(m)$ at $m$. 
Hence, $\mathrm{d}_m\Phi (T_mM)+\mathbb{•R}\,\Phi (m)=\mathfrak{g}^\vee$.
Suppose that $\xi\in \ker\big(\Phi(m)\big)$, and that $\xi_M(m)=0$.
Pick $\alpha\in \mathfrak{g}^\vee$. Then there exists 
$v\in T_mM$ and $\lambda\in \mathbb{R}$ such that
$\alpha=\mathrm{d}_m\Phi(v)+\lambda\,\Phi(m)$. 
Thus 
\begin{eqnarray*}
\alpha (\xi)&=&\mathrm{d}_m\Phi(v)(\xi)+\lambda\,\Phi(m)(\xi)\\
&=&\mathrm{d}_m\Phi(v)(\xi)=\mathrm{d}_m\Phi^\xi(v)=2\,\omega \big(\xi_M(m),v\big)=0.
\end{eqnarray*}
Thus $\alpha  (\xi)=0$ $\forall\,\alpha\,\in \,\mathfrak{g}^\vee$, whence $\xi=0$.
\end{proof}

For example, when 
$\phi=\phi_{\mathbf{L},\mathbf{K}}$ for a uniform representation
(Definition \ref{defn:uniform rep}), 
$\phi$ is bound to have constant rank $3$ along $M^G_{\mathcal{O}}$.

We shall accordingly distinguish two cases: 
1): $\phi$ has constant rank $3$ along $M^G_{\mathcal{O}}$;
2): $\phi$ is generically locally free along $M^G_{\mathcal{O}}$.
Before, however, it is in order to sum up some general facts.

If $m\in M^G_{\mathcal{O}}$, then by definition there exist
unique $\lambda_{\boldsymbol{\nu}} (m)>0$ and 
$h_m\,T\in G/T$ such that
\begin{equation}
\label{eqn:PhiGhm}
\Phi(m)=\imath\,\lambda_{\boldsymbol{\nu}} (m)\,h_m\,
\begin{pmatrix}
\nu_1&0\\
0&\nu_2
\end{pmatrix}\,h_m^{-1}.
\end{equation}
The applications $\lambda_{\boldsymbol{\nu}}$
and $m\in h_m\,T$ are $\mathcal{C}^\infty$.
Furthermore,
$h_{\mu_g(m)}\,T=
g\,h_m\,T$ and $\lambda_{\boldsymbol{\nu}}=\lambda_{\boldsymbol{\nu}}\circ \mu_g$
by the equivariance of $\Phi$.

Let us define 
\begin{equation}
\label{eqn:Tnum}
T^1_{\boldsymbol{\nu}_\perp ,m}:=
h_m\,T^1_{\boldsymbol{\nu}_\perp}\,h_m^{-1},\quad
T_m:=h_m\,T\,h_m^{-1}\qquad 
(m\in M^G_{\mathcal{O}}).
\end{equation}
Then $T^1_{\boldsymbol{\nu}_\perp ,m}\leqslant T_m\leqslant G$ 
are well-defined, and 
\begin{equation}
\label{eqn:Tnumg}
T^1_{\boldsymbol{\nu}_\perp,\mu_g(m)}=g\,T_{\boldsymbol{\nu}_\perp, m}\,g^{-1}\leqslant
T_{\mu_g(m)}=g\,T_m\,g^{-1}\qquad
(g\in G,\, m\in M^G_{\mathcal{O}}).
\end{equation}
In particular, if $g\in T_m$ then $T_{\mu_g(m)}=T_m$;
hence $T_{m'}=T_m$ for every $m'\in T_m\cdot m$; similarly for
$T_{\boldsymbol{\nu}_\perp, m}$.

\begin{defn}
\label{defn:moving action}
Let us define the action 
$\rho:S^1\times M^G_{\mathcal{O}}\rightarrow M^G_{\mathcal{O}}$
by setting
$$
\rho_{e^{\imath\vartheta}}(m):=\phi_{h_m\,\kappa_{\boldsymbol{\nu}}(
e^{\imath\vartheta})\,h_m^{-1}}(m),
$$
where $\kappa_{\boldsymbol{\nu}}:S^1\rightarrow T^1_{\boldsymbol{\nu}^\perp}$
is as in (\ref{eqn:T1nu parametrized}).
\end{defn}

Thus the $\rho$-orbit of $m\in M^G_{\mathcal{O}}$ is $T_m\cdot m$.
The following facts are more or less well-known, and are either discussed in \cite{gp}, or can
be deduced using arguments in \cite{gs-hq}, \cite{gp}:

\begin{lem}
\label{lem:isotropic leaves O}
$M^G_{\mathcal{O}}\subset M$ is a compact and connected $G$-invariant hypersurface,
and $\rho$ is locally free.
The isotropic leaves
of $M_{\mathcal{O}}^{G}$ are the 
$\rho$-orbits. Hence, the quotient $\overline{M}_{\mathcal{O}}^{G}$ is an orbifold of real dimension $2d-2$, with a reduced symplectic structure $\omega_{•\overline{M}_{\mathcal{O}}^{G}}$.
\end{lem}

Let $p:M^G_{\mathcal{O}}\rightarrow \overline{M}^G_{\mathcal{O}}$
be the projection. Then $p(M^G_{\boldsymbol{\nu}})$ is diffeomorphic to
$\overline{M}^G_{\boldsymbol{\nu}}$ in (\ref{eqn:MTnu MGnu}); 
with abuse of notation, we shall write
$\overline{M}^G_{\boldsymbol{\nu}} = p(M^G_{\boldsymbol{\nu}})$.
We have seen that $\overline{M}^G_{\boldsymbol{\nu}}$ has an intrinsic
symplectic structure $\omega_{•\overline{M}^G_{\boldsymbol{\nu}}}$, and that
$(\overline{M}^G_{\boldsymbol{\nu}},\omega_{•\overline{M}^G_{\boldsymbol{\nu}}})$ 
is a symplectic suborbifold of $(\overline{M}^T_{\boldsymbol{\nu}},
\omega_{•\overline{M}^T_{\boldsymbol{\nu}}})$.
Arguing as in \S \ref{sctn:kahler structure}
one obtains the following.

\begin{lem}
\label{lem:sympl suborb MGnu}
Under the previous identification,
$(\overline{M}^G_{\boldsymbol{\nu}},\omega_{•\overline{M}^G_{\boldsymbol{\nu}}})$ 
is a symplectic suborbifold of $(\overline{M}^T_{\mathcal{O}},
\omega_{•\overline{M}^T_{\mathcal{O}}})$.
\end{lem}

Furthermore, we have:

\begin{lem}
\label{lem:commuting actions}
For every $e^{\imath\vartheta}\in S^1$, 
$g\in G$, $m\in M_{\mathcal{O}}^{G}$ we have
$$
\rho_{e^{\imath\vartheta}}\circ \phi_g(m)
=\phi_g\circ \rho_{e^{\imath\vartheta}}(m).
$$
\end{lem}

\begin{cor}
$\phi$ (restricted to
$M_{\mathcal{O}}^{G}$) descends to a smooth action 
$$
\overline{\phi}:G\times \overline{M}_{\mathcal{O}}^{G}\rightarrow
\overline{M}_{\mathcal{O}}^{G}.
$$
Furthermore,
$\overline{\phi}$ is symplectic for
$\overline{\omega}$.

\end{cor}

In view of (\ref{eqn:PhiGhm}) and Definition
\ref{defn:moving action}, $\left.\Phi\right|_{M_{\mathcal{O}}^{G}•}$ is $\rho$-invariant, and therefore it descends to a smooth function $\overline{\Phi}:
\overline{M}_{\mathcal{O}}^{G}\rightarrow
\mathfrak{g}$.

\begin{cor}
\label{cor:hamiltonian action mu}
$\overline{\phi}$ is Hamiltonian for
$2\,\overline{\omega}$, with moment map
$\overline{\Phi}$.

\end{cor}

\subsection{Case 1)}

In this case, we shall establish in Theorem \ref{thm:Delta simplettomorfismo}
that $\overline{M}_{\mathcal{O}}^{G}$ factors symplectically as the product of
$\overline{M}_{\boldsymbol{\nu}}^T$ and a coadjoint orbit.

\begin{prop}
\label{prop:centro triviale}
If the rank of $\phi$ along $M_{\mathcal{O}}^{G}$ is generically $3$,
then it is $3$ everywhere on $M_{\mathcal{O}}^{G}$. 
Furthermore, the stabilizer $F_m\leqslant G$ of any 
$m\in M_{\mathcal{O}}^{G}$ is 1-dimensional subgroup $F_m\leqslant T_m$,
transverse to $T^1_{\boldsymbol{\nu}_\perp ,m}$ in $T_m$. 

\end{prop}

This will be the case, for instance, if $\mu$ is associated to a uniform representation,
in which case the connected component of $F_m$ is $Z(G)$.

\begin{proof}
[Proof of Proposition \ref{prop:centro triviale}]
Let us first assume that $m\in M_{\boldsymbol{\nu}}^{G}$. 
Then any $g\in F_m$ commutes with $\Phi (m)$, therefore
$g\in T$ since $\nu_1\neq \nu_2$. Thus $F_m\leqslant T$. Since the action of $T^1_{\boldsymbol{\nu}_\perp}$
is locally free at $m$, $F_m$ has to be transverse to $T^1_{\boldsymbol{\nu}_\perp}$
in $T$.
The general case follows from this and (\ref{eqn:Tnumg}).

\end{proof}

For $\overline{m}\in \overline{M}^G_{\mathcal{O}}$, let
$\overline{F}_{\overline{m}}$ denote the stabilizer of $\overline{m}$ for $\overline{\mu}$.

\begin{cor}
\label{cor:stabilizer overline m}
Under the hypothesis of Proposition \ref{prop:centro triviale},
$\overline{F}_{\overline{m}}=T_m$, for any $\overline{m}\in \overline{M}^G_{\mathcal{O}}$ and
$m\in p^{-1}(\overline{m})$. In particular, 
$\overline{F}_{\overline{m}}=T$, for any $\overline{m}\in \overline{M}^G_{\boldsymbol{\nu}}$.
\end{cor}

\begin{cor}
\label{cor:trivial central action}
Under the hypothesis of Proposition \ref{prop:centro triviale},
$\overline{\phi}$ is trivial on $Z(G)$. If, in addition, $\nu_1+\nu_2\neq 0$, 
then $\lambda_{\boldsymbol{\nu}}$
is constant.
\end{cor}

\begin{proof}
[Proof of Corollary \ref{cor:trivial central action}]
For any $\overline{m}\in \overline{M}^G_{\mathcal{O}}$,
$\overline{F}_{\overline{m}}$ is a maximal torus, hence contains
$Z(G)$. This proves the first statement.
As to the second, $\lambda_{\boldsymbol{\nu}}$ descends to a well-defined smooth
function on $\overline{M}^G_{\mathcal{O}}$, which we shall denote by the same symbol.
Furthermore, the Hamiltonian function for the (trivial) action of 
$Z(G)$ on $(\overline{M}^G_{\mathcal{O}},\,2\omega_{\overline{M}^G_{\mathcal{O}}•})$
is $\langle\overline{\Phi},\imath\,I_2\rangle =\lambda_{\boldsymbol{\nu}}\,
(\nu_1+\nu_2)$. Since $\nu_1+\nu_2\neq 0$, $\lambda_{\boldsymbol{\nu}}$
needs to be contant.
\end{proof}

By (\ref{eqn:PhiGhm}), if 
$\overline{m}\in 
\overline{M}^G_{\mathcal{O}}$ 
and $m\in p^{-1}(\overline{m})$ we have 
$$
\mu_{h_m^{-1}}(m)\in M^G_{\boldsymbol{\nu}},
\qquad 
\overline{\mu}_{•h_m^{-1}}(\overline{m})\in \overline{M}^G_{\boldsymbol{\nu}}.
$$
Thus we obtain well-defined and $\mathcal{C}^\infty$ orbifold maps 
\begin{equation}
\label{eqn:Delta map}
\Delta:\overline{m}\in \overline{M}^G_{\mathcal{O}}\mapsto
\left(\overline{\mu}_{•h_m^{-1}}(\overline{m}),
h_{\overline{m}}\,T\right)\in \overline{M}^G_{\boldsymbol{\nu}}\times (G/T),
\end{equation}
and
\begin{equation}
\label{eqn:Theta map}
\Theta:
\left(\overline{m},\,h\,T\right)\in \overline{M}^G_{\boldsymbol{\nu}}\times (G/T)
\mapsto \overline{\mu}_h(\overline{m})\in  \overline{M}^G_{\mathcal{O}}.
\end{equation}
Notice that $\Delta$ and $\Theta$ are well-defined by Corollary \ref{cor:stabilizer overline m},
and $\Theta=\Delta^{-1}$.
Hence $\Delta$ and $\Theta$ are diffeomorphism. 
Furthermore, $G$ acts on $\overline{M}^G_{\boldsymbol{\nu}}\times (G/T)$ 
by
$$
\alpha_g\left(\overline{m},h\,T\right):=\left(\overline{m},\,g h\,T\right).
$$
It is clear from (\ref{eqn:Theta map}) that $\Theta$ intertwines $\alpha$ and $\overline{\phi}$,
that is,
$\Theta\circ \alpha_g=\overline{\phi}_g\circ \Theta$ for all $g\in G$.

Let us identify $G/T$ with $\mathbb{P}^1$ by the equivariant diffeomorphism 
$$
\sigma:h\,T\in G/T\mapsto [h\,e_2]\in \mathbb{P}^1,
$$ where $(e_1,e_2)$ is the
standard basis of $\mathbb{C}^2$.
We have proved the following:

\begin{prop}
\label{prop:product structure}
Under the hypothesis of Proposition \ref{prop:centro triviale},
$\overline{M}^G_{\mathcal{O}}$ is equivariantly diffeomorphic to 
$\overline{M}^G_{\boldsymbol{\nu}}\times \mathbb{P}^1$.
\end{prop}

By the K\H{u}nneth formula, we obtain:

\begin{cor}
Under the hypothesis of Proposition \ref{prop:centro triviale}, there is a ring isomorphism
$H^*(\overline{M}^G_{\mathcal{O}})\cong H^*(\overline{M}^G_{\boldsymbol{\nu}})
\otimes H^*(\mathbb{P}^1)$.
\end{cor}

Let us set $\omega_{G/T}:=\sigma^*(\omega_{FS})$, where
$\omega_{FS}$ is the Fubini-Study form.  
On $\overline{M}^G_{\boldsymbol{\nu}}\times (G/T)$ consider the product symplectic
structure $\omega_{•\overline{M}^G_{\boldsymbol{\nu}}}\oplus \omega_{G/T}$.
Let us assume that $\nu_1+\nu_2\neq 0$,
Then $\lambda_{\boldsymbol{\nu}}>0$ is a constant (Corollary
\ref{cor:trivial central action}), and we may consider the 
symplectic form
$$
\omega'_{G/T}:=2\,(\nu_1+\nu_2)\,\lambda_{\boldsymbol{\nu}}\,\omega_{G/T}.
$$
We can strengthen Proposition \ref{prop:product structure} in the following manner:

\begin{thm}
\label{thm:Delta simplettomorfismo}
Under the assumptions on Proposition
\ref{prop:centro triviale}, assume in addition that
$\nu_1+\nu_2\neq 0$. Then
$$\Delta:(\overline{M}^G_{\mathcal{O}},\omega_{\overline{M}^G_{\mathcal{O}}})\rightarrow
\big(\overline{M}^G_{\boldsymbol{\nu}}\times (G/T),
\omega_{•\overline{M}^G_{\boldsymbol{\nu}}}\oplus \omega'_{G/T}\big)$$
is a symplectomorphism.
\end{thm}

\begin{rem} 
The assumption that $\nu_1+\nu_2\neq 0$ is guaranteed in the case 
of $\mathbb{P}(W_{\mathbf{L},\mathbf{K}}))$, 
by Corollary \ref{cor:non zero trace moment}.
\end{rem}

\begin{proof}
[Proof of Theorem \ref{thm:Delta simplettomorfismo}]
$\overline{M}^G_{\mathcal{O}}$ is the $\overline{\mu}$-saturation of $\overline{M}^G_{\boldsymbol{\nu}}$; furthermore, $\overline{M}^G_{\boldsymbol{\nu}}$
maps diffeomorphically under $\Delta$ onto
$\overline{M}^G_{\boldsymbol{\nu}}\times \{I_2\,T\}$. Since 
$\overline{\phi}$ is symplectic on $(\overline{M}^G_{\mathcal{O}},\omega_{\overline{M}^G_{\mathcal{O}}})$, 
$\alpha$ is symplectic on $\big(\overline{M}^G_{\boldsymbol{\nu}}\times (G/T),
\omega_{•\overline{M}^G_{\boldsymbol{\nu}}}\oplus \omega_{G/T}\big)$,
and $\Delta$ intertwines the two symplectic actions,
it suffices to prove the statement along $\overline{M}^G_{\boldsymbol{\nu}}$.
Explicitly, suppose $\overline{m}_0\in \overline{M}^G_{\boldsymbol{\nu}}$ and
$\overline{m}=\overline{\mu}_g(\overline{m}_0)$ for some $g\in G$;
then $\Delta\circ \overline{\phi}_{g}=\alpha_{g}\circ \Delta$ implies 
$\mathrm{d}_{\overline{m}}\Delta\circ \mathrm{d}_{\overline{m}_0}\overline{\phi}_{g}
=\mathrm{d}_{\Delta(\overline{m}_0)}\alpha_{g}\circ \mathrm{d}_{\overline{m}_0}\Delta$.
Hence if $\mathrm{d}_{\overline{m}}\Delta$ is a linear symplectomorphism for every
$\overline{m}\in \overline{M}^G_{\boldsymbol{\nu}}$, then it is so also for every 
$\overline{m}\in \overline{M}^G_{\mathcal{O}}$.

For every $\upsilon\in \mathfrak{g}$, 
let $\upsilon_{\overline{M}^G_{\mathcal{O}}•}$
denote the corresponding orbifold vector field on $\overline{M}^G_{\mathcal{O}}$ 
(see [LT]). 
If $\xi,\,\eta,\,\mathfrak{a}$ are as in (\ref{eqn:defn of xieta}),
Lemma \ref{lem:symplectic direct sum} and Corollary \ref{cor:aM symplectic} imply
that there is a symplectic direct sum of orbifold (uniformized) tangent bundles
$$\jmath^*(T\overline{M}^G_{\mathcal{O}})=T\overline{M}^G_{\boldsymbol{\nu}}
\oplus \jmath^*(\mathfrak{a}_{\overline{M}^G_{\mathcal{O}}•}),$$ where $\jmath:\overline{M}^G_{\boldsymbol{\nu}}\hookrightarrow \overline{M}^G_{\mathcal{O}}$ is the inclusion.

Let us fix $\overline{m}\in \overline{M}^G_{\boldsymbol{\nu}}$, so
that $\Delta(\overline{m})=(\overline{m},I_2\,T)$. We have 
$$
T_{•(\overline{m},I_2\,T)}\big(\overline{M}^G_{\boldsymbol{\nu}}\times (G/T)\big)
\cong T_{•\overline{m}}(\overline{M}^G_{\boldsymbol{\nu}})\oplus
T_{I_2\,T}(G/T)\cong T_{•\overline{m}}(\overline{M}^G_{\boldsymbol{\nu}})\times
\mathfrak{a};
$$
in both cases, the two summands are symplectically orthogonal.
Furthermore, it is apparent from our definition of $\Delta$ that, in terms of
the previous isomorphisms 
$T_{\overline{m}}\overline{M}^G_{\mathcal{O}}\cong 
T_{•\overline{m}}(\overline{M}^G_{\boldsymbol{\nu}})\times
\mathfrak{a}\cong T_{•(\overline{m},I_2\,T)}\big(\overline{M}^G_{\boldsymbol{\nu}}\times (G/T)\big)$, 
$\mathrm{d}_{\overline{m}}\Delta$ corresponds to the identity map
$T_{•\overline{m}}(\overline{M}^G_{\boldsymbol{\nu}})\times
\mathfrak{a}\rightarrow T_{•\overline{m}}(\overline{M}^G_{\boldsymbol{\nu}})\times
\mathfrak{a}$. Therefore, we are reduced to comparing the symplectic structures 
on $\mathfrak{a}$ coming from $\omega_{G/T}$ and from 
$\overline{M}^G_{\mathcal{O}}$.

On the one hand, with $\omega_0$ the standard symplectic structure 
on $\mathbb{C}^2$,
$$
\omega_{G/T,I_2\,T}(\xi,\eta)=\omega_0(\xi\,e_1,\eta_1)=\frac{\imath}{•2}\,
\left(\sum_{j=1}^2\mathrm{d}z_j\wedge\mathrm{d}\overline{z}_j\right)\left(
\begin{pmatrix}
\imath\\
0
\end{pmatrix},\begin{pmatrix}
1\\0
\end{pmatrix}\right)=-1.
$$
On the other, 
\begin{eqnarray*}
\omega_{\overline{M}^G_{\mathcal{O}},\overline{m}}\big(\xi_{\overline{M}^G_{\mathcal{O}}}(\overline{m}), 
\eta_{\overline{M}^G_{\mathcal{O}}}(\overline{m})\big)
&=&\mathrm{d}_{\overline{m}}\overline{\Phi}^\xi\left(
\eta_{\overline{M}^G_{\mathcal{O}}}(\overline{m})\right)\\
&=&\left\langle \left[\eta,\overline{\Phi}(\overline{m})\right],
  \xi    \right\rangle = -2\,(\nu_1+\nu_2)\,\lambda_{\boldsymbol{\nu}}.
\end{eqnarray*}

\end{proof}

\subsection{Case 2)}

Let us relax the assumption that the rank of $\mu$ is everywhere $3$ on
$M^G_{\boldsymbol{\nu}}$. 
On $\overline{M}^G_{\boldsymbol{\nu}}
\times \overline{•B(\mathbf{0},\pi/2)}$ let us define a relation $\sim$ as follows:
$(\overline{m}_1,z_1)\sim (\overline{m}_2,z_2)$ if and only if either
$(\overline{m}_1,z_1)= (\overline{m}_2,z_2)$, or else 
$z_j=(\pi/2)\,e^{\imath\,\theta_j}$, 
$j=1,2$, and $\overline{m_2}
=\overline{\mu}_{D(\theta_1,\theta_2)}(\overline{m}_1)$,
where 
$$
D(\theta_1,\theta_2):=\begin{pmatrix}
e^{\imath\,(\theta_2-\theta_1)}&0\\
0&e^{\imath\,(\theta_1-\theta_2)}
\end{pmatrix}.
$$
Let $\hat{M}^G_{\mathcal{O}}:=\overline{M}^G_{\boldsymbol{\nu}}
\times \overline{•B(\mathbf{0},\delta)}/\sim$ 
denote the corresponding identification space.
If the rank of $\mu$ along $M^G_{\mathcal{O}}$ is constant and equal to three, as in
Proposition \ref{prop:product structure}, then $T$ acts trivially on $\overline{M}^G_{\boldsymbol{\nu}}$; hence there is a homeomorphism
$\hat{M}^G_{\mathcal{O}}=\overline{M}^G_{\boldsymbol{\nu}}\times S^2$.

\begin{thm}
\label{thm:MGO general case}
Suppose that $0\not\in \Phi(M)$, and that $\Phi$ is transverse to
$\mathcal{C}(\mathcal{O})$. Then:
\begin{enumerate}
\item $M^G_{\mathcal{O}}$ is homeomorphic to 
$\hat{M}^G_{\mathcal{O}}$.
\item For every $q$ we have an isomorphism
$$
H^{q}\big(\overline{M}^G_{\mathcal{O}}\big)\cong H^{q-2}\big(\overline{M}^G_{\boldsymbol{\nu}}\big)
\oplus H^{q}\big(\overline{M}^G_{\boldsymbol{\nu}}\big).
$$
\end{enumerate}

\end{thm}

\begin{proof}
[Proof of Theorem \ref{thm:MGO general case}]
Let us consider the $\mathbb{R}$-linear isomorphism 

\begin{equation}
\label{defn:Beta}
B:z\in\mathbb{C}\mapsto B_z:=\imath\,
\begin{pmatrix}
0&z\\
\overline{z}&0
\end{pmatrix}\in  \mathfrak{a}\subset \mathfrak{g}.
\end{equation}

\begin{lem}
\label{lem:exp B eta}
For any $z\in \mathbb{C}$, we have
$$
e^{B_z}=
\begin{pmatrix}
\cos(|z|) & \imath\,\frac{\sin(|z|)}{|z|}\,z\\
\imath\,\frac{\sin(|z|)}{|z|}\,\overline{•z}&\cos(|z|)
\end{pmatrix}=
\cos(|z|)\,I_2+B_{\sin(|z|)\,z/|z|}.
$$
\end{lem}

The previous expression is well-defined only for $z\neq 0$, but
$\sin(w)/w$ extends to an even analytic function $F(w^2)$ on $\mathbb{C}$; therefore
$\sin(|z|)\,z/|z|=F(|z|^2)\,z$ extends to a real-analytic function of 
$z$. We shall regard $e^{B_z}$ as a real-analytic function 
$\mathbb{C}\rightarrow G$.

\begin{proof}
[Proof of Lemma \ref{lem:exp B eta}]
The statement follows from a computation based on the identities
$$
B_z^{2k}=(-1)^k\,|z|^{2k}\,I_2=(\imath\,|z|)^{2k}\,I_2,\quad B_z^{2k+1}=
(-1)^k\,|z|^{2k}\,B_z=(\imath\,|z|)^{2k}\,B_z.
$$
\end{proof}

Let $D_{\boldsymbol{\nu}}$ be the diagonal matrix with diagonal entries $\begin{pmatrix}
\nu_1&\nu_2
\end{pmatrix}$.
Then by Lemma \ref{lem:exp B eta} we have
\begin{eqnarray}
\label{eqn:e b eta D}
\lefteqn{e^{B_z}\,D_{\boldsymbol{\nu}}\,e^{-B_z}}\\
&=&
\begin{pmatrix}
\nu_1\,\cos(|z|)^2+\nu_2\,\sin(|z|)^2&\imath\,(\nu_2-\nu_1)\,\cos(|z|)\,•\sin(|z|)\,
\frac{z}{|z|•}\\
\imath\,(\nu_1-\nu_2)\,\cos(|z|)\,\sin(|z|)\,\frac{\overline{z}•}{|z|•}\,
&\nu_2\,\cos(|z|)^2+\nu_1\,\sin(|z|)^2
\end{pmatrix}.\nonumber
\end{eqnarray}

The function $\lambda_{\boldsymbol{\nu}}:M^G_{\mathcal{O}}\rightarrow \mathbb{R}$,
being $G$-invariant, descends to a smooth function on $\overline{M}^G_{\mathcal{O}}$,
that will be denoted by the same symbol.

\begin{cor}
\label{cor:moment map eta}
Let $\overline{\Phi}_{T^1_{\boldsymbol{\nu}_\perp}}:\overline{M}^G_{\mathcal{O}}
\rightarrow \imath\,\mathbb{R}$ be the moment map for the
Hamiltonian action of $T^1_{\boldsymbol{\nu}_\perp}$ on the symplectic orbifold
$(\overline{M}^G_{\mathcal{O}},\omega_{•\overline{M}^G_{\mathcal{O}}})$. 
Let us identify $T^1_{\boldsymbol{\nu}_\perp}$ with $S^1$ by the isomorphism
$\kappa_{\boldsymbol{\nu}}$ in (\ref{eqn:T1nu parametrized}).
Then for every
$\overline{m}\in \overline{M}^G_{\boldsymbol{\nu}}$ and $z\in \mathbb{C}$ we have
$$
\overline{\Phi}_{T^1_{\boldsymbol{\nu}_\perp}}\big(\overline{\mu}_{e^{B_z}}(\overline{m}    \big)
=\imath\,\left(\nu_1^2-\nu_2^2\right)\,\lambda_{\boldsymbol{\nu}}(\overline{m})\,
\sin(|z|)^2.
$$
\end{cor}

Let us set $\boldsymbol{\nu}':=\begin{pmatrix}
\nu_2&\nu_1
\end{pmatrix}$, $M^G_{\boldsymbol{\nu}'}:=
\Phi^{-1}(\mathbb{R}_+\cdot \boldsymbol{\nu}')$.
Hence,
$$
\overline{M}^G_{\boldsymbol{\nu}'}:=\overline{\Phi}^{-1}(\mathbb{R}_+\cdot \boldsymbol{\nu}')
=p\left(M^G_{\boldsymbol{\nu}'}\right).$$
Furthermore,
\begin{equation}
\label{eqn:defn di gamma}
M^G_{\boldsymbol{\nu}'}=\mu_\gamma\left(M^G_{\boldsymbol{\nu}}\right),\quad
\overline{M}^G_{\boldsymbol{\nu}'}=\overline{\mu}_\gamma\left(\overline{M}^G_{\boldsymbol{\nu}}\right),
\quad \gamma:=\begin{pmatrix}
0&\imath\\
\imath&0
\end{pmatrix}=e^{B_{\pi/2}}.
\end{equation}

\begin{prop}
\label{prop:surjective map}

The map 
$$
F:(\overline{m},z)\in \overline{M}^G_{\boldsymbol{\nu}}\times \overline{•B(\mathbf{0},\delta)}\mapsto
\overline{\mu}_{e^{B_z}}(\overline{m})\in \overline{M}^G_{\mathcal{O}}.
$$
satisfies the following properties:
\begin{enumerate}
\item $F$ is surjective;
\item $F$ restricts to a diffeomorphism 
$\overline{M}^G_{\boldsymbol{\nu}}\times •B(\mathbf{0},\pi/2)
\rightarrow \overline{M}^G_{\mathcal{O}}\setminus\overline{M}^G_{\boldsymbol{\nu}}$;
\item $F$ induces a homeomorphism between $\hat{M}^G_{\mathcal{O}}\cong \overline{M}^G_{\mathcal{O}}$. 
\end{enumerate}

\end{prop}

\begin{proof}
[Proof of Proposition \ref{prop:surjective map}]
Let us prove that $F$ is surjective. First note that $\overline{M}^G_{\boldsymbol{\nu}}=F\big(\overline{M}^G_{\boldsymbol{\nu}}\times \{0\}\big)$ and that
$\overline{M}^G_{\boldsymbol{\nu}'}=F\big(\overline{M}^G_{\boldsymbol{\nu}}\times \{\pi/2\}\big)$ by (\ref{eqn:defn di gamma}).
Pick $\overline{m}\in 
\overline{M}^G_{\mathcal{O}}\setminus \big(\overline{M}^G_{\boldsymbol{\nu}}\cup\overline{M}^G_{\boldsymbol{\nu}'}\big)$.
Then there exists $g\in G$ such that
$\overline{m}\in \overline{\mu}_g(\overline{M}^G_{\boldsymbol{\nu}})$, and we need to show that $g$ may be chosen of the form $e^{B_z}$, 
for some $z\in B(\mathbf{0},\pi/2)$.
We know that $g$ is neither diagonal nor antidiagonal.
Furthermore, since $\overline{M}^G_{\boldsymbol{\nu}}$ is $T$-invariant, we are free to replace
$g$ by any element in $g\,T$. 
In particular, we may assume $g\in SU(2)$ and then, muliplying by a suitable diagonal matrix in $SU(2)$, that it has the form 
$$
g=\begin{pmatrix}
\cos(x)& -\sin(x)\,e^{-\imath\,\gamma}\\
\sin(x)\,e^{\imath\,\gamma}&\cos(x)
\end{pmatrix}.
$$
Perhaps multiplying by $-I_2$, we may further assume that $\cos(x)>0$, and since $g$ is not diagonal we may assume $x\in (-\pi/2,0)\cup (0,\pi/2)$. If $x\in (0,\pi/2)$, set $z=\imath\, x\,e^{-\imath\,\gamma}$; 
we conclude from Lemma \ref{lem:exp B eta} that 
$g=e^{B_z}$. If $x\in (-\pi/2,0)$, replace it by $x'=-x\in (0,\pi/2)$ to reach the 
same conclusion. 

Let us prove that $F$ is injective on 
$\overline{M}^G_{\boldsymbol{\nu}}\times •B(\mathbf{0},\pi/2)$. Suppose
$(\overline{m}_j,z_j)\in \overline{M}^G_{\boldsymbol{\nu}}\times •B(\mathbf{0},\pi/2)$ and 
$F(\overline{m}_1,z_1)=F(\overline{m}_2,z_2)$.
We may assume that $|z_j|>0$ for $j=1,2$.
We have, by definition of $F$,
$$
\overline{\mu}_{e^{B_{z_1}}}\left(
\overline{m}_1\right)=
\overline{\mu}_{e^{B_{z_2}}}\left(
\overline{m}_2\right)\quad
\Rightarrow\quad \overline{m}_2=
\overline{\mu}_{e^{-B_{z_2}}\,e^{B_{z_1}}}\left(
\overline{m}_1\right).
$$
Since $\nu_1\neq \nu_2$, this forces
$$
e^{B_{-z_2}}\,e^{B_{z_1}}=e^{-B_{z_2}}\,e^{B_{z_1}}\in T.
$$
Computing the $(1,2)$ entry of the latter product by Lemma \ref{lem:exp B eta},
we obtain
$$
\imath\,\left[\cos(|z_2|)\,\sin(|z_1|)
\frac{z_1}{•|z_1|}-
\sin(|z_2|)\,\cos(|z_1|)\,\frac{z_2}{•|z_2|}\right]=0.
$$
Given that $|z_j|\in (0,\pi/2)$, 
this implies $z_1=z_2$; 
it also follows therefore that
$\overline{m}_1=\overline{m}_2$.

Let us prove that $F$ is an orbifold embedding
on $\overline{M}^G_{\boldsymbol{\nu}}\times •B(\mathbf{0},\pi/2)$. 
We can lift (the restriction of)
$F$ to 
a map
$$
\tilde{F}:
(m,z)\in 
M^G_{\boldsymbol{\nu}}\times •B(\mathbf{0},\pi/2)\mapsto \mu_{e^{B_z}}(m)\in  M^G_{\mathcal{O}}\setminus 
\overline{M}^G_{\boldsymbol{\nu}'}.
$$
Let $S^1$ act on 
$M^G_{\boldsymbol{\nu}}\times •B(\mathbf{0},\pi/2)$ by the product of
the action of $T^1_{\boldsymbol{\nu}_\perp}\cong S^1$ on $M^G_{\boldsymbol{\nu}}$
and the trivial action on $B(\mathbf{0},\pi/2)$.
If $\rho$ is as in Definition \ref{defn:moving action}, 
it follows from Lemma \ref{lem:commuting actions} that $\tilde{F}$ is $S^1$-equivariant,
and $F$ is the map induced by $\tilde{F}$
on the quotient spaces.
To prove the claim, it thus suffices to show that
$\tilde{F}$ is a (local) diffeomorphism.
We know that $\tilde{F}$ is a local diffeorphism along $M^G_{\boldsymbol{\nu}}\times\{0\}$. If $m\in M^G_{\boldsymbol{\nu}}$
and $v\in T_mM^G_{\boldsymbol{\nu}}$,
then for any $z\in B(\mathbf{0},\pi/2)$ we have
\begin{equation}
\label{eqn:dF 1st comp}
\mathrm{d}_{(m,z)}F\big((v,0)\big)
=\mathrm{d}_{m}\mu_{e^{B_z}}(v),
\end{equation}
which is tangent to 
$\mu_{e^{B_z}}(M^G_{\boldsymbol{\nu}})$ 
at $\mu_{e^{B_z}}(m)$.
On the other hand, for $\delta\sim 0\in \mathbb{C}$ we have 
$$e^{B_{z+\delta}}
=e^{B_z}\,e^{B_\delta}=
e^{B_z}\,e^{B_\delta
-\frac{1}{2•}\,[B_z,B_\delta]+R_3(\delta)}.$$
Hence
\begin{equation}
\label{eqn:df 2nd component}
\mathrm{d}_{(m,z)}F
\big((0,\delta)\big)=
\mathrm{d}_{m}\mu_{e^{B_z}}
\left((B_\delta)_M(m)-\frac{1}{•2}\,
[B_z,B_\delta]_M(m)\right).
\end{equation}

Since $[B_z,B_\delta]$ is diagonal
and $T_mM^G_{\boldsymbol{\nu}}$ is
$T$-invariant,
$[B_z,B_\delta]_M(m)\in T_mM^G_{\boldsymbol{\nu}}$. On the other hand,
$(B_\delta)_M(m)\neq 0$ for $\delta\neq 0$,
and is normal to $M^G_{\boldsymbol{\nu}}$.
Hence it follows (\ref{eqn:dF 1st comp})
and (\ref{eqn:df 2nd component})
that 
$$
\mathrm{d}_{(m,0)}F:
T_{(m,z)}\left(M^G_{\boldsymbol{\nu}}\times B(\mathbf{0},\pi/2)\right)\cong 
T_mM^G_{\boldsymbol{\nu}}\times \mathbb{C}\rightarrow
T_{\mu_{e^{B_z}}(m)}M^G_{\mathcal{O}}
$$
is an isomorphism of real vector spaces.

Finally, let us show that the topology of
$\overline{M}^G_{\mathcal{O}}$ is indeed the quotient topology of $F$. Clearly $F$ is continuous, hence $F^{-1}(U)$ is open for every $U\subset \overline{M}^G_{\mathcal{O}}$.
Suppose by contradiction that $F^{-1}(U)$ is open for some $U\subset \overline{M}^G_{\mathcal{O}}$ which is not open.
Let $\overline{m}\in U$ be such that there exists a sequence $\overline{m}_j\in \overline{M}^G_{\mathcal{O}}$, $j=1,2,\ldots$, such that 
$\overline{m}_j\rightarrow \overline{m}$ and 
$\overline{m}_j\not\in U$ for every $j$.
The subset $R:=\{\overline{m}\}_j\cup\{\overline{m}\}\subset \overline{M}^G_{\mathcal{O}}$ is compact, and since $F$ is proper so is $F^{-1}(R)$. Consider $(\overline{n}_j,z_j)\in M^G_{\boldsymbol{\nu}}\times  \overline{B(\mathbf{0},\pi/2)}$ such that 
$F(\overline{n}_j,z_j)=\overline{m}_j$ for
every $j$. Perhaps passing to a subsequence,
we may assume $\overline{n}_j\rightarrow\overline{n}\in M^G_{\boldsymbol{\nu}}$ and 
$z_j\rightarrow z\in \overline{B(\mathbf{0},\pi/2)}$, and therefore
by continuity and uniqueness of the limit $F(\overline{n},z)=\overline{m}\in U$.
Hence $(\overline{n},z)\in F^{-1}(U)$, and
since the latter is open by assumption we
need to have $(\overline{n}_j,z_j)\in F^{-1}(U)$ for all $j\gg 0$. But then
$\overline{m}_j=F(\overline{n}_j,z_j)\in U$,
a contradiction.

\end{proof}

These considerations may be repeated inverting the roles of 
$\boldsymbol{\nu}$ and $\boldsymbol{\nu}'$.
We can define a map $F':(\overline{m},\eta)\in M^G_{\boldsymbol{\nu}}\times \overline{•B(\mathbf{0},\pi/2)} 
\mapsto \overline{\mu}_{e^{B_\eta}}(\overline{m})\in  M^G_{\mathcal{O}}$, and prove an analogue of
Proposition \ref{prop:surjective map}. In particular, we obtain two diffeomorphisms
$$
M^G_{\boldsymbol{\nu}}\times B^*(\mathbf{0},\pi/2)
\stackrel{F}{\longrightarrow}  
M^G_{\mathcal{O}}\setminus \left(M^G_{\boldsymbol{\nu}}\cup M^G_{\boldsymbol{\nu}'}\right)
\stackrel{F'}{\longleftarrow}
M^G_{\boldsymbol{\nu}'}\times B^*(\mathbf{0},\pi/2),
$$
where $B^*(\mathbf{0},\pi/2):=B(\mathbf{0},\pi/2)\setminus \{\mathbf{0}\}$.

\begin{lem}
\label{lem:eta modulo}
Suppose $(\overline{m}_1,z_1)\in M^G_{\boldsymbol{\nu}}\times B^*(\mathbf{0},\pi/2)$,
$(\overline{m}_2,z_2)\in M^G_{\boldsymbol{\nu}'}\times B^*(\mathbf{0},\pi/2)$, and
$F(\overline{m}_1,z_1)=F'(\overline{m}_2,z_2)$.
Then $|z_1|+|z_2|=\pi/2$. 

\end{lem}

\begin{proof}
[Proof of Lemma \ref{lem:eta modulo}]
Let $\overline{m}:=F(\overline{m}_1,z_1)$. Then $\overline{m},\,\overline{m}_1,\,\overline{m}_2$
are all in the same $G$-orbit. Therefore, $\lambda_{\boldsymbol{\nu}}(\overline{m}_1)=
\lambda_{\boldsymbol{\nu}}(\overline{m})=\lambda_{\boldsymbol{\nu}}(\overline{m}_2)$. 
By (\ref{eqn:e b eta D}) and Corollary \ref{cor:moment map eta} and their
analogues with $\boldsymbol{\nu}$ and $\boldsymbol{\nu}'$ interchanged, we have

$$
\overline{\Phi}_{T^1_{\boldsymbol{\nu}_\perp}}(\overline{m)}
=\imath\,\left(\nu_1^2-\nu_2^2\right)\,\lambda_{\boldsymbol{\nu}}(\overline{m})\,
\sin(|z_1|)^2=\imath\,\left(\nu_1^2-\nu_2^2\right)\,\lambda_{\boldsymbol{\nu}}(\overline{m})\,
\cos(|z_2|)^2.
$$
Since
$|z_1|,|z_2|\in (0,\pi/2)$, this forces 
$|z_1|+|z_2|=\pi/2$.

\end{proof}

Let us set
\begin{equation}
\label{eqn:def UU'}
U:=F\left( \overline{M}^G_{\boldsymbol{\nu}}\times B(\mathbf{0},3\,\pi/8)  \right),\quad
U':=F'\left( \overline{M}^G_{\boldsymbol{\nu}'}\times B(\mathbf{0},3\,\pi/8)  \right).
\end{equation}
Then $U,\,U'\subset \overline{M}^G_{\mathcal{O}}$ are open 
and diffeomorphic to $\overline{M}^G_{\boldsymbol{\nu}}\times B(\mathbf{0},3\,\pi/8)$ by Proposition 
\ref{prop:surjective map} and its analogue for $F'$.  
Furthermore, by Lemma 
\ref{lem:eta modulo}, 
\begin{eqnarray}
\label{eqn:U'c}
{U'}^c&:=&F'\left(\left\{(\overline{m},z)
\in \overline{M}^G_{\boldsymbol{\nu}'}\times \overline{•B(\mathbf{0},\pi/2)} :
|z|\ge \frac{3}{8•}\,\pi\right\}\right)
\nonumber\\
&=&F\left(\left\{(\overline{m},z)\in \overline{M}^G_{\boldsymbol{\nu}}\times \overline{•B(\mathbf{0},\pi/2)} :
|z|\le \frac{1}{8•}\,\pi\right\}\right)\subset U.
\end{eqnarray}
Hence $\{U,U'\}$ is an open cover of $M^G_{\mathcal{O}}$.
By (\ref{eqn:def UU'})
and (\ref{eqn:U'c}) we have
\begin{equation}
\label{eqn:UcapU'}
U\cap U'=F\left(\overline{M}^G_{\boldsymbol{\nu}}\times 
A\left(\mathbf{0},\frac{1}{8}\,\pi,\,\frac{3}{•8}\,\pi\right)\right),
\end{equation}
where for $a<b<0$ we set 
$A(\mathbf{0},a,b)=\{z\in \mathbb{C}\,:\,a<|z|<b\}$.
Also, $F$ induces a diffeomorphism
$\overline{M}^G_{\boldsymbol{\nu}}\times 
A\left(\mathbf{0},\pi/8,\,3\,\pi/8\right)$.

Therefore, the Mayer-Vietoris sequence for the open cover $\{U,U'\}$ of
$\overline{M}^G_{\mathcal{O}}$ has the form
\begin{eqnarray}
\label{eqn:mayer vietoris}
\lefteqn{...\rightarrow H^q\left(\overline{M}^G_{\mathcal{O}}\right)\rightarrow
H^q\big(\overline{M}^G_{\boldsymbol{\nu}}\big)\oplus 
H^q\big(\overline{M}^G_{\boldsymbol{\nu}}\big)}\\
&&\rightarrow 
H^q\big(\overline{M}^G_{\boldsymbol{\nu}}\big)\oplus 
H^{q-1}\big(\overline{M}^G_{\boldsymbol{\nu}}\big)\rightarrow
H^{q+1}\big(\overline{M}^G_{\mathcal{O}}\big)\rightarrow \cdots.
\nonumber
\end{eqnarray}
which splits in short exact sequences
$$
0\rightarrow H^{q-1}\big(\overline{M}^G_{\boldsymbol{\nu}}\big)\rightarrow
H^{q+1}\big(\overline{M}^G_{\mathcal{O}}\big)\rightarrow 
H^{q+1}\big(\overline{M}^G_{\boldsymbol{\nu}}\big)\rightarrow 0.
$$

\end{proof}

\end{document}